\documentclass[12pt]{article}
\usepackage{amsmath,mathdots}
\usepackage{amssymb}
\usepackage{latexsym}
\usepackage{amsthm}
\usepackage{amsmath,amssymb,amsfonts,amsthm}
\usepackage{mdframed}
\usepackage{graphicx}
\usepackage{stmaryrd}
\usepackage{gensymb}
\usepackage{bm}
\usepackage{tikz}  
\usepackage[symbol]{footmisc}
\usepackage{stmaryrd}
\usepackage{multirow}
\usepackage{tikz}
\usepackage{longtable}
\usepackage{caption}
\usepackage{subcaption}
\usepackage{adjustbox}
\usetikzlibrary{arrows,patterns,positioning}

\tikzset{
  dot/.style={circle,fill=black,inner sep=2.2pt},
  arr/.style={-{Stealth[length=3mm]},thick},   % 新版
  lab/.style={font=\small}
}
\usepackage{graphicx}
\usepackage{tikz}
\usetikzlibrary{matrix,arrows}
\usetikzlibrary{positioning}
\usetikzlibrary{fit}
\usetikzlibrary{patterns}
\usepackage{tikz}
\usetikzlibrary{calc,patterns,positioning}
\usepackage[colorlinks,
linkcolor=blue,
anchorcolor=blue,
citecolor=blue
]{hyperref}
\usepackage[margin=2.9cm]{geometry}
\DeclareMathOperator{\asc}{asc}
\DeclareMathOperator{\des}{des}
\DeclareMathOperator{\rmax}{rmax}
\DeclareMathOperator{\lmax}{lmax}
\DeclareMathOperator{\rmin}{rmin}
\DeclareMathOperator{\lmin}{lmin}

\usepackage[normalem]{ulem}

\newtheorem{thm}{Theorem}
\newtheorem{lemma}{Lemma}

\usepackage{graphicx}
\usepackage{tikz}
\usetikzlibrary{matrix,arrows}
\usetikzlibrary{positioning}
\usetikzlibrary{fit}
\usetikzlibrary{patterns}

\newcommand{\pattern}[4]{										% mesh pattern
	\raisebox{0.6ex}{
		\begin{tikzpicture}[scale=0.35, baseline=(current bounding box.center), #1]
			\foreach \x/\y in {#4}		\fill[gray!50] (\x,\y) rectangle +(1,1);
			\draw (0.01,0.01) grid (#2+0.99,#2+0.99);
			\foreach \x/\y in {#3}		\filldraw (\x,\y) circle (6pt);
	\end{tikzpicture}}
}

\newcommand{\dbrac}[1]{{\llbracket #1 \rrbracket}} 	% double brackets
\newcommand{\boks}[2]{({#1, #2})}   % A box with the old
\definecolor{red}{rgb}{1,0,0}
\DeclareMathOperator{\red}{red}

\allowdisplaybreaks

\begin{document}
	
\begin{center}
	{\large \bf Distributions of Mesh Patterns of Short Lengths on Separable Permutations}
\end{center}
	
\begin{center}
Alice L.L. Gao$^{1}$, Sergey Kitaev$^{2}$, Ya-Xing Li$^{3}$, and Yun Li$^{4}$\\[6pt]
\end{center}
\begin{small}
$^{1,3,4}$School of Mathematics and Statistics, Northwestern Polytechnical University, Xi'an, Shaanxi 710072, P.R. China

$^{1,3,4}$Shenzhen Research Institute of Northwestern Polytechnical University, Sanhang Science $\&$ Technology Building, No.\ 45th, Gaoxin South 9th Road, Nanshan District,   Shenzhen City,  518057,  P.R. China

$^{2}$Department of Mathematics and Statistics, University of Strathclyde, 26 Richmond Street, Glasgow G1 1XH, United Kingdom 

\noindent
Emails: $^{1}${\tt llgao@nwpu.edu.cn},
$^{2}${\tt sergey.kitaev@strath.ac.uk},
$^{3}${\tt liya@mail.nwpu.edu.cn}, \\
$^{4}${\tt liyun091402@163.com}.
\end{small}
%\end{center}
	
\noindent\textbf{Abstract.}
This paper contributes to the long line of research on the distribution of mesh patterns in permutations. We extend these studies to separable permutations and carry out a comprehensive analysis of mesh patterns of length at most~2. For mesh patterns of length~1, we determine the distributions for all six equivalence classes. In addition, we obtain the joint distribution for the patterns in the class containing the well-known permutation statistic known as the strict fixed point.

Furthermore, computer experiments suggest that at most 124 pairs of mesh patterns of type 12 and 21 with identical shading are equidistributed. Using symmetry operations, we partition these pairs into 38 equivalence classes with respect to distribution. By explicitly determining 24 of these distributions (in many cases, in fact, joint equidistributions of the respective pairs), we reduce the number of equivalence classes to~31, which turns out to be the true number of equivalence classes. Enumerating 14 of these classes is left as an open problem.

	
	
%\noindent 
%\newpage	
%\tableofcontents
%\newpage	
	
\section{Introduction}\label{intro}

A permutation of length $n$ is a rearrangement of the set $[n]:=\{1,2,\ldots,n\}$. Denote by  $S_n$  the set of permutations of $[n]$ and let $\varepsilon$ be the empty permutation. A {\em pattern} is a permutation. A permutation $\pi_1\pi_2\cdots\pi_n\in S_n$ avoids a pattern $p=p_1p_2\cdots p_k\in S_k$  if there is no subsequence $\pi_{i_1}\pi_{i_2}\cdots\pi_{i_k}$ such that $\pi_{i_j}<\pi_{i_m}$ if and only if $p_j<p_m$.  Let $S_n(p)$ denote the set of $p$-avoiding permutations of length $n$. For a permutation $\pi=\pi_1\pi_2\cdots\pi_n$, its {\em reverse} is the permutation $\pi^r=\pi_n\pi_{n-1}\cdots \pi_1$ and its {\em complement} is the permutation $\pi^c=c(\pi_1)c(\pi_2)\cdots c(\pi_n)$ where $c(x)=n+1-x$. Also, the length of a permutation $\pi$, denoted by $|\pi|$, is the number of elements in $\pi$. For example, for $\pi=423156$, $|\pi|=6$, $\pi^r=651324$, and $\pi^c=354621$. Finally, for a sequence $\sigma$ of distinct numbers, the {\em reduced form} of $\sigma$, denoted $\red(\sigma)$, is obtained by replacing the $i$-th smallest number in $\sigma$ by $i$. For example, $\red(4295)=2143$. 

The concept of a {\em mesh pattern} was introduced by Br\"{a}nd\'{e}n and Claesson \cite{BrCl}. The study of mesh patterns  has become a prominent area of research within the theory of patterns in permutations, which has attracted significant attention (see~\cite{Kit}). For enumerative results related to mesh patterns, see, for example,~\cite{AKV,  Hilmarsson2015Wilf, KL, KZ, KZZ, KLv, T2}.

A mesh pattern of length $k$ is a pair $(\tau,R)$, where $\tau \in S_k$ and~$R$ is a subset of $\dbrac{0,k} \times \dbrac{0,k}$, where
$\dbrac{0,k}$ denotes the interval of integers from $0$ to $k$.
Let $\boks{i}{j}$ be the box whose corners have coordinates $(i,j), (i,j+1),
(i+1,j+1)$, and $(i+1,j)$, and let the horizontal lines (resp.,   vertical lines) represent the values (resp.,   positions) in the pattern. Mesh patterns can be drawn by shading the boxes in $R$. A mesh pattern with $\tau=123$ and $R = \{\boks{0}{0},\boks{1}{2},\boks{2}{1},\boks{3}{0}\}$ is drawn as follows
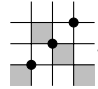
\[
\pattern{scale=0.8}{3}{1/1,2/2,3/3}{0/0,1/2, 2/1,3/0}.
\] 
%A mesh pattern has a {\em symmetric shading}
%if the shading of the box $(i,j)$  implies the shading of  the box $(j,i)$, and 
%{\em antipodal shading} if, for any $i\neq j$, precisely  one of the boxes $(i,j)$ and $(j,i)$ is shaded. See Table~\ref{tab-sy-shadings} (resp., Table~\ref{tab-anti-shadings}) for examples of symmetric (resp., antipodal) shadings of mesh patterns. 

Given a permutation $\pi = \pi_1 \pi_2 \cdots \pi_n \in S_n$, a subsequence $\pi' = \pi_{i_1} \pi_{i_2} \cdots \pi_{i_k}$ is an \emph{occurrence} of a mesh pattern $p = (\tau, R)$ in $\pi$ if $(i)$ $\pi'$ is order-isomorphic to $\tau$, and $(ii)$ the shaded regions defined by $R$ contain no elements of $\pi$ that are not in $\pi'$.
Otherwise, $\pi$ is said to avoid $p$, denoted by $\pi \in S_n(p)$. 
For example, consider the mesh pattern $p = \big(123, \{(0,0),(1,2),(2,1),(3,1)\}\big)$.
The permutation $23154$ contains two occurrences of $p$, whereas the permutation $14325$ avoids $p$. We let $p(\pi)$ denote the number of occurrences of (mesh) pattern $p$ in $\pi$.

Two patterns $q_1$ and $q_2$ are {\em Wilf-equivalent}, denoted $q_1\sim q_2$, if  $|S_n(q_1)|=|S_n(q_2)|$. Moreover, $q_1$ and $q_2$ are called {\em equidistributed}, denoted~$q_1\sim_d q_2$, 
if the number of permutations of length $n$ with $k$ occurrences of $q_1$ is equal to that with $k$ occurrences of $q_2$ for any nonnegative integers $n$ and $k$. 
The stricter notion of \textit{joint equidistribution} of two patterns, denoted $q_1\sim_{jd}q_2$, means that the number of $n$-permutations with $k$ occurrences of $q_1$ and $\ell$ occurrences of $q_2$ is equal to that with $\ell$  occurrences of $q_1$ and $k$ occurrences of $q_2$. 

For $n \ge 2$, a permutation $\pi_1 \pi_2 \cdots \pi_n$ is \textit{irreducible} if there is no $i$, 
$2 \le i \le n$, such that any element in $\pi_1 \cdots \pi_{i-1}$ is less than each element in 
$\pi_i \cdots \pi_n$. For example, the permutation $42513$ is irreducible. By definition, the permutation $1$ is irreducible. A permutation is \textit{reducible} if it is not irreducible. The empty permutation $\varepsilon$ is neither reducible nor irreducible. Each reducible permutation can be cut into irreducible pieces called 
\textit{irreducible components}, or just \textit{components}. For example, the components of the permutation $312546978$ are $312$, $54$, $6$, and $978$.

Suppose $\pi=\pi_1\pi_2\cdots\pi_m \in S_m$ and $\sigma=\sigma_1\sigma_2\cdots\sigma_n\in S_n$. We define the {\em direct sum}\index{direct sum $\oplus$} (or simply, {\em sum}) $\oplus$, and the {\em skew sum}\index{skew sum $\ominus$} $\ominus$ by building the permutations $\pi\oplus\sigma$ and $\pi\ominus\sigma$ as follows:
\begin{eqnarray*}
(\pi\oplus\sigma)_i&=& \left\{
\begin{array}{ll}
\pi_i  &  \mbox{ if }1\leq i\leq m,\\
\sigma_{i-m}+m & \mbox{ if }m+1\leq i\leq m+n,
\end{array}\right.\nonumber \\
(\pi\ominus\sigma)_i &=& \left\{
\begin{array}{ll}
\pi_i+n  &  \mbox{ if }1\leq i\leq m,\\
\sigma_{i-m} & \mbox{ if }m+1\leq i\leq m+n.
\end{array}\right.\nonumber
\end{eqnarray*}
For example, $14325\oplus 4231=143259786$ and $14325\ominus 4231= 587694231$. 

The {\em separable permutations} are those which can be built from the permutation $1$ by repeatedly applying the $\oplus$ and $\ominus$ operations. Bose, Buss and Lubiw~\cite{BBL98} introduced the notion of separable permutation in 1998, and it is well-known \cite{Kit} that the set of all separable permutations of length $n\geq 1$ is precisely $S_n(2413,3142)$ (we assume here that by definition, the empty permutation $\varepsilon$ is not separable).  Separable permutations appear in the literature in various contexts (see \cite{CKZ2024} for references). 

Clearly, the reverse, or complement, or (usual group theoretic) inverse, or any composition of these operations applied to a separable permutation gives a separable permutation. It is not difficult to see that any $\pi\in S_n(2413,3142)$ has the following structure (see also Figure~\ref{sepStructure} for a schematic representation, where a permutation is viewed as a diagram with a dot in position $(i,\pi_i)$ for each element $\pi_i$ of $\pi$):
\begin{align}
	\pi=L_1L_2\cdots
	L_mnR_mR_{m-1}\cdots R_1\notag
\end{align}
where
\begin{itemize}
	\item[(1)] for $1\leq i\leq m$, $L_i$ and $R_i$ are non-empty ($\neq \varepsilon$), with possible exception for
	$L_m$ and $R_1$, separable permutations which are intervals in $\pi$ (that is, consist of all elements in $\{a,a+1,\ldots,b\}$ for some $a$ and $b$);
	\item[(2)]  $R_1<L_1<R_2<L_2<\cdots <R_m<L_m$, where $A<B$, for two
	permutations $A$ and $B$, means that each element of $A$ is less than
	every element of $B$. In particular, $R_1$, if it is not empty,
	contains 1.
\end{itemize}

\noindent
For example, if $\pi=2165743$ then $L_1=21$, $L_2=65$, $R_1=\varepsilon$ and $R_2=43$. Note that a separable permutation is irreducible if and only if $R_1\neq \epsilon$.

\begin{figure}[htbp]
	\begin{center}
		\begin{tikzpicture}[line width=0.5pt,scale=0.24]
			\coordinate (O) at (0,0);
			
			\path (25,1)  node {$n$};
			\draw [dashed] (O)--++(50,0);
			\fill[black!100] (O)++(25,0) circle(1.5ex);
			
			\draw (20,-1) rectangle (24,-3);
			\path (22,-2)  node {$L_m$};
			\path (13.5,-2)  node {possibly empty$\boldsymbol{\rightarrow}$};
			
			\draw [dashed] (0,-4)--++(50,0);
			\draw [dashed] (0,-8)--++(50,0);
			\draw [dashed] (0,-17)--++(50,0);
			\draw [dashed] (0,-21)--++(50,0);
			\draw [dashed] (0,-25)--++(50,0);
			
			\draw [dashed] (25,0)--++(0,-29);
			\draw [dashed] (31,0)--++(0,-29);
			\draw [dashed] (37,0)--++(0,-29);
			\draw [dashed] (43,0)--++(0,-29);
			
			\draw [dashed] (19,0)--++(0,-29);
			\draw [dashed] (13,0)--++(0,-29);
			\draw [dashed] (7,0)--++(0,-29);
			
			\draw (26,-5) rectangle (30,-7);
			\path (28,-6)  node {$R_m$};
			
			\fill[black!100] (O)++(17,-9) circle(1.0ex);
			\fill[black!100] (O)++(16,-10) circle(1.0ex);
			\fill[black!100] (O)++(15,-11) circle(1.0ex);
			\draw (8,-14) rectangle (12,-16);
			\path (10,-15)  node {$L_2$};
			
			\draw (2,-22) rectangle (6,-24);
			\path (4,-23)  node {$L_1$};

			\fill[black!100] (O)++(32,-11) circle(1.0ex);
			\fill[black!100] (O)++(33,-12) circle(1.0ex);
			\fill[black!100] (O)++(34,-13) circle(1.0ex);
			\draw (38,-18) rectangle (42,-20);
			\path (40,-19)  node {$R_2$};
			
			\draw (45,-26) rectangle (49,-28);
			\path (47,-27)  node {$R_1$};
			\path (38.5,-27)  node {possibly empty$\boldsymbol{\rightarrow}$};
			
		\end{tikzpicture}
		\caption{A schematic representation of the permutation diagrams corresponding to separable permutations. Each $L_i$ and $R_j$ is a separable permutation.}\label{sepStructure}
	\end{center}
\end{figure}
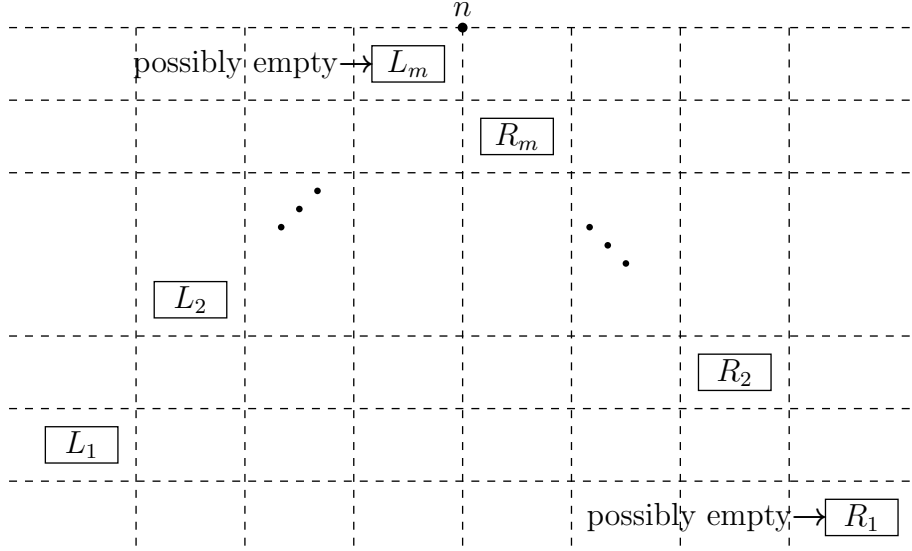
Throughout this paper, we use the abbreviation ``g.f.'' to denote a \textit{generating function}. We denote by $F(x)$ the generating function (g.f.) for separable permutations:
\begin{equation}\label{eq:separable-gf}
	F(x) := \sum_{n \ge 1} |S_n(3142, 2413)| x^n 
	= \frac{1 - x - \sqrt{1 - 6x + x^2}}{2}.
\end{equation}
The notation \(F(x,t)\), and likewise \(F(x,t_1,t_2)\), is local to
each theorem and its proof: the marked pattern or patterns are specified
there. Throughout the paper, \(F(x)\) always denotes the series in
\eqref{eq:separable-gf}.

\subsection{Symmetry operations}
Clearly, applying the reverse, complement, or (usual group-theoretic) inverse operation, or any composition of these operations, to a separable permutation yields a separable permutation. For a given mesh pattern $(\tau, R)$ of length~$n$, we define
$$
(\tau, R)^{r} = (\tau^{r}, R^{r}), \quad
(\tau, R)^{c} = (\tau^{c}, R^{c}), \quad
(\tau, R)^{i} = (\tau^{i}, R^{i}),
$$
where $\tau^{i}$ is the usual inverse, and $R^{r} = \{(n - x, y) : (x, y) \in R\}$, $R^{c} = \{(x, n - y) : (x, y) \in R\}$, and $R^{i} = \{(y, x) : (x, y) \in R\}$.

Hence, \emph{reverse} is a reflection around the vertical center line, 
\emph{complement} is a reflection around the horizontal center line, 
and \emph{inverse} is the reflection around the southwest-to-northeast diagonal. 
Figure~\ref{fig:symmetry-example} illustrates applications of these operations 
on the pattern $p = (312, \{(0,1), (1,3), (2,2)\})$.

\begin{figure}[htbp]
	\centering
	\begin{tikzpicture}[
		arr/.style={->,thick},
		lab/.style={font=\small},
		symax/.style={dash pattern=on 7pt off 3pt, line width=0.9pt, draw=black!70}
		]
		
		\def\GAP{3.6}
		
		\node (P) at (0,0)
		{\pattern{execute at end picture={\draw[symax] (2,-0.7) -- (2,4.7);}}
			{3}{1/3,2/1,3/2}{0/1,1/3,2/2}};
		\node at ($(P.center)+(0.05,-1.15)$) {$p$};
		
		\draw[arr] ($(P.east)+(0.3,0)$) -- ++(1.0,0);
		\node[lab] at ($(P.east)+(0.8,-0.35)$) {$r$};
		
		\node (PR) at ($(P.center)+(\GAP,0)$)
		{\pattern{execute at end picture={\draw[symax] (-0.5,-0.5) -- (4.5,4.5);}}
			{3}{1/2,2/1,3/3}{1/2,2/3,3/1}};
		\node at ($(PR.center)+(0.05,-1.15)$) {$p^{r}$};
		
		\draw[arr] ($(PR.east)+(0.3,0)$) -- ++(1.0,0);
		\node[lab] at ($(PR.east)+(0.8,-0.35)$) {$i$};
		
		\node (PRI) at ($(PR.center)+(\GAP,0)$)
		{\pattern{execute at end picture={\draw[symax] (-0.5,2) -- (4.5,2);}}
			{3}{1/2,2/1,3/3}{1/3,2/1,3/2}};
		\node at ($(PRI.center)+(0.05,-1.15)$) {$p^{ri}$};
		
		\draw[arr] ($(PRI.east)+(0.3,0)$) -- ++(1.0,0);
		\node[lab] at ($(PRI.east)+(0.8,-0.35)$) {$c$};
		
		\node (PRIC) at ($(PRI.center)+(\GAP,0)$)
		{\pattern{}{3}{1/2,2/3,3/1}{1/0,2/2,3/1}};
		\node at ($(PRIC.center)+(0.05,-1.15)$) {$p^{ric}$};
	\end{tikzpicture}
	\caption{Several symmetries of a mesh pattern.}
	\label{fig:symmetry-example}
\end{figure}
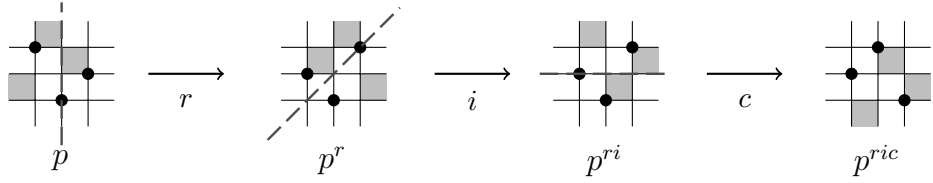
\subsection{Distributions of statistics on separable permutations}
For a permutation $\pi = \pi_1\pi_2\cdots \pi_n$, the \emph{descent} (resp.,\ \emph{ascent}) statistic on $\pi$, $\des(\pi)$ (resp.,\ $\asc(\pi)$), is defined as the number of $i \in \{1,2,\ldots,n-1\}$ such that $\pi_i > \pi_{i+1}$ (resp.,\ $\pi_i < \pi_{i+1}$). For example, $\des(561423)=2$ and $\asc(25134)=3$.

The following lemma will be used in Theorem~\ref{equ:Nr.13} below. It gives the distribution of descents (equivalently, by using reverse, ascents) on separable permutations.

\begin{lemma}[{\cite[Cor.~2.4]{fu2018two}}]\label{Sxt}
  The g.f.
  $$Q(x,t):=\sum_{n\ge1} x^n \sum_{\pi\in S_n(2413,3142)}
  t^{\asc(\pi)}=\sum_{n\ge1} x^n \sum_{\pi\in S_n(2413,3142)}
  t^{\des(\pi)}$$ satisfies the functional equation
  \begin{equation}\label{eq:Stq}
  	tQ^3(x,t) + txQ^2(x,t) + ((1+t)x - 1)Q(x,t) + x = 0. 
  \end{equation}
\end{lemma}
For a permutation $\pi = \pi_1 \pi_2 \cdots \pi_n$, 
$\pi_i$ is a \emph{right-to-left maximum} (resp., \emph{right-to-left minimum}) 
if $\pi_i > \pi_j$ (resp., $\pi_i < \pi_j$) for all $j > i$. In particular, $\pi_n$ is a right-to-left maximum and a right-to-left minimum. 
The number of right-to-left maxima (resp., right-to-left minima) is denoted by $\rmax(\pi)$ (resp., $\rmin(\pi)$). 
For example, $\rmax(436512) = 3$ and $\rmin(326154) = 2$. 
Also, $\pi_i$ is a \emph{left-to-right maximum} (resp., \emph{left-to-right minimum}), denoted by $\lmax(\pi)$ (resp., $\lmin(\pi)$) if $\pi_i > \pi_j$ (resp., $\pi_i < \pi_j$) for all $j < i$. In particular, $\pi_1$ is a left-to-right maximum and a left-to-right minimum. For example, $\lmax(325164) = \lmin(326154) = 3$.

In this paper, we will frequently make use of the following theorem from~\cite{CKZ2024}, which describes the distributions of four permutation statistics on separable permutations, as well as on their reducible and irreducible subclasses.

\begin{thm}[{\cite[Thm.~5]{CKZ2024}}]\label{four-stat}
	The g.f.\ $G(x,t) := \sum_{\pi} x^{|\pi|} t^{\operatorname{stat}(\pi)}$ 
	giving the distribution of any statistic 
	$\operatorname{stat} \in \{\lmax, \rmax, \lmin, \rmin\}$ 
	on separable permutations is
	\begin{small}
	\begin{equation}\label{eq:four-stat}
\frac{4\sqrt{\left(\frac{x}{4}+\frac{5}{4}-\frac{1}{4} \sqrt{x^2-6 x+1}-tx\right)^2+\sqrt{x^2-6 x+1}-x-1}-\sqrt{x^2-6 x+1}+4t x+x - 3}{2 \left(\sqrt{x^2-6 x+1}-x-1\right)}.
	\end{equation}
	\end{small}
	The distributions of $\operatorname{stat}$ on irreducible and reducible separable permutations 
	are given in Table~\ref{tab-dist-rmax-thm}.
\end{thm}

\begin{table}[htbp]
	\centering
	\begin{tabular}{c|c|c}
		g.f. & irreducible & reducible \\
		\hline
		& & \\[-2mm]
		$\frac{G^2(x,t)}{1+G(x,t)} + xt$
		& stat $\in \{\rmax, \lmin\}$
		& stat $\in \{\lmax, \rmin\}$ \\[2mm]
		\hline
		& & \\[-2mm]
		$\frac{G(x,t)}{1+G(x,t)} - xt$
		& stat $\in \{\lmax, \rmin\}$
		& stat $\in \{\rmax, \lmin\}$
	\end{tabular} 
	\caption{Distributions of statistics in Theorem~\ref{four-stat}.}
	\label{tab-dist-rmax-thm}
\end{table}

\subsection{Distribution of mesh patterns in the literature}

Kitaev and Zhang~\cite{KZ} initiated a systematic study of the distributions of short mesh patterns (of length at most~2) in permutations, extending the earlier work of~\cite{Hilmarsson2015Wilf} on the systematic study of mesh pattern avoidance in permutations. These investigations were later extended to certain infinite families of mesh patterns in~\cite{KZZ}. The systematic study of mesh patterns of length~3 was initiated in~\cite{KLv} and subsequently continued in~\cite{LvZhang,LvZhang1}. Other contributions to the literature on distributions of mesh patterns include~\cite{FFKL,HanZeng,HanZeng1}. The distribution of mesh patterns on a particular subclass of permutations, namely \emph{king permutations}, was studied in~\cite{LiZhang}.

\subsection{Results in this paper}

In this paper, we first determine the distributions of all mesh patterns of length~1 (see Table~\ref{tab:length-1} and the references therein). In one instance, Theorem~\ref{joint-separ-length-1}, we derive the joint equidistribution of strict fixed points and reverse strict fixed points. We then classify the mesh patterns $12$ and $21$ with identical shadings into 31 equivalence classes based on their distributions; these classes are listed in Tables~\ref{tab:sym-1group}--\ref{tab:sym-4group}, with further mergers described in Table~\ref{tab:fur-mergy}. These results are obtained using symmetry operations as well as by explicitly determining distributions of mesh patterns (we derived 24 distributions). In many cases, we are also able to determine the joint distributions of the corresponding pairs. Interestingly,  although the distributions of the patterns
\(\hspace{-1mm}\pattern{scale = 0.6}{2}{1/1,2/2}{0/0,0/2,1/0,1/1,1/2,2/0,2/2}\)
and
\(\hspace{-1mm}\pattern{scale = 0.6}{2}{1/1,2/2}{0/0,0/1,1/0,1/1,1/2,2/0,2/2}\) are the same,
the joint distributions of the pairs \(\{\hspace{-1mm}\pattern{scale = 0.6}{2}{1/1,2/2}{0/0,0/2,1/0,1/1,1/2,2/0,2/2},\hspace{-1mm}\pattern{scale = 0.6}{2}{1/2,2/1}{0/0,0/2,1/0,1/1,1/2,2/0,2/2}\}\) and \(\{\hspace{-1mm}\pattern{scale = 0.6}{2}{1/1,2/2}{0/0,0/1,1/0,1/1,1/2,2/0,2/2},\hspace{-1mm}\pattern{scale = 0.6}{2}{1/2,2/1}{0/0,0/1,1/0,1/1,1/2,2/0,2/2}\}\) are different.

\section{Distribution of mesh pattern of length 1 on separable permutations}
%In this section, we investigate the distribution of mesh patterns of length $1$ on separable permutations.
There are $2^{4} = 16$ mesh patterns of length~1. Under inverse, reverse, and complement, these patterns fall into six equidistribution classes, listed in Table~\ref{tab:length-1}. For each class, we determine the corresponding distribution. In addition, we compute the joint distribution of the patterns \hspace{-1mm}$\pattern{scale = 0.6}{1}{1/1}{0/1,1/0}$ and \hspace{-1mm}$\pattern{scale = 0.6}{1}{1/1}{0/0,1/1}$. In Table~\ref{tab:length-1}, mesh patterns that are equidistributed by the reverse, complement, or inverse operations, are grouped together within the same frame. 

\begin{table}[ht]
	{
		\renewcommand{\arraystretch}{1.5}
		\setlength{\tabcolsep}{6pt}
		\begin{center}
			\begin{tabular}{|c c |c|| c c|c|| c c|c|}
				\hline
				\footnotesize{Nr.} & \footnotesize{Pattern} &
				\footnotesize{Ref.} &
				\footnotesize{Nr.} & \footnotesize{Pattern} &
				\footnotesize{Ref.} &
				\footnotesize{Nr.} & \footnotesize{Pattern} &
				\footnotesize{Ref.} \\
				\hline
				\hline
				
				\footnotesize{N1} &
				$\pattern{scale = 0.8}{1}{1/1}{1/1}$ &
				\multirow{5}{*}{\footnotesize{Thm.~\ref{N1}}}  &
				\footnotesize{N5} &
				$\pattern{scale = 0.8}{1}{1/1}{0/1,0/0}$ &
				\multirow{5}{*}{\footnotesize{Thm.~\ref{N5}}}  &
				\footnotesize{N9} &
				$\pattern{scale = 0.8}{1}{1/1}{1/0,1/1,0/0}$&
				\multirow{5}{*}{\footnotesize{Thm.~\ref{N9}}}  \\[5pt]
				
				\footnotesize{N2} &
				$\pattern{scale = 0.8}{1}{1/1}{0/1}$  &
				\footnotesize{\multirow{1.2}{1.3cm}{}} &
				\footnotesize{N6} &
				$\pattern{scale = 0.8}{1}{1/1}{1/0,1/1}$ &
				\footnotesize{\multirow{1.2}{1.3cm}{}} &
				\footnotesize{N10} &
				$\pattern{scale = 0.8}{1}{1/1}{1/0,0/1,0/0}$&
				\footnotesize{\multirow{1.2}{1.3cm}{}} \\[5pt]

				\footnotesize{N3} &
				$\pattern{scale = 0.8}{1}{1/1}{1/0}$  &
				\footnotesize{\multirow{1.2}{1.3cm}{}} &
				\footnotesize{N7} &
				$\pattern{scale = 0.8}{1}{1/1}{1/1,0/1}$ &
				\footnotesize{\multirow{1.2}{1.3cm}{}} &
				\footnotesize{N11} &
				$\pattern{scale = 0.8}{1}{1/1}{0/0,1/1,0/1}$ &
				\footnotesize{\multirow{1.2}{1.3cm}{}} \\[5pt]

				\footnotesize{N4} &
				$\pattern{scale = 0.8}{1}{1/1}{0/0}$ &
				\footnotesize{\multirow{1.2}{1.3cm}{}} &
				\footnotesize{N8} &
				$\pattern{scale = 0.8}{1}{1/1}{0/0,1/0}$ &
				\footnotesize{\multirow{1.2}{1.3cm}{}} &
				\footnotesize{N12} &
				$\pattern{scale = 0.8}{1}{1/1}{1/0,1/1,0/1}$ &
				\footnotesize{\multirow{1.2}{1.3cm}{}} \\[5pt]
				\hline
				\footnotesize{N13} &
				$\pattern{scale = 0.8}{1}{1/1}{1/0,0/1}$ &
				\multirow{2.5}{*}{\footnotesize{Thm.~\ref{joint-separ-length-1}}} &
				\footnotesize{N15} &
				$\pattern{scale = 0.8}{1}{1/1}{}$ &
				\footnotesize{\multirow{1.2}{1.3cm}{Thm.~\ref{N15}}} &
				\footnotesize{N16} &
			    $\pattern{scale = 0.8}{1}{1/1}{1/0,1/1,0/0,0/1}$&
				\footnotesize{\multirow{1.2}{1.3cm}{Thm.~\ref{N16}}} \\[5pt]
				\cline{4-9}
				\footnotesize{N14} &
				$\pattern{scale = 0.8}{1}{1/1}{0/0,1/1}$ &
				\footnotesize{\multirow{1.2}{1.3cm}{}} \\[5pt]
				\cline{1-3}
			\end{tabular}
		\end{center}
	}
	\vspace{-0.5cm}
	\caption{Symmetry-induced equidistribution classes for mesh patterns of length $1$.}
	\label{tab:length-1}
\end{table}

%Here, we demonstrate the operation for N1 to N4 as an example, while the remaining cases follow in an analogous manner.
%\begin{center}
%	\begin{tikzpicture}[
%		arr/.style={->, line width=0.6pt},
%		lab/.style={font=\small},
%		node distance=2.0cm
%		]
%		
%		\def\myscale{1.10}
%		
%		\node (P) {\pattern{scale=\myscale}{1}{1/1}{1/1}};
%		\node at ($(P.south)+(0,-0.34)$) {$p$};
%		
%		\node[right=of P] (PR) {\pattern{scale=\myscale}{1}{1/1}{0/1}};
%		\node at ($(PR.south)+(0,-0.34)$) {$p^{r}$};
%		
%		\draw[arr, shorten >=5pt, shorten <=5pt] (P.east) -- (PR.west)
%		node[midway, below=2pt] {$r$};
%		
%		\node[right=of PR] (PRI) {\pattern{scale=\myscale}{1}{1/1}{1/0}};
%		\node at ($(PRI.south)+(0,-0.34)$) {$p^{ri}$};
%		
%		\draw[arr, shorten >=5pt, shorten <=5pt] (PR.east) -- (PRI.west)
%		node[midway, below=2pt] {$i$};
%		
%		\node[right=of PRI] (PRIC) {\pattern{scale=\myscale}{1}{1/1}{0/0}};
%		\node at ($(PRIC.south)+(0,-0.34)$) {$p^{rir}$};
%		
%		\draw[arr, shorten >=5pt, shorten <=5pt] (PRI.east) -- (PRIC.west)
%		node[midway, below=2pt] {$r$};
%		
%	\end{tikzpicture}
%\end{center}

In a separable permutation, occurrences of the patterns   \hspace{-1mm}$\pattern{scale = 0.6}{1}{1/1}{0/1}$,  \hspace{-1mm}$\pattern{scale = 0.6}{1}{1/1}{0/0}$,  \hspace{-1mm}$\pattern{scale = 0.6}{1}{1/1}{1/1}$, and  \hspace{-1mm}$\pattern{scale = 0.6}{1}{1/1}{1/0}$ correspond, respectively, to left-to-right maxima, left-to-right minima, right-to-left maxima, and right-to-left minima. Consequently, we can use Theorem~\ref{four-stat} to obtain the following result, which establishes the distributions of the mesh patterns~N1--N4.

\begin{thm}\label{N1}
	For $p\in\{\hspace{-2mm}\pattern{scale = 0.6}{1}{1/1}{1/1},\pattern{scale = 0.6}{1}{1/1}{0/1}, \pattern{scale = 0.6}{1}{1/1}{1/0}, \pattern{scale = 0.6}{1}{1/1}{0/0}\}$, and  $G(x,t)$ given by~\eqref{eq:four-stat}, we have
\begin{equation*}
\sum_{n \ge 1} x^n \sum_{\pi \in S_n(2413,3142)}t^{p(\pi)} = G(x,t).
\end{equation*}
\end{thm}

The proof of the following theorem is immediate from the fact that each element in a permutation contributes an occurrence of the pattern N15. 
\begin{thm}\label{N15}
	We have
	$$
	\sum_{n \ge 1} x^n \sum_{\pi \in S_n(2413,3142)} t^{\,\pattern{scale = 0.5}{1}{1/1}{}(\pi)} = F(xt) = \frac{1 - t x - \sqrt{\,1 - 6 t x + t^{2} x^{2}}}{2}.$$
\end{thm}

\begin{thm}[Pattern N16]\label{N16}
We have
	$$
	F(x,t) := \sum_{n \ge 1} x^n \sum_{\pi \in S_n(2413,3142)} 
	t^{\,\pattern{scale = 0.5}{1}{1/1}{1/0,1/1,0/0,0/1}(\pi)}= \frac{1 + (2t-3 )\,x - \sqrt{\,1 - 6x + x^{2}}}{2}.
	$$
\end{thm}

\begin{proof}
	Clearly, the only separable permutation that contains the pattern  \hspace{-1mm}$\pattern{scale = 0.6}{1}{1/1}{1/0,1/1,0/0,0/1}$ is $1$, so  
	\begin{equation}\label{N16-expr-A}
		F(x,t) = tx + F(x) - x .
	\end{equation}
	By substituting~\eqref{eq:separable-gf}~into~\eqref{N16-expr-A}, we obtain the desired result.
\end{proof}

\begin{thm}[Patterns N5--N8]\label{N5}
	For $p\in\{\hspace{-1mm}\pattern{scale = 0.6}{1}{1/1}{0/1,0/0},\pattern{scale = 0.6}{1}{1/1}{1/0,1/1},\pattern{scale = 0.6}{1}{1/1}{0/1,1/1},\pattern{scale = 0.6}{1}{1/1}{1/0,0/0}\}$, we have
	$$
\begin{aligned}
	F(x,t):
	&= \sum_{n \ge 1} x^n \sum_{\pi \in S_n(2413,3142)}
	t^{p(\pi)}= \frac{t(1 - x - \sqrt{1 - 6x + x^{2}})}{2}.
\end{aligned}
    $$
\end{thm}

\begin{proof}
W.l.o.g., consider N5=\hspace{-1mm}$\pattern{scale = 0.6}{1}{1/1}{0/1,0/0}$. It occurs exactly once, as the leftmost element, in any separable permutation, so $F(x,t) = tF(x)$, and the result follows by~\eqref{eq:separable-gf}.
\end{proof}

\begin{thm}[Patterns N9--N12]\label{N9}
	For $p\in\{\hspace{-1mm}\pattern{scale = 0.6}{1}{1/1}{1/0,1/1,0/0},\pattern{scale = 0.6}{1}{1/1}{1/0,0/1,0/0},\pattern{scale = 0.6}{1}{1/1}{0/0,1/1,0/1},\pattern{scale = 0.6}{1}{1/1}{1/0,1/1,0/1}\}$, we have
	$$
	\begin{aligned}
		F(x,t):
		&= \sum_{n \ge 1} x^n \sum_{\pi \in S_n(2413,3142)}
		t^{p(\pi)}  \\[4pt]
		&= \frac{1}{2}\!\left(
		1 - x - \sqrt{1 - 6x + x^{2}}
		-(t-1)\,x\left(x-3 + \sqrt{1 - 6x + x^{2}}\right)
		\right).
	\end{aligned}
	$$
\end{thm}

\begin{proof}
W.l.o.g., consider the pattern N9 $=\hspace{-1mm}\pattern{scale = 0.6}{1}{1/1}{1/0,1/1,0/0}$. 
Only the element $1$ occurring at the end of a permutation can be an occurrence of the pattern 
\hspace{-1mm}$\pattern{scale = 0.6}{1}{1/1}{1/0,1/1,0/0}$, and in this position, $1$ cannot be part of an occurrence 
of the patterns $2413$ or $3142$. Hence, if $A(x)$ denotes the g.f.\ for separable permutations that end with 
their smallest element, we have
\begin{equation}\label{N9-expr-A}
    F(x,t) = tA(x) + F(x) - A(x).
\end{equation}
Moreover, any permutation counted by $A(x)$ can be generated by taking any separable permutation or the empty 
permutation and adjoining the new smallest element to its right, so $A(x) = x\,(F(x)+1)$. Together with 
 \eqref{eq:separable-gf} and \eqref{N9-expr-A}, this gives the desired result.
\end{proof}

Occurrences of the pattern $\hspace{-1mm}\pattern{scale=0.5}{1}{1/1}{0/1,1/0}$ are known as \emph{strict fixed points}. 
The avoidance of this pattern on permutations was studied in \cite{Baril}, and its distribution in~\cite{KZ}. 
We extend the study of the distribution of strict fixed points to separable permutations. 
Since separable permutations are invariant under the reverse operation, the distribution of the pattern 
N13 $= \hspace{-1mm}\pattern{scale = 0.6}{1}{1/1}{0/1,1/0}$ coincides with that of the pattern 
N14 $= \hspace{-1mm}\pattern{scale = 0.6}{1}{1/1}{0/0,1/1}$. 
Moreover, we obtain the joint distribution of the patterns 
 \hspace{-1mm}$\pattern{scale = 0.6}{1}{1/1}{0/1,1/0}$ and 
 \hspace{-1mm}$\pattern{scale = 0.6}{1}{1/1}{0/0,1/1}$. 
The following theorem establishes the distributions of the mesh patterns N13 and N14, as well as their joint distribution.

\begin{thm}[Patterns N13 and N14]\label{joint-separ-length-1}
		We have $\hspace{-1mm}\pattern{scale = 0.6}{1}{1/1}{0/1,1/0}\sim_{jd}
 \hspace{-1mm}\pattern{scale = 0.6}{1}{1/1}{0/0,1/1}$. Moreover, if
	$$
	F(x,t) := \sum_{n \ge 1} x^n \sum_{\pi \in S_n(2413,3142)} 
	t^{\hspace{-1mm}\pattern{scale = 0.5}{1}{1/1}{0/1,1/0}(\pi)}
	= \sum_{n \ge 1} x^n \sum_{\pi \in S_n(2413,3142)} 
	t^{\hspace{-1mm}\pattern{scale = 0.5}{1}{1/1}{0/0,1/1}(\pi)}\mbox{ \ \ and }
	$$
	$$
	F(x,t_1,t_2) := \sum_{n \ge 1} x^n \sum_{\pi \in S_n(2413,3142)} 
	t_1^{\hspace{-1mm}\pattern{scale = 0.5}{1}{1/1}{0/1,1/0}(\pi)}
	t_2^{\hspace{-1mm}\pattern{scale = 0.5}{1}{1/1}{0/0,1/1}(\pi)},
	$$
then
	\begin{equation}\label{eq:separ-length-1}
		F(x , t) = \frac{1 + (4t-3)x - \sqrt{1 + (x-6)x}}
		{3 + (3 - 4t)x + \sqrt{1 + (x-6)x}};
		\end{equation}
	\begin{align}
	\label{eq:joint-separ-length-1}
		F(x,t_1,t_2)  & = 
        \frac{1}{4} \Bigg(4t_1t_2x + \frac{\big(-1 + (3 - 4t_1)x + \sqrt{1 - 6x + x^2}\big)^2}
       {3 + (3 - 4t_1)x + \sqrt{1 - 6x + x^2}} + \nonumber \\ 
       &\frac{\big(-1 + (3 - 4t_2)x + \sqrt{1 - 6x + x^2}\big)^2} {3 + (3 - 4t_2)x + \sqrt{1 - 6x + x^2}}\Bigg).
	\end{align}
	\end{thm}

\begin{figure}[t]
	\centering
	
	% ===== (a) =====
	\begin{subfigure}{0.47\textwidth}
		\centering
		\begin{tikzpicture}[line width=0.5pt,scale=0.24,baseline=(bb.center)]
			\path[use as bounding box] (0,-10) rectangle (24,1);
			\coordinate (bb) at (12,-5); 
			\coordinate (O) at (0,0);
			\path (15,1)  node {$n$};
			\draw [dashed] (0,0)--++(24,0);
			\draw [dashed] (0,-5)--++(24,0);
			\fill[black!100] (O)++(15,0) circle(1.5ex);
			\draw (10,-1) rectangle (20,-4);
			\path (15,-2.5)  node {$A$};
			\draw (4,-6) rectangle (8,-9);
			\path (6,-7.5)  node {$B$};
			\draw [dashed] (15,0)--++(0,-10);
			\draw [dashed] (9,0)--++(0,-10);
			\draw [dashed] (3,0)--++(0,-10);
			\draw [dashed] (21,0)--++(0,-10);
		\end{tikzpicture}
		\subcaption{}\label{fig:joint-separ-length-1-a}
	\end{subfigure}
	\hfill
	% ===== (b) =====
	\begin{subfigure}{0.47\textwidth}
		\centering
		\begin{tikzpicture}[line width=0.5pt,scale=0.24,baseline=(bb.center)]
		\path[use as bounding box] (6,-10) rectangle (30,1);
		\coordinate (bb) at (18,-5); 
			\coordinate (O) at (0,0);
			\path (15,1)  node {$n$};
			\draw [dashed] (6,0)--++(24,0);
			\draw [dashed] (6,-5)--++(24,0);
			\fill[black!100] (O)++(15,0) circle(1.5ex);
			\draw (10,-1) rectangle (20,-4);
			\path (15,-2.5)  node {$A$};
			\draw (22,-6) rectangle (26,-9);
			\path (24,-7.5)  node {$B$};
			\draw [dashed] (15,0)--++(0,-10);
			\draw [dashed] (9,0)--++(0,-10);
			\draw [dashed] (27,0)--++(0,-10);
			\draw [dashed] (21,0)--++(0,-10);
		\end{tikzpicture}
		\subcaption{}\label{fig:joint-separ-length-1-b}
	\end{subfigure}
	
	\caption{Related to the proof of Theorem~\ref{joint-separ-length-1}.}\label{fig:joint-separ-length-1}
\end{figure}
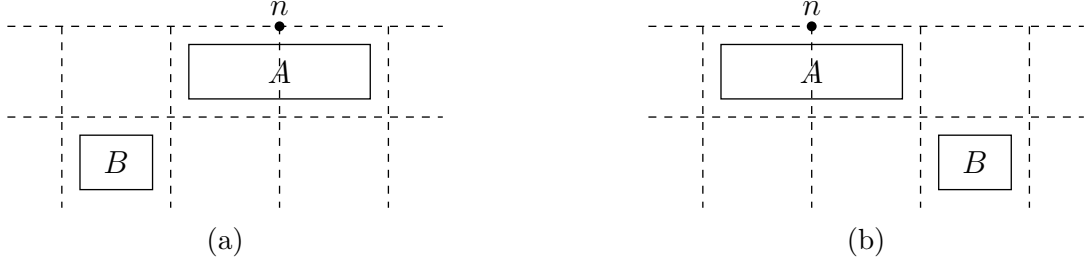

\begin{proof}
The joint equidistribution follows by applying the reverse operation that swaps occurrences of the patterns. Then we utilize the structure of a separable permutation $\pi$ of length $n$ shown in Figure~\ref{sepStructure} and consider two cases, depending on whether  $\pi$ is reducible or not.
 
\noindent \textbf{Case 1.} In the reducible case (see Figure~\ref{fig:joint-separ-length-1-a}), the block~$B$ must be nonempty, and can be any separable permutation. Case 1 can be divided into two subcases:

\begin{itemize}
		\item[(a)] $A = \varepsilon$ and $B \neq \varepsilon$. In this subcase, we obtain the term $x t_1 F(x,t_1)$,
since the element $n$ contributes an occurrence of the pattern
 \hspace{-1mm}$\pattern{scale = 0.6}{1}{1/1}{0/1,1/0}$ but not of
 \hspace{-1mm}$\pattern{scale = 0.6}{1}{1/1}{0/0,1/1}$.
Moreover, because of $n$, no element in $B$ can be an occurrence of
 \hspace{-1mm}$\pattern{scale = 0.6}{1}{1/1}{0/0,1/1}$, while $n$ cannot destroy any occurrence of the pattern
 \hspace{-1mm}$\pattern{scale = 0.6}{1}{1/1}{0/1,1/0}$ in $B$.

	\item[(b)] $A \neq \varepsilon$ and $B \neq \varepsilon$. In this subcase, to be disjoint from~(a),
$A$ has at least one element to the right of $n$, and by the structure shown in
Figure~\ref{sepStructure}, the minimum element $m$ of $A$ lies to the right of $n$.
Therefore, in this subcase we obtain the term
\[
\frac{F(x)-x}{2}\,F(x,t_1).
\]
Indeed,
\begin{itemize}
  \item[$\bullet$] 
  $B$ contributes the term $F(x,t_1)$, since, because of $n$, it cannot contribute any
  occurrences of the pattern
  \hspace{-1mm}$\pattern{scale = 0.6}{1}{1/1}{0/0,1/1}$, and is otherwise independent of the rest of $\pi$.

  \item[$\bullet$] 
  Since $B \neq \varepsilon$, the rest of $\pi$ cannot contribute any
  occurrences of
  \hspace{-1mm}$\pattern{scale = 0.6}{1}{1/1}{0/0,1/1}$.
  Moreover, because of $m$ (resp., $n$), no element to the left (resp., to the right) of $n$,
  including $n$ itself, can contribute an occurrence of the pattern
  \hspace{-1mm}$\pattern{scale = 0.6}{1}{1/1}{0/1,1/0}$.
  Therefore, $A$ together with $n$, being an irreducible permutation of length at least~$2$,
  contributes the term $\frac{F(x)-x}{2}$ (by reversal symmetry, half of separable
  permutations are reducible and half are irreducible).
\end{itemize}
\end{itemize}

\noindent \textbf{Case 2.}  In the irreducible case (see Figure~\ref{fig:joint-separ-length-1-b}), the block $B$ is nonempty, except in the trivial case $n=1$, and can be any separable permutation. We consider three subcases.

\begin{itemize}
	\item[(a)] $A=B=\varepsilon$, then  $\pi=1$ giving the term of  $t_1t_2x$.

	\item[(b)] $A = \varepsilon$ and $B \neq \varepsilon$, which gives $xt_2\,F(x,t_2)$ by reversing the permutations in Case~1(a).	
	\item[(c)]$A \neq \varepsilon$ and $B \neq \varepsilon$ , which gives $\frac{F(x) - x}{2}\, F(x,t_2)$ by reversing the permutations in Case 1(b).
	
\end{itemize}
	Combining the two cases above, we obtain
\begin{equation}\label{equ:joint-separ}
	F(x,t_1,t_2) = x t_1 t_2 + x t_1\, F(x,t_1) + x t_2\, F(x,t_2) + \frac{F(x) - x}{2}\, \big(F(x,t_1) + F(x,t_2)\big).
\end{equation}
Letting $t_2=1$ in~\eqref{equ:joint-separ}, we get
\begin{equation}\label{equ:joint-separ-A}
	F(x,t_1) = x t_1 + x t_1\, F(x,t_1) + x \, F(x) + \frac{F(x) - x}{2}\, \big(F(x,t_1) + F(x)\big).
\end{equation}

Substituting~\eqref{eq:separable-gf} into~\eqref{equ:joint-separ-A}, solving for $F(x,t_1)$,
and then replacing $t_1$ by $t$ yields~\eqref{eq:separ-length-1}.
Finally, substituting $F(x,t_1)$ and $F(x,t_2)$, both given by~\eqref{eq:separ-length-1},
into~\eqref{equ:joint-separ} yields~\eqref{eq:joint-separ-length-1}.\end{proof}

\section{Equidistribution of mesh patterns of length 2 with the same shading on separable permutations} \label{section:3}

There are $n!\cdot 2^{(n+1)^2}$ mesh patterns of length~$n$, and for $n=2$ there are already $1024$ distinct patterns, among which $2^{9}=512$ have distinct shadings. Computer experiments indicate that, on separable permutations, when the patterns $12$ and $21$ are regarded as mesh patterns with selectively shaded boxes, at most $124$ equidistributed pairs can exist.

In this section, we employ symmetry operations and generating function methods to prove that the $124$ pairs of mesh patterns presented in Tables~\ref{tab:sym-1group}--\ref{tab:g.f.} are equidistributed. These pairs are further classified into $38$ distinct symmetry-induced classes according to the corresponding symmetry operations.

\subsection{Symmetry operations}
Out of the $124$ equidistributed pairs obtained, $112$ pairs can be naturally derived through (composition of) symmetry operations. The explicit correspondences and symmetry-induced groupings of these pairs are presented in Tables~\ref{tab:sym-1group}--\ref{tab:sym-4group}.

\begin{table}[ht]
{
\renewcommand{\arraystretch}{1.5}
\setlength{\tabcolsep}{5pt}
\begin{center}
\begin{tabular}{|c c |c|| c c|c|| c c|c|}
\hline
\footnotesize{Nr.} & \footnotesize{Patterns} &
\footnotesize{Ref.} &
\footnotesize{Nr.} & \footnotesize{Patterns} &
\footnotesize{Ref.} &
\footnotesize{Nr.} & \footnotesize{Patterns} &
\footnotesize{Ref.} \\
\hline
\hline

\footnotesize{S1} &
$\pattern{scale = 0.6}{2}{1/1,2/2}
{0/0,0/1,0/2,1/0,1/1,1/2,2/0,2/1,2/2}\pattern{scale = 0.6}{2}{1/2,2/1}{0/0,0/1,0/2,1/0,1/1,1/2,2/0,2/1,2/2}$ &
\footnotesize{\multirow{1.2}{1.3cm}{Thm.~\ref{equ:Nr.1}}} &
\footnotesize{S2} &
$\pattern{scale = 0.6}{2}{1/1,2/2}{0/0,0/1,0/2,1/0,1/2,2/0,2/1,2/2}\pattern{scale = 0.6}{2}{1/2,2/1}{0/0,0/1,0/2,1/0,1/2,2/0,2/1,2/2}$ &
\footnotesize{\multirow{1.2}{1.3cm}{Thm.~\ref{equ:Nr.2}}} &
\footnotesize{S3} &
$\pattern{scale = 0.6}{2}{1/1,2/2}{0/0,0/2,1/1,2/0,2/2}\pattern{scale = 0.6}{2}{1/2,2/1}{0/0,0/2,1/1,2/0,2/2}$ & \footnotesize{\multirow{1.2}{1.3cm}{}} \\[5pt]  
\hline

\footnotesize{S4} &
$\pattern{scale = 0.6}{2}{1/1,2/2}
{0/1,1/0,1/2,2/1}\pattern{scale = 0.6}{2}{1/2,2/1}
{0/1,1/0,1/2,2/1}$ &
&
\footnotesize{S5} &
$\pattern{scale = 0.6}{2}{1/1,2/2}
{0/1,1/0,1/1,1/2,2/1}\pattern{scale = 0.6}{2}{1/2,2/1}{0/1,1/0,1/1,1/2,2/1}$ &
\footnotesize{\multirow{1.2}{1.3cm}{Thm.~\ref{equ:Nr.5}}} &
\footnotesize{S6} &
$\pattern{scale = 0.6}{2}{1/1,2/2}
{0/0,0/2,2/0,2/2}\pattern{scale = 0.6}{2}{1/2,2/1}
{0/0,0/2,2/0,2/2}$ & \footnotesize{\multirow{1.2}{1.3cm}{}} 
 \\[5pt] 
\hline
\footnotesize{S7} &
$\pattern{scale = 0.6}{2}{1/1,2/2}
{1/1}\pattern{scale = 0.6}{2}{1/2,2/1}
{1/1}$ &
 &
\footnotesize{S8} &
$\pattern{scale = 0.6}{2}{1/1,2/2}
{}\pattern{scale = 0.6}{2}{1/2,2/1}{}$ &
 \\[5pt]
\cline{1-6}
\end{tabular}
\end{center}
}
\vspace{-0.5cm}
\caption{Symmetry-induced equidistribution classes of mesh patterns with a single pattern per class, and references to their known enumeration.}
\label{tab:sym-1group}
\end{table}

For the patterns $12$ and $21$, both the reverse and the complement operations transform one into the other; that is, applying either operation maps $12$ to $21$, and vice versa. In all pairs of mesh patterns listed in Tables~\ref{tab:sym-1group}--\ref{tab:sym-4group}, the shaded regions are symmetric with respect to the horizontal or vertical axis, corresponding to the complement and reverse operations, respectively. Consequently, for each of the $112$ pairs of mesh patterns shown in these three tables, the two patterns within each pair are equidistributed under the action of the reverse or complement symmetry.

In Table~\ref{tab:sym-1group}, the single pair of mesh patterns is invariant under all three symmetries. In Table~\ref{tab:sym-2group}, each frame contains two such pairs that can be mapped to each other by one of the inverse, complement, or reverse operations. Similarly, in Table~\ref{tab:sym-4group}, each frame comprises four pairs of mesh patterns that can be mutually mapped and readily verified via these three symmetries or their compositions.

\begin{table}[ht]
{
\renewcommand{\arraystretch}{1.5}    
\setlength{\tabcolsep}{5pt}          
\begin{center}
\begin{tabular}{|c c c c|c||c c c c|c|}
\hline
\footnotesize{Nr.} & \footnotesize{Patterns} &
\footnotesize{Nr.} & \footnotesize{Patterns} &
\footnotesize{Ref.} &
\footnotesize{Nr.} & \footnotesize{Patterns} &
\footnotesize{Nr.} & \footnotesize{Patterns} &
\footnotesize{Ref.} \\
\hline
\hline

\footnotesize{S9} &
$\pattern{scale = 0.6}{2}{1/1,2/2}
{0/0,0/1,0/2,2/0,2/1,2/2}\pattern{scale = 0.6}{2}{1/2,2/1}{0/0,0/1,0/2,2/0,2/1,2/2}$ &
\footnotesize{S10} &
$\pattern{scale = 0.6}{2}{1/1,2/2}
{0/0,0/2,1/0,1/2,2/0,2/2}\pattern{scale = 0.6}{2}{1/2,2/1}{0/0,0/2,1/0,1/2,2/0,2/2}$ &
\footnotesize{\multirow{1.2}{1.3cm}{Thm.~\ref{equ:Nr.9}}}&
\footnotesize{S11} &
$\pattern{scale = 0.6}{2}{1/1,2/2}
{0/0,0/2,1/0,1/1,1/2,2/0,2/2}\pattern{scale = 0.6}{2}{1/2,2/1}{0/0,0/2,1/0,1/1,1/2,2/0,2/2}$ &
\footnotesize{S12} &
$\pattern{scale = 0.6}{2}{1/1,2/2}{0/0,0/1,0/2,1/1,2/0,2/1,2/2}\pattern{scale = 0.6}{2}{1/2,2/1}{0/0,0/1,0/2,1/1,2/0,2/1,2/2}$ &
\footnotesize{\multirow{1.2}{1.3cm}{Thm.~\ref{equ:Nr.11}}} \\[5pt]
\hline
\footnotesize{S13} &
$\pattern{scale = 0.6}{2}{1/1,2/2}
{1/0,1/1,1/2}\pattern{scale = 0.6}{2}{1/2,2/1}
{1/0,1/1,1/2}$&
\footnotesize{S14} &
$\pattern{scale = 0.6}{2}{1/1,2/2}
{0/1,1/1,2/1}\pattern{scale = 0.6}{2}{1/2,2/1}
{0/1,1/1,2/1}$ &
\footnotesize{\multirow{1.2}{1.3cm}{Thm.~\ref{equ:Nr.13}}} &
\footnotesize{S15} &
$\pattern{scale = 0.6}{2}{1/1,2/2}
{1/0,1/2}\pattern{scale = 0.6}{2}{1/2,2/1}
{1/0,1/2}$ &
\footnotesize{S16} &
$\pattern{scale = 0.6}{2}{1/1,2/2}
{0/1,2/1}\pattern{scale = 0.6}{2}{1/2,2/1}
{0/1,2/1}$ &
 \\[5pt]
\hline
	\end{tabular}
\end{center} 
}
\vspace{-0.5cm}
 	\caption{Symmetry-induced equidistribution classes of mesh patterns with two patterns per class, and references to their known enumeration.}\label{tab:sym-2group}
\end{table}

{
\setlength{\LTcapwidth}{\textwidth}
\captionsetup[table]{margin=0pt}
\renewcommand{\arraystretch}{1.85}   
\setlength{\tabcolsep}{5pt}          
\footnotesize
\begin{longtable}{|c c c c c c c c|c|}
\hline
\footnotesize{Nr.} & {\footnotesize Patterns} &
\footnotesize{Nr.} & {\footnotesize Patterns} &
\footnotesize{Nr.} & {\footnotesize Patterns} &
\footnotesize{Nr.} & {\footnotesize Patterns} &
{\footnotesize Ref.} \\
\hline
\hline
\endfirsthead

\hline
\footnotesize{Nr.} & {\footnotesize Patterns} &
\footnotesize{Nr.} & {\footnotesize Patterns} &
\footnotesize{Nr.} & {\footnotesize Patterns} &
\footnotesize{Nr.} & {\footnotesize Patterns} &
{\footnotesize Ref.} \\
\hline
\hline
\endhead

\footnotesize{S17} & $\pattern{scale = 0.6}{2}{1/1,2/2}{0/0,0/1,0/2,1/1,1/2,2/1,2/2,2/0}\pattern{scale = 0.6}{2}{1/2,2/1}{0/0,0/1,0/2,1/1,1/2,2/1,2/2,2/0}$  &
\footnotesize{S18} & $\pattern{scale = 0.6}{2}{1/1,2/2}{0/0,0/1,1/0,0/2,1/1,2/1,2/2,2/0}\pattern{scale = 0.6}{2}{1/2,2/1}{0/0,0/1,1/0,0/2,1/1,2/1,2/2,2/0}$  &
\footnotesize{S19} & $\pattern{scale = 0.6}{2}{1/1,2/2}{0/0,0/1,1/0,0/2,1/1,1/2,2/2,2/0}\pattern{scale = 0.6}{2}{1/2,2/1}{0/0,0/1,1/0,0/2,1/1,1/2,2/2,2/0}$  &
\footnotesize{S20} & $\pattern{scale = 0.6}{2}{1/1,2/2}{0/0,1/0,0/2,1/1,1/2,2/1,2/2,2/0}\pattern{scale = 0.6}{2}{1/2,2/1}{0/0,1/0,0/2,1/1,1/2,2/1,2/2,2/0}$  &
\footnotesize{\multirow{1.2}{1.3cm}{Thm.~\ref{equ:Nr.17}}}\\[5pt]
\hline
\footnotesize{S21} & $\pattern{scale = 0.6}{2}{1/1,2/2}{0/0,0/1,1/0,1/1,1/2,2/1,2/0}\pattern{scale = 0.6}{2}{1/2,2/1}{0/0,0/1,1/0,1/1,1/2,2/1,2/0}$  &
\footnotesize{S22}& $\pattern{scale = 0.6}{2}{1/1,2/2}{2/2,0/1,1/0,1/1,1/2,2/1,0/2}\pattern{scale = 0.6}{2}{1/2,2/1}{2/2,0/1,1/0,1/1,1/2,2/1,0/2}$  &
\footnotesize{S23}&$\pattern{scale = 0.6}{2}{1/1,2/2}{0/0,0/1,1/0,1/1,1/2,2/1,0/2}\pattern{scale = 0.6}{2}{1/2,2/1}{0/0,0/1,1/0,1/1,1/2,2/1,0/2}$  &
\footnotesize{S24} &$\pattern{scale = 0.6}{2}{1/1,2/2}{2/0,0/1,1/0,1/1,1/2,2/1,2/2}\pattern{scale = 0.6}{2}{1/2,2/1}{2/0,0/1,1/0,1/1,1/2,2/1,2/2}$  &
\footnotesize{\multirow{1.2}{1.3cm}{Thm.~\ref{equ:Nr.21}}}\\[5pt]
\hline
\footnotesize{S25} & $\pattern{scale = 0.6}{2}{1/1,2/2}{0/0,0/1,0/2,1/0,2/0,2/1,2/2}\pattern{scale = 0.6}{2}{1/2,2/1}{0/0,0/1,0/2,1/0,2/0,2/1,2/2}$  &
\footnotesize{S26} & $\pattern{scale = 0.6}{2}{1/1,2/2}{0/0,0/1,0/2,1/2,2/0,2/1,2/2}\pattern{scale = 0.6}{2}{1/2,2/1}{0/0,0/1,0/2,1/2,2/0,2/1,2/2}$  &
\footnotesize{S27} & $\pattern{scale = 0.6}{2}{1/1,2/2}{0/0,0/2,1/0,1/2,2/0,2/1,2/2}\pattern{scale = 0.6}{2}{1/2,2/1}{0/0,0/2,1/0,1/2,2/0,2/1,2/2}$  &
\footnotesize{S28} & $\pattern{scale = 0.6}{2}{1/1,2/2}{0/0,0/1,0/2,1/0,1/2,2/0,2/2}\pattern{scale = 0.6}{2}{1/2,2/1}{0/0,0/1,0/2,1/0,1/2,2/0,2/2}$  &
\footnotesize{\multirow{1.2}{1.3cm}{Thm.~\ref{equ:Nr.21}}}\\[5pt]
\hline
\footnotesize{S29} & $\pattern{scale = 0.6}{2}{1/1,2/2}{0/0,0/1,0/2,1/0,1/1,1/2}\pattern{scale = 0.6}{2}{1/2,2/1}{0/0,0/1,0/2,1/0,1/1,1/2}$ &
\footnotesize{S30}& $\pattern{scale = 0.6}{2}{1/1,2/2}{0/1,0/2,1/1,1/2,2/1,2/2}\pattern{scale = 0.6}{2}{1/2,2/1}{0/1,0/2,1/1,1/2,2/1,2/2}$  &
\footnotesize{S31} & $\pattern{scale = 0.6}{2}{1/1,2/2}{1/0,1/1,1/2,2/0,2/1,2/2}\pattern{scale = 0.6}{2}{1/2,2/1}{1/0,1/1,1/2,2/0,2/1,2/2}$  &
\footnotesize{S32} & $\pattern{scale = 0.6}{2}{1/1,2/2}{0/0,0/1,1/0,1/1,2/0,2/1}\pattern{scale = 0.6}{2}{1/2,2/1}{0/0,0/1,1/0,1/1,2/0,2/1}$& \footnotesize{\multirow{1.2}{1.3cm}{Thm.~\ref{equ:Nr.9}}}\\[5pt]
\hline
\footnotesize{S33} & $\pattern{scale = 0.6}{2}{1/1,2/2}{0/0,0/1,1/1,1/2,2/0,2/1}\pattern{scale = 0.6}{2}{1/2,2/1}{0/0,0/1,1/1,1/2,2/0,2/1}$ &
\footnotesize{S34} & $\pattern{scale = 0.6}{2}{1/1,2/2}{0/1,1/0,1/1,1/2,2/0,2/2}\pattern{scale = 0.6}{2}{1/2,2/1}{0/1,1/0,1/1,1/2,2/0,2/2}$  &
\footnotesize{S35} & $\pattern{scale = 0.6}{2}{1/1,2/2}{0/1,0/2,1/0,1/1,2/1,2/2}\pattern{scale = 0.6}{2}{1/2,2/1}{0/1,0/2,1/0,1/1,2/1,2/2}$  &
\footnotesize{S36} & $\pattern{scale = 0.6}{2}{1/1,2/2}{0/0,0/2,1/0,1/1,1/2,2/1}\pattern{scale = 0.6}{2}{1/2,2/1}{0/0,0/2,1/0,1/1,1/2,2/1}$
& \footnotesize{\multirow{1.2}{1.3cm}{Thm.~\ref{equ:Nr.33}}}\\[5pt]
\hline
\footnotesize{S37} & $\pattern{scale = 0.6}{2}{1/1,2/2}{0/0,0/1,0/2,2/0,2/2}\pattern{scale = 0.6}{2}{1/2,2/1}{0/0,0/1,0/2,2/0,2/2}$  &
\footnotesize{S38} & $\pattern{scale = 0.6}{2}{1/1,2/2}{0/0,0/2,1/2,2/0,2/2}\pattern{scale = 0.6}{2}{1/2,2/1}{0/0,0/2,1/2,2/0,2/2}$  &
\footnotesize{S39} & $\pattern{scale = 0.6}{2}{1/1,2/2}{0/0,0/2,2/0,2/1,2/2}\pattern{scale = 0.6}{2}{1/2,2/1}{0/0,0/2,2/0,2/1,2/2}$  &
\footnotesize{S40} & $\pattern{scale = 0.6}{2}{1/1,2/2}{0/0,0/2,1/0,2/0,2/2}\pattern{scale = 0.6}{2}{1/2,2/1}{0/0,0/2,1/0,2/0,2/2}$  & \footnotesize{\multirow{1.2}{1.3cm}{Thm.~\ref{equ:Nr.41}}}\\[5pt]
\hline
\footnotesize{S41} & $\pattern{scale = 0.6}{2}{1/1,2/2}{0/0,0/1,0/2,1/0,1/2}\pattern{scale = 0.6}{2}{1/2,2/1}{0/0,0/1,0/2,1/0,1/2}$  &
\footnotesize{S42} & $\pattern{scale = 0.6}{2}{1/1,2/2}{0/1,0/2,1/2,2/1,2/2}\pattern{scale = 0.6}{2}{1/2,2/1}{0/1,0/2,1/2,2/1,2/2}$  &
\footnotesize{S43} & $\pattern{scale = 0.6}{2}{1/1,2/2}{1/0,1/2,2/0,2/1,2/2}\pattern{scale = 0.6}{2}{1/2,2/1}{1/0,1/2,2/0,2/1,2/2}$  &
\footnotesize{S44} & $\pattern{scale = 0.6}{2}{1/1,2/2}{0/0,0/1,1/0,2/0,2/1}\pattern{scale = 0.6}{2}{1/2,2/1}{0/0,0/1,1/0,2/0,2/1}$  & \footnotesize{\multirow{1.2}{1.3cm}{Thm.~\ref{equ:Nr.37}}}\\[5pt]
\hline
\footnotesize{S45} & $\pattern{scale = 0.6}{2}{1/1,2/2}{0/0,0/2,1/0,1/1,2/0,2/2}\pattern{scale = 0.6}{2}{1/2,2/1}{0/0,0/2,1/0,1/1,2/0,2/2}$  &
\footnotesize{S46} & $\pattern{scale = 0.6}{2}{1/1,2/2}{0/0,0/2,1/1,1/2,2/0,2/2}\pattern{scale = 0.6}{2}{1/2,2/1}{0/0,0/2,1/1,1/2,2/0,2/2}$  &
\footnotesize{S47} & $\pattern{scale = 0.6}{2}{1/1,2/2}{0/0,0/2,1/1,2/0,2/1,2/2}\pattern{scale = 0.6}{2}{1/2,2/1}{0/0,0/2,1/1,2/0,2/1,2/2}$  &
\footnotesize{S48} & $\pattern{scale = 0.6}{2}{1/1,2/2}{0/0,0/1,0/2,1/1,2/0,2/2}\pattern{scale = 0.6}{2}{1/2,2/1}{0/0,0/1,0/2,1/1,2/0,2/2}$ & \footnotesize{\multirow{1.2}{1.3cm}{Thm.~\ref{equ:Nr.45}}}\\[5pt]
\hline
\footnotesize{S49} & $\pattern{scale = 0.6}{2}{1/1,2/2}{0/0,1/0,1/1,1/2,2/0}\pattern{scale = 0.6}{2}{1/2,2/1}{0/0,1/0,1/1,1/2,2/0}$  &
\footnotesize{S50} & $\pattern{scale = 0.6}{2}{1/1,2/2}{0/2,1/0,1/1,1/2,2/2}\pattern{scale = 0.6}{2}{1/2,2/1}{0/2,1/0,1/1,1/2,2/2}$  &
\footnotesize{S51} & $\pattern{scale = 0.6}{2}{1/1,2/2}{0/1,1/1,2/0,2/1,2/2}\pattern{scale = 0.6}{2}{1/2,2/1}{0/1,1/1,2/0,2/1,2/2}$  &
\footnotesize{S52} & $\pattern{scale = 0.6}{2}{1/1,2/2}{0/0,0/1,0/2,1/1,2/1}\pattern{scale = 0.6}{2}{1/2,2/1}{0/0,0/1,0/2,1/1,2/1}$  & \footnotesize{\multirow{1.2}{1.3cm}{Thm.~\ref{equ:Nr.49}}}\\[5pt]
\hline
\footnotesize{S53} & $\pattern{scale = 0.6}{2}{1/1,2/2}{1/1,2/0,2/1,2/2}\pattern{scale = 0.6}{2}{1/2,2/1}
{1/1,2/0,2/1,2/2}$ &
\footnotesize{S54} & $\pattern{scale = 0.6}{2}{1/1,2/2}{0/2,1/1,1/2,2/2}\pattern{scale = 0.6}{2}{1/2,2/1}
{0/2,1/1,1/2,2/2}$  &
\footnotesize{S55} & $\pattern{scale = 0.6}{2}{1/1,2/2}{0/0,1/0,1/1,2/0}\pattern{scale = 0.6}{2}{1/2,2/1}
{0/0,1/0,1/1,2/0}$ &
\footnotesize{S56} & $\pattern{scale = 0.6}{2}{1/1,2/2}{0/0,0/1,0/2,1/1}\pattern{scale = 0.6}{2}{1/2,2/1}
{0/0,0/1,0/2,1/1}$  & \footnotesize{\multirow{1.2}{1.3cm}{Thm.~\ref{equ:Nr.53}}}\\[5pt]
\hline
\footnotesize{S57} & $\pattern{scale = 0.6}{2}{1/1,2/2}{0/0,1/0,1/2,2/0}\pattern{scale = 0.6}{2}{1/2,2/1}
{0/0,1/0,1/2,2/0}$  &
\footnotesize{S58} & $\pattern{scale = 0.6}{2}{1/1,2/2}{0/2,1/0,1/2,2/2}\pattern{scale = 0.6}{2}{1/2,2/1}
{0/2,1/0,1/2,2/2}$  &
\footnotesize{S59} & $\pattern{scale = 0.6}{2}{1/1,2/2}{0/1,2/0,2/1,2/2}\pattern{scale = 0.6}{2}{1/2,2/1}
{0/1,2/0,2/1,2/2}$  &
\footnotesize{S60} & $\pattern{scale = 0.6}{2}{1/1,2/2}{0/0,0/1,0/2,2/1}\pattern{scale = 0.6}{2}{1/2,2/1}
{0/0,0/1,0/2,2/1}$  & \footnotesize{\multirow{1.2}{1.3cm}{Thm.~\ref{equ:Nr.57}}}\\[5pt]
\hline
\footnotesize{S61} & $\pattern{scale = 0.6}{2}{1/1,2/2}{0/2,1/2,2/2}\pattern{scale = 0.6}{2}{1/2,2/1}
{0/2,1/2,2/2}$ &
\footnotesize{S62} & $\pattern{scale = 0.6}{2}{1/1,2/2}{0/0,1/0,2/0}\pattern{scale = 0.6}{2}{1/2,2/1}
{0/0,1/0,2/0}$  &
\footnotesize{S63} & $\pattern{scale = 0.6}{2}{1/1,2/2}{2/0,2/1,2/2}\pattern{scale = 0.6}{2}{1/2,2/1}
{2/0,2/1,2/2}$  &
\footnotesize{S64} & $\pattern{scale = 0.6}{2}{1/1,2/2}{0/0,0/1,0/2}\pattern{scale = 0.6}{2}{1/2,2/1}
{0/0,0/1,0/2}$  & \footnotesize{\multirow{1.2}{1.3cm}{Thm.~\ref{equ:Nr.61}}}\\[5pt]
\hline
\footnotesize{S65} & $\pattern{scale = 0.6}{2}{1/1,2/2}{0/0,0/1,2/0,2/1}\pattern{scale = 0.6}{2}{1/2,2/1}
{0/0,0/1,2/0,2/1}$ &
\footnotesize{S66} & $\pattern{scale = 0.6}{2}{1/1,2/2}{1/0,1/2,2/0,2/2}\pattern{scale = 0.6}{2}{1/2,2/1}
{1/0,1/2,2/0,2/2}$  &
\footnotesize{S67} & $\pattern{scale = 0.6}{2}{1/1,2/2}{0/1,0/2,2/1,2/2}\pattern{scale = 0.6}{2}{1/2,2/1}
{0/1,0/2,2/1,2/2}$  &
\footnotesize{S68} & $\pattern{scale = 0.6}{2}{1/1,2/2}{0/0,0/2,1/0,1/2}\pattern{scale = 0.6}{2}{1/2,2/1}
{0/0,0/2,1/0,1/2}$ & \footnotesize{\multirow{1.2}{1.3cm}{Thm.~\ref{equ:Nr.65}}}\\[5pt]
\hline
\footnotesize{S69} & $\pattern{scale = 0.6}{2}{1/1,2/2}{0/0,0/1,1/0,1/2,2/0,2/1}\pattern{scale = 0.6}{2}{1/2,2/1}{0/0,0/1,1/0,1/2,2/0,2/1}$  &
\footnotesize{S70} & $\pattern{scale = 0.6}{2}{1/1,2/2}{0/1,0/2,1/0,1/2,2/1,2/2}\pattern{scale = 0.6}{2}{1/2,2/1}{0/1,0/2,1/0,1/2,2/1,2/2}$  &
\footnotesize{S71} & $\pattern{scale = 0.6}{2}{1/1,2/2}{0/1,1/0,1/2,2/0,2/1,2/2}\pattern{scale = 0.6}{2}{1/2,2/1}{0/1,1/0,1/2,2/0,2/1,2/2}$  &
\footnotesize{S72} & $\pattern{scale = 0.6}{2}{1/1,2/2}{0/0,0/1,0/2,1/0,1/2,2/1}\pattern{scale = 0.6}{2}{1/2,2/1}{0/0,0/1,0/2,1/0,1/2,2/1}$
& \footnotesize{\multirow{1.2}{1.3cm}{Thm.~\ref{equ:S69}}} \\[5pt] 
\hline
\footnotesize{S73} & $\pattern{scale = 0.6}{2}{1/1,2/2}{0/0,0/2,1/0,1/1,1/2}\pattern{scale = 0.6}{2}{1/2,2/1}{0/0,0/2,1/0,1/1,1/2}$  &
\footnotesize{S74} & $\pattern{scale = 0.6}{2}{1/1,2/2}{1/0,1/1,1/2,2/0,2/2}\pattern{scale = 0.6}{2}{1/2,2/1}{1/0,1/1,1/2,2/0,2/2}$  &
\footnotesize{S75} & $\pattern{scale = 0.6}{2}{1/1,2/2}{0/1,0/2,1/1,2/1,2/2}\pattern{scale = 0.6}{2}{1/2,2/1}{0/1,0/2,1/1,2/1,2/2}$  &
\footnotesize{S76} & $\pattern{scale = 0.6}{2}{1/1,2/2}{0/0,0/1,1/1,2/0,2/1}\pattern{scale = 0.6}{2}{1/2,2/1}{0/0,0/1,1/1,2/0,2/1}$  & \footnotesize{\multirow{1.2}{1.3cm}{Thm.~\ref{equ:Nr.73}}}\\[5pt]
\hline
\footnotesize{S77} & $\pattern{scale = 0.6}{2}{1/1,2/2}{0/0,0/2,1/0,1/2,2/1}\pattern{scale = 0.6}{2}{1/2,2/1}{0/0,0/2,1/0,1/2,2/1}$  &
\footnotesize{S78} & $\pattern{scale = 0.6}{2}{1/1,2/2}{0/1,1/0,1/2,2/0,2/2}\pattern{scale = 0.6}{2}{1/2,2/1}{0/1,1/0,1/2,2/0,2/2}$  &
\footnotesize{S79} & $\pattern{scale = 0.6}{2}{1/1,2/2}{0/1,0/2,1/0,2/1,2/2}\pattern{scale = 0.6}{2}{1/2,2/1}{0/1,0/2,1/0,2/1,2/2}$  &
\footnotesize{S80} & $\pattern{scale = 0.6}{2}{1/1,2/2}{0/0,0/1,1/2,2/0,2/1}\pattern{scale = 0.6}{2}{1/2,2/1}{0/0,0/1,1/2,2/0,2/1}$  & \footnotesize{\multirow{1.2}{1.3cm}{Thm.~\ref{equ:Nr.77}}}\\[5pt]
\hline
\footnotesize{S81} & $\pattern{scale = 0.6}{2}{1/1,2/2}{0/0,1/1,1/2,2/0}\pattern{scale = 0.6}{2}{1/2,2/1}
{0/0,1/1,1/2,2/0}$  &
\footnotesize{S82} & $\pattern{scale = 0.6}{2}{1/1,2/2}{0/2,1/0,1/1,2/2}\pattern{scale = 0.6}{2}{1/2,2/1}
{0/2,1/0,1/1,2/2}$  &
\footnotesize{S83} & $\pattern{scale = 0.6}{2}{1/1,2/2}{0/1,1/1,2/0,2/2}\pattern{scale = 0.6}{2}{1/2,2/1}
{0/1,1/1,2/0,2/2}$  &
\footnotesize{S84} & $\pattern{scale = 0.6}{2}{1/1,2/2}{0/0,1/1,0/2,2/1}\pattern{scale = 0.6}{2}{1/2,2/1}
{0/0,1/1,0/2,2/1}$  & \\[5pt]
\hline
\footnotesize{S85} & $\pattern{scale = 0.6}{2}{1/1,2/2}{0/0,1/1,2/0}\pattern{scale = 0.6}{2}{1/2,2/1}
{0/0,1/1,2/0}$  &
\footnotesize{S86} & $\pattern{scale = 0.6}{2}{1/1,2/2}{0/2,1/1,2/2}\pattern{scale = 0.6}{2}{1/2,2/1}
{0/2,1/1,2/2}$  &
\footnotesize{S87} & $\pattern{scale = 0.6}{2}{1/1,2/2}{2/0,1/1,2/2}\pattern{scale = 0.6}{2}{1/2,2/1}
{2/0,1/1,2/2}$  &
\footnotesize{S88} & $\pattern{scale = 0.6}{2}{1/1,2/2}{0/0,1/1,0/2}\pattern{scale = 0.6}{2}{1/2,2/1}
{0/0,1/1,0/2}$  & \\[5pt]
\hline

\footnotesize{S89} & $\pattern{scale = 0.6}{2}{1/1,2/2}{0/0,1/2,2/0}\pattern{scale = 0.6}{2}{1/2,2/1}
{0/0,1/2,2/0}$  &
\footnotesize{S90} & $\pattern{scale = 0.6}{2}{1/1,2/2}{0/2,1/0,2/2}\pattern{scale = 0.6}{2}{1/2,2/1}
{0/2,1/0,2/2}$  &
\footnotesize{S91} & $\pattern{scale = 0.6}{2}{1/1,2/2}{2/0,0/1,2/2}\pattern{scale = 0.6}{2}{1/2,2/1}
{2/0,0/1,2/2}$  &
\footnotesize{S92} & $\pattern{scale = 0.6}{2}{1/1,2/2}{0/0,2/1,0/2}\pattern{scale = 0.6}{2}{1/2,2/1}
{0/0,2/1,0/2}$  & \\[5pt]
\hline
\footnotesize{S93} & $\pattern{scale = 0.6}{2}{1/1,2/2}{0/0,2/0}\pattern{scale = 0.6}{2}{1/2,2/1}
{0/0,2/0}$  &
\footnotesize{S94} & $\pattern{scale = 0.6}{2}{1/1,2/2}{0/2,2/2}\pattern{scale = 0.6}{2}{1/2,2/1}
{0/2,2/2}$  &
\footnotesize{S95} & $\pattern{scale = 0.6}{2}{1/1,2/2}{2/0,2/2}\pattern{scale = 0.6}{2}{1/2,2/1}
{2/0,2/2}$  &
\footnotesize{S96} & $\pattern{scale = 0.6}{2}{1/1,2/2}{0/0,0/2}\pattern{scale = 0.6}{2}{1/2,2/1}
{0/0,0/2}$  & \\[5pt]
\hline
\footnotesize{S97} & $\pattern{scale = 0.6}{2}{1/1,2/2}{0/1,1/0,1/1,2/1}\pattern{scale = 0.6}{2}{1/2,2/1}
{0/1,1/0,1/1,2/1}$  &
\footnotesize{S98} & $\pattern{scale = 0.6}{2}{1/1,2/2}{0/1,1/2,1/1,2/1}\pattern{scale = 0.6}{2}{1/2,2/1}
{0/1,1/2,1/1,2/1}$  &
\footnotesize{S99} & $\pattern{scale = 0.6}{2}{1/1,2/2}{1/0,1/2,1/1,2/1}\pattern{scale = 0.6}{2}{1/2,2/1}
{1/0,1/2,1/1,2/1}$  &
\footnotesize{S100} & $\pattern{scale = 0.6}{2}{1/1,2/2}{1/0,1/2,1/1,0/1}\pattern{scale = 0.6}{2}{1/2,2/1}
{1/0,1/2,1/1,0/1}$  & \\[5pt]
\hline
\footnotesize{S101} & $\pattern{scale = 0.6}{2}{1/1,2/2}{0/1,1/0,2/1}\pattern{scale = 0.6}{2}{1/2,2/1}
{0/1,1/0,2/1}$  &
\footnotesize{S102} & $\pattern{scale = 0.6}{2}{1/1,2/2}{0/1,1/2,2/1}\pattern{scale = 0.6}{2}{1/2,2/1}
{0/1,1/2,2/1}$  &
\footnotesize{S103} & $\pattern{scale = 0.6}{2}{1/1,2/2}{1/0,1/2,2/1}\pattern{scale = 0.6}{2}{1/2,2/1}
{1/0,1/2,2/1}$  &
\footnotesize{S104} & $\pattern{scale = 0.6}{2}{1/1,2/2}{1/0,1/2,0/1}\pattern{scale = 0.6}{2}{1/2,2/1}
{1/0,1/2,0/1}$  & \\[5pt]
\hline
\footnotesize{S105} & $\pattern{scale = 0.6}{2}{1/1,2/2}{1/0,1/1}\pattern{scale = 0.6}{2}{1/2,2/1}
{1/0,1/1}$  &
\footnotesize{S106} & $\pattern{scale = 0.6}{2}{1/1,2/2}{1/2,1/1}\pattern{scale = 0.6}{2}{1/2,2/1}
{1/2,1/1}$  &
\footnotesize{S107} & $\pattern{scale = 0.6}{2}{1/1,2/2}{2/1,1/1}\pattern{scale = 0.6}{2}{1/2,2/1}
{2/1,1/1}$  &
\footnotesize{S108} & $\pattern{scale = 0.6}{2}{1/1,2/2}{0/1,1/1}\pattern{scale = 0.6}{2}{1/2,2/1}
{0/1,1/1}$  & \\[5pt]
\hline
\footnotesize{S109} & $\pattern{scale = 0.6}{2}{1/1,2/2}
{1/0}\pattern{scale = 0.6}{2}{1/2,2/1}
{1/0}$  &
\footnotesize{S110} & $\pattern{scale = 0.6}{2}{1/1,2/2}
{1/2}\pattern{scale = 0.6}{2}{1/2,2/1}
{1/2}$  &
\footnotesize{S111} & $\pattern{scale = 0.6}{2}{1/1,2/2}
{2/1}\pattern{scale = 0.6}{2}{1/2,2/1}
{2/1}$  &
\footnotesize{S112} & $\pattern{scale = 0.6}{2}{1/1,2/2}
{0/1}\pattern{scale = 0.6}{2}{1/2,2/1}
{0/1}$  & \\[5pt]
\hline
\caption{Symmetry-induced equidistribution classes of mesh patterns, with four patterns per class, and references to their known enumeration.}\label{tab:sym-4group}
\end{longtable}
} 

\subsection{The generating functions method}

Unlike the other pairs of mesh patterns, the equidistribution of each pair 
in Table~\ref{tab:g.f.} cannot be established solely by symmetry arguments. 
Instead, we employ generating function techniques to provide a proof.

\begin{table}[ht]
 	{
 		\renewcommand{\arraystretch}{1.5}
 		\setlength{\tabcolsep}{5pt}
 \begin{center} 
 		\begin{tabular}{|c c c c  c c c
 		c|c|}
 			\hline
 		\footnotesize{Nr.}	& {\footnotesize patterns}  & \footnotesize{Nr.}	& {\footnotesize Patterns}  & \footnotesize{Nr.}	& {\footnotesize Patterns}  & \footnotesize{Nr.}	& {\footnotesize Patterns} & {\footnotesize Ref.} \\
 			\hline
 			\hline
\footnotesize{S113} & $\pattern{scale = 0.6}{2}{1/1,2/2}{0/0,0/1,1/0,0/2,1/1,1/2,2/1,2/2}\pattern{scale = 0.6}{2}{1/2,2/1}{0/0,0/1,1/0,0/2,1/1,1/2,2/1,2/2}$  &
\footnotesize{S114} & $\pattern{scale = 0.6}{2}{1/1,2/2}{0/1,1/0,0/2,1/1,1/2,2/1,2/2,2/0}\pattern{scale = 0.6}{2}{1/2,2/1}{0/1,1/0,0/2,1/1,1/2,2/1,2/2,2/0}$  &
\footnotesize{S115} & $\pattern{scale = 0.6}{2}{1/1,2/2}{0/0,0/1,1/0,1/1,1/2,2/1,2/2,2/0}\pattern{scale = 0.6}{2}{1/2,2/1}{0/0,0/1,1/0,1/1,1/2,2/1,2/2,2/0}$  &
\footnotesize{S116} & $\pattern{scale = 0.6}{2}{1/1,2/2}{0/0,0/1,1/0,0/2,1/1,1/2,2/1,2/0}\pattern{scale = 0.6}{2}{1/2,2/1}{0/0,0/1,1/0,0/2,1/1,1/2,2/1,2/0}$ & \footnotesize{\multirow{1.2}{1.3cm}{Thm.~\ref{equ:Nr.113}}}\\[5pt]
\hline
\footnotesize{S117} & $\pattern{scale = 0.6}{2}{1/1,2/2}{0/0,0/1,1/0,1/1,1/2,2/0,2/2}\pattern{scale = 0.6}{2}{1/2,2/1}{0/0,0/1,1/0,1/1,1/2,2/0,2/2}$  &
\footnotesize{S118} & $\pattern{scale = 0.6}{2}{1/1,2/2}{0/1,0/2,1/0,1/1,1/2,2/0,2/2}\pattern{scale = 0.6}{2}{1/2,2/1}{0/1,0/2,1/0,1/1,1/2,2/0,2/2}$  &
\footnotesize{S119} & $\pattern{scale = 0.6}{2}{1/1,2/2}{0/0,0/2,1/0,1/1,1/2,2/1,2/2}\pattern{scale = 0.6}{2}{1/2,2/1}{0/0,0/2,1/0,1/1,1/2,2/1,2/2}$  &
\footnotesize{S120} & $\pattern{scale = 0.6}{2}{1/1,2/2}{0/0,0/2,1/0,1/1,1/2,2/0,2/1}\pattern{scale = 0.6}{2}{1/2,2/1}{0/0,0/2,1/0,1/1,1/2,2/0,2/1}$  
&\footnotesize{\multirow{2.5}{1.3cm}{Thm.~\ref{equ:Nr.117}}} \\[5pt]
\footnotesize{S121} & $\pattern{scale = 0.6}{2}{1/1,2/2}{0/0,0/2,0/1,1/1,1/2,2/0,2/1}\pattern{scale = 0.6}{2}{1/2,2/1}{0/0,0/2,0/1,1/1,1/2,2/0,2/1}$  &
\footnotesize{S122} & $\pattern{scale = 0.6}{2}{1/1,2/2}{0/0,0/1,1/0,1/1,0/2,2/1,2/2}\pattern{scale = 0.6}{2}{1/2,2/1}{0/0,0/1,1/0,1/1,0/2,2/1,2/2}$  &
\footnotesize{S123} & $\pattern{scale = 0.6}{2}{1/1,2/2}{0/0,1/2,0/1,1/1,2/0,2/1,2/2}\pattern{scale = 0.6}{2}{1/2,2/1}{0/0,1/2,0/1,1/1,2/0,2/1,2/2}$  &
\footnotesize{S124} & $\pattern{scale = 0.6}{2}{1/1,2/2}{0/1,0/2,1/0,1/1,2/0,2/1,2/2}\pattern{scale = 0.6}{2}{1/2,2/1}{0/1,0/2,1/0,1/1,2/0,2/1,2/2}$ &
\\[5pt]
\hline	
	\end{tabular}
\end{center} 
}

\vspace{-0.5cm}
 	\caption{Symmetry-induced equidistribution classes of mesh patterns proved via generating functions.}\label{tab:g.f.}
\end{table}

In Table~\ref{tab:g.f.}, pairs of mesh patterns sharing the same equidistribution are grouped within the same frame. It suffices to prove equidistribution for only one pair of patterns in each frame; the remaining pairs are obtained by the inverse, reverse, and complement operations.

\begin{thm}[Pairs~S113--S116]\label{equ:Nr.113}
We have  $\pattern{scale = 0.6}{2}{1/1,2/2}{0/0,0/1,1/0,0/2,1/1,1/2,2/1,2/2}\sim_{jd}\hspace{-1mm}
\pattern{scale = 0.6}{2}{1/2,2/1}{0/0,0/1,1/0,0/2,1/1,1/2,2/1,2/2}$. Moreover, 
\begin{eqnarray}F(x,t) &:=& \sum_{n \ge 1} x^n \sum_{\pi \in S_n(2413,3142)} 
	t^{\hspace{-1mm}\pattern{scale = 0.5}{2}{1/1,2/2}{0/0,0/1,1/0,0/2,1/1,1/2,2/1,2/2}(\pi)}
	= \sum_{n \ge 1} x^n \sum_{\pi \in S_n(2413,3142)} 
	t^{\hspace{-1mm}\pattern{scale = 0.5}{2}{1/2,2/1}{0/0,0/1,1/0,0/2,1/1,1/2,2/1,2/2}(\pi)} \nonumber \\
	&=& \frac{1}{2}\Big( 1 - x - \sqrt{1-6x+x^{2}}  + (t-1)x^2\big(3-x-\sqrt{1-6x+x^{2}}\big) \label{eq:Nr.113}
\Big);
\end{eqnarray}
\begin{eqnarray}F(x,t_1,t_2) &:=& \sum_{n \ge 1} x^n \sum_{\pi \in S_n(2413,3142)} 
	t_1^{\hspace{-1mm}\pattern{scale = 0.5}{2}{1/1,2/2}{0/0,0/1,1/0,0/2,1/1,1/2,2/1,2/2}(\pi)}
	t_2^{\hspace{-1mm}\pattern{scale = 0.5}{2}{1/2,2/1}{0/0,0/1,1/0,0/2,1/1,1/2,2/1,2/2}(\pi)} \nonumber \\
	&=& \frac{1}{2}\Big( 1 - x - \sqrt{1-6x+x^{2}}  + (t_1+t_2-2)x^2\big(3-x-\sqrt{1-6x+x^{2}}\big) \nonumber
\Big).
\end{eqnarray}
\end{thm}

\begin{proof}
Clearly, if a separable permutation $\pi = \pi_1 \cdots \pi_n$ has an occurrence of the pattern \hspace{-1mm}$\pattern{scale = 0.6}{2}{1/1,2/2}{0/0,0/1,1/0,0/2,1/1,1/2,2/1,2/2}$ (resp., \hspace{-1mm}
$\pattern{scale = 0.6}{2}{1/2,2/1}{0/0,0/1,1/0,0/2,1/1,1/2,2/1,2/2}$), then $\pi_1\pi_2=(n-1)n$ (resp., $\pi_1\pi_2=n(n-1)$) is the only occurrence of the pattern in $\pi$, and $\pi_3\cdots\pi_n$ is any separable permutation in $S_{n-2}$. Therefore, the map that swaps $n$ and $n-1$ in all permutations of length~$n$ establishes the joint equidistribution $\pattern{scale = 0.6}{2}{1/1,2/2}{0/0,0/1,1/0,0/2,1/1,1/2,2/1,2/2}\sim_{jd}\hspace{-1mm}
\pattern{scale = 0.6}{2}{1/2,2/1}{0/0,0/1,1/0,0/2,1/1,1/2,2/1,2/2}$. 

Consider the pattern \hspace{-1mm}$\pattern{scale = 0.6}{2}{1/1,2/2}{0/0,0/1,1/0,0/2,1/1,1/2,2/1,2/2}$, and let $A(x,t)$ be the g.f.\ for permutations with one occurrence of the pattern. Then, from the above, $A(x,t) = t x^2 ( F(x) + 1 )$, and 
\begin{equation} \label{Nr113-expr}
    F(x,t) = A(x,t) + ( F(x) - A(x,1))=(t-1) x^2 ( F(x) + 1 ) + F(x).
\end{equation}
Substituting \eqref{eq:separable-gf} into \eqref{Nr113-expr} yields \eqref{eq:Nr.113}. 

Finally, because both patterns cannot occur in the same permutation, we have
\begin{eqnarray} \label{Nr113-expr-1}
    F(x,t_1,t_2) &=& A(x,t_1) + A(x,t_2) + ( F(x) - 2A(x,1)) \nonumber \\
    &=& (t_1+t_2-2) x^2 ( F(x) + 1 ) + F(x),
\end{eqnarray}
and the formula for $F(x,t_1,t_2)$ is obtained by substituting \eqref{eq:separable-gf} into \eqref{Nr113-expr-1}.
\end{proof}

\begin{thm}[Pairs~S117--S124]\label{equ:Nr.117}
We have  \hspace{-1mm}$\pattern{scale = 0.6}{2}{1/1,2/2}{0/0,0/1,1/0,1/1,1/2,2/0,2/2}
\sim_{jd}\hspace{-1mm}\pattern{scale = 0.6}{2}{1/2,2/1}{0/0,0/1,1/0,1/1,1/2,2/0,2/2}$. Moreover, 
\begin{eqnarray}
	F(x,t) &:=& \sum_{n \ge 1} x^n \sum_{\pi \in S_n(2413,3142)} 
	t^{\hspace{-1mm}\pattern{scale = 0.5}{2}{1/1,2/2}{0/0,0/1,1/0,1/1,1/2,2/0,2/2}(\pi)}
	= \sum_{n \ge 1} x^n \sum_{\pi \in S_n(2413,3142)} 
	t^{\hspace{-1mm}\pattern{scale = 0.5}{2}{1/2,2/1}{0/0,0/1,1/0,1/1,1/2,2/0,2/2}(\pi)} \nonumber \\
	  &=& \frac{1}{4} \Big(
     t x^2(x-3 + \sqrt{1 - 6x + x^2})^2\nonumber\\
  &&\quad   -x^2(x-3 + \sqrt{1 - 6x + x^2})^2
     - 2(x-1 + \sqrt{1 - 6x + x^2})
     \Big);\label{eq:Nr.117}
\end{eqnarray}	
\begin{eqnarray}F(x,t_1,t_2) &:=& \sum_{n \ge 1} x^n \sum_{\pi \in S_n(2413,3142)} 
	t_1^{\hspace{-1mm}\pattern{scale = 0.5}{2}{1/1,2/2}{0/0,0/1,1/0,1/1,1/2,2/0,2/2}(\pi)}
	t_2^{\hspace{-1mm}\pattern{scale = 0.5}{2}{1/2,2/1}{0/0,0/1,1/0,1/1,1/2,2/0,2/2}(\pi)} \nonumber \\
	&=& F(x)+x^2\Big((t_1+t_2)(1-x)+(t_1t_2+1)x-2\Big)(F(x)+1)^2, \label{eq:Nr117-jd}
\end{eqnarray}
where $F(x)$ is given by (\ref{eq:separable-gf}).
\end{thm}

\begin{proof}
Clearly, if a separable permutation $\pi$ contains an occurrence of the pattern 
$\hspace{-1mm}\pattern{scale = 0.6}{2}{1/1,2/2}{0/0,0/1,1/0,1/1,1/2,2/0,2/2}$ 
(resp., $\hspace{-1mm}\pattern{scale = 0.6}{2}{1/2,2/1}{0/0,0/1,1/0,1/1,1/2,2/0,2/2}$), 
then for some element $m$, we can write $\pi = A1 m B$ (resp., $\pi = A m 1  B$), 
where each element of $A$, if any, is larger than $m$ and each element of $B$, if any, is smaller than $m$. 
Moreover, each mesh pattern can occur at most once in $\pi$, and the reduced forms of $A$ and $B$ can be any separable permutations 
($\pi$ is separable if and only if both $A$ and $B$ are separable).  Therefore, if $A(x,t)$ is the g.f.\ for permutations with exactly one occurrence of either pattern, 
then $$F(x,t) = A(x,t) + ( F(x) - A(x,1)),$$ and 
\begin{equation}\label{A(x,t):Nr117}A(x,t) = t x^2 ( F(x) + 1 )^2,
\end{equation} 
from which, together with (\ref{eq:separable-gf}), we obtain (\ref{eq:Nr.117}) and $\pattern{scale = 0.6}{2}{1/1,2/2}{0/0,0/1,1/0,1/1,1/2,2/0,2/2}
\sim_{d}\hspace{-1mm}\pattern{scale = 0.6}{2}{1/2,2/1}{0/0,0/1,1/0,1/1,1/2,2/0,2/2}$.

To deal with the joint distribution, first observe that the only way both patterns can occur in the same permutation \( \pi \) is if $\pi = A\,(m+1)\,1\,m\,B$,
where every element in \(A\), if any, is larger than \(m+1\), and every element in \(B\), if any, is smaller than \(m\). The g.f.\ for such permutations is clearly $t_1t_2x^3(F(x)+1)^2$. Now, the g.f. for permutations containing $\hspace{-1mm}\pattern{scale = 0.6}{2}{1/1,2/2}{0/0,0/1,1/0,1/1,1/2,2/0,2/2}$ (resp., $\hspace{-1mm}\pattern{scale = 0.6}{2}{1/2,2/1}{0/0,0/1,1/0,1/1,1/2,2/0,2/2}$), but not $\hspace{-1mm}\pattern{scale = 0.6}{2}{1/2,2/1}{0/0,0/1,1/0,1/1,1/2,2/0,2/2}$ (resp., $\hspace{-1mm}\pattern{scale = 0.6}{2}{1/1,2/2}{0/0,0/1,1/0,1/1,1/2,2/0,2/2}$), is $A(x,t_1)-t_1x^3(F(x)+1)^2$ (resp., $A(x,t_2)-t_2x^3(F(x)+1)^2$). Moreover, the g.f.\ for separable permutations avoiding both patterns is
$$F(x)-(2A(x,1)-x^3(F(x)+1)^2).$$
Combining the four generating functions we have that $F(x,t_1,t_2)$ is given by
$$A(x,t_1)+A(x,t_2)-(t_1+t_2)x^3(F(x)+1)^2+t_1t_2x^3(F(x)+1)^2+F(x)-(2A(x,1)-x^3(F(x)+1)^2),$$
and formula (\ref{eq:Nr117-jd}) is obtained by applying (\ref{A(x,t):Nr117}). Since $F(x,t_1,t_2)=F(x,t_2,t_1)$, we have $\pattern{scale = 0.6}{2}{1/1,2/2}{0/0,0/1,1/0,1/1,1/2,2/0,2/2}
\sim_{jd}\hspace{-1mm}\pattern{scale = 0.6}{2}{1/2,2/1}{0/0,0/1,1/0,1/1,1/2,2/0,2/2}$.
\end{proof}

\section{Distribution of mesh pattern of length 2 on separable permutations}

In this section, we determine the distributions of 21 increasing (resp., decreasing) mesh patterns of length~2 on separable permutations. For the pairs of mesh patterns presented in Tables~\ref{tab:sym-1group}--\ref{tab:sym-4group}, we select the first pair of patterns in each class (that is, within each frame) as the representative pair for determining the corresponding distribution, since the equidistribution of all pairs has already been established. In our analysis, we always choose the pattern~$12$ as the representative.

\subsection{“Trivial” distributions}
By a ``trivial'' distribution we mean the situation when either the pattern in question can occur at most once, or the pattern's occurrences can be easily inferred from the shape of the pattern. There are nine classes of pairs of such patterns:
\vspace{0.1cm}
\begin{center}
\begin{tabular}{rrrrr}
S1 = $\pattern{scale = 0.6}{2}{1/1,2/2}
{0/0,0/1,0/2,1/0,1/1,1/2,2/0,2/1,2/2}\pattern{scale = 0.6}{2}{1/2,2/1}{0/0,0/1,0/2,1/0,1/1,1/2,2/0,2/1,2/2}$; & \hspace{-0.3cm}
S2 = $\pattern{scale = 0.6}{2}{1/1,2/2}
{0/0,0/1,0/2,1/0,1/2,2/0,2/1,2/2}\pattern{scale = 0.6}{2}{1/2,2/1}{0/0,0/1,0/2,1/0,1/2,2/0,2/1,2/2}$;&\hspace{-0.3cm}
S9 = $\pattern{scale = 0.6}{2}{1/1,2/2}
{0/0,0/1,0/2,2/0,2/1,2/2}\pattern{scale = 0.6}{2}{1/2,2/1}{0/0,0/1,0/2,2/0,2/1,2/2}$;&\hspace{-0.3cm}
S11 = $\pattern{scale = 0.6}{2}{1/1,2/2}
{0/0,0/2,1/0,1/1,1/2,2/0,2/2}\pattern{scale = 0.6}{2}{1/2,2/1}{0/0,0/2,1/0,1/1,1/2,2/0,2/2}$;\\[5pt]
S17 = $\pattern{scale = 0.6}{2}{1/1,2/2}{0/0,0/1,0/2,1/1,1/2,2/1,2/2,2/0}\pattern{scale = 0.6}{2}{1/2,2/1}{0/0,0/1,0/2,1/1,1/2,2/1,2/2,2/0}$;&\hspace{-0.3cm}
S21 = $\pattern{scale = 0.6}{2}{1/1,2/2}{0/0,0/1,1/0,1/1,1/2,2/1,2/0}\pattern{scale = 0.6}{2}{1/2,2/1}{0/0,0/1,1/0,1/1,1/2,2/1,2/0}$; &\hspace{-0.3cm}
S25 = $\pattern{scale = 0.6}{2}{1/1,2/2}{0/0,0/1,0/2,1/0,2/0,2/1,2/2}\pattern{scale = 0.6}{2}{1/2,2/1}{0/0,0/1,0/2,1/0,2/0,2/1,2/2}$;  &\hspace{-0.3cm}
S29 = $\pattern{scale = 0.6}{2}{1/1,2/2}{0/0,0/1,0/2,1/0,1/1,1/2}\pattern{scale = 0.6}{2}{1/2,2/1}{0/0,0/1,0/2,1/0,1/1,1/2}$;\\[5pt]
S49 = $\pattern{scale = 0.6}{2}{1/1,2/2}{0/0,1/0,1/1,1/2,2/0}\pattern{scale = 0.6}{2}{1/2,2/1}{0/0,1/0,1/1,1/2,2/0}$.
\end{tabular}
\end{center}

\begin{thm}[Pair~S1]\label{equ:Nr.1}
We have 
\begin{eqnarray}\label{eq:Nr.1}
F(x,t) &:=& \sum_{n \ge 1} x^n \sum_{\pi \in S_n(2413,3142)} 
	t^{\hspace{-1mm}\pattern{scale = 0.5}{2}{1/1,2/2}{0/0,0/1,0/2,1/0,1/1,1/2,2/0,2/1,2/2}(\pi)}
	= \sum_{n \ge 1} x^n \sum_{\pi \in S_n(2413,3142)} 
	t^{\hspace{-1mm}\pattern{scale = 0.5}{2}{1/2,2/1}{0/0,0/1,0/2,1/0,1/1,1/2,2/0,2/1,2/2}(\pi)} \nonumber \\
&=& \frac{1}{2} \left( 1 - x + 2(t-1)x^2 - \sqrt{1 - 6x + x^2} \right).
\end{eqnarray}
$$F(x,t_1,t_2) := \sum_{n \ge 1} x^n \sum_{\pi \in S_n(2413,3142)} 
t_1^{\hspace{-1mm}\pattern{scale = 0.5}{2}{1/1,2/2}{0/0,0/1,0/2,1/0,1/1,1/2,2/0,2/1,2/2}(\pi)}
t_2^{\hspace{-1mm}\pattern{scale = 0.5}{2}{1/2,2/1}{0/0,0/1,0/2,1/0,1/1,1/2,2/0,2/1,2/2}(\pi)} =(t_1+t_2-2)x^2+\frac{1 - x - \sqrt{1 - 6x + x^2}}{2}.$$
\end{thm}

\begin{proof}
The only separable permutation that contains the pattern \hspace{-1mm}$\pattern{scale = 0.6}{2}{1/1,2/2}{0/0,0/1,0/2,1/0,1/1,1/2,2/0,2/1,2/2}$ is $12$. Hence, 
\begin{equation}\label{Nr1-expr-A}
    F(x,t) = tx^2 + F(x) - x^2 .
\end{equation}
By substituting~\eqref{eq:separable-gf}~into~\eqref{Nr1-expr-A}, we obtain~\eqref{eq:Nr.1}. Also, it is clear that 
$$F(x,t_1,t_2)=t_1x^2+t_2x^2+F(x)-2x^2,$$ which, together with \eqref{eq:separable-gf}, leads to the desired result.
\end{proof}

%The following theorem establishes the distribution of the pair S2, $\pattern{scale = 0.6}{2}{1/2,2/1}{0/0,0/1,0/2,1/0,1/2,2/0,2/1,2/2}$ and$\pattern{scale = 0.6}{2}{1/1,2/2}{0/0,0/1,0/2,1/0,1/2,2/0,2/1,2/2}$.

\begin{thm}[Pair~S2]\label{equ:Nr.2}
Theorem~\ref{equ:Nr.113} also holds for the mesh patterns  \hspace{-1mm}$\pattern{scale = 0.6}{2}{1/1,2/2}{0/0,0/1,0/2,1/0,1/2,2/0,2/1,2/2}$ and \hspace{-1mm}$\pattern{scale = 0.6}{2}{1/2,2/1}{0/0,0/1,0/2,1/0,1/2,2/0,2/1,2/2}$.
\end{thm}

\begin{proof}
The pattern~$\pattern{scale = 0.6}{2}{1/1,2/2}{0/0,0/1,0/2,1/0,1/2,2/0,2/1,2/2}$ can occur at most once in any permutation $\pi = \pi_1 \cdots \pi_n$. In this case, we must have $\pi_1 = 1$ and $\pi_n = n$ and $\red(\pi_2\cdots\pi_{n-1})$ can be any separable permutation of length $n-2$. If $A(x,t)$ is the g.f.\ of separable permutations with one occurrence of the pattern, then all formulas and functional relations in the proof of Theorem~\ref{equ:Nr.113} are exactly the same, leading to the desired result.
\end{proof}

The following lemma is an immediate corollary of the fact that separable permutations are closed under reversal and under the complement.

\begin{lemma}\label{half-first-last}
For every \(n\ge 2\), exactly half of the separable permutations of length \(n\) have their first element larger than their last element. Likewise, exactly half begin with an ascent, and half begin with a descent.
\end{lemma}

\begin{thm}[Pairs~S9, S10, S29--S32]\label{equ:Nr.9}
We have
\begin{eqnarray}
F(x,t) &:=& \sum_{n \ge 1} x^n \sum_{\pi \in S_n(2413,3142)} 
	t^{\hspace{-1mm}\pattern{scale = 0.5}{2}{1/1,2/2}{0/0,0/1,0/2,2/0,2/1,2/2}(\pi)}
	= \sum_{n \ge 1} x^n \sum_{\pi \in S_n(2413,3142)} 
	t^{\hspace{-1mm}\pattern{scale = 0.5}{2}{1/2,2/1}{0/0,0/1,0/2,2/0,2/1,2/2}(\pi)} \nonumber \\
	&=& \sum_{n \ge 1} x^n \sum_{\pi \in S_n(2413,3142)} 
	t^{\hspace{-1mm}\pattern{scale = 0.5}{2}{1/1,2/2}
    {0/0,0/1,0/2,1/0,1/1,1/2}(\pi)}
	= \sum_{n \ge 1} x^n \sum_{\pi \in S_n(2413,3142)} 
	t^{\hspace{-1mm}\pattern{scale = 0.5}{2}{1/2,2/1}
   {0/0,0/1,0/2,1/0,1/1,1/2}(\pi)} \nonumber \\
	&=&  x + \frac{1}{4}(1 + t)\left(1-3x - \sqrt{1 - 6x + x^2}\right).\label{eq:Nr.9}
	\end{eqnarray}
\begin{eqnarray}
F(x,t_1,t_2) &:=& \sum_{n \ge 1} x^n \sum_{\pi \in S_n(2413,3142)} 
t_1^{\hspace{-1mm}\pattern{scale = 0.5}{2}{1/1,2/2}{0/0,0/1,0/2,2/0,2/1,2/2}(\pi)}
t_2^{\hspace{-1mm}\pattern{scale = 0.5}{2}{1/2,2/1}{0/0,0/1,0/2,2/0,2/1,2/2}(\pi)} = \sum_{n \ge 1} x^n \sum_{\pi \in S_n(2413,3142)} 
t_1^{\hspace{-1mm}\pattern{scale = 0.5}{2}{1/1,2/2}{0/0,0/1,0/2,1/0,1/1,1/2}(\pi)}
t_2^{\hspace{-1mm}\pattern{scale = 0.5}{2}{1/2,2/1}{0/0,0/1,0/2,1/0,1/1,1/2}(\pi)} \nonumber\\
&=& x + \frac{1}{4}(t_1 + t_2)\left(1-3x - \sqrt{1 - 6x + x^2}\right). \nonumber
\end{eqnarray}
\end{thm}

\begin{proof}
The pattern \(\hspace{-1mm}\pattern{scale = 0.6}{2}{1/1,2/2}{0/0,0/1,0/2,2/0,2/1,2/2}\) can occur at most once in a permutation, in which case the first element is smaller than the last one. If \(A(x,t)\) denotes the generating function for separable permutations of length at least \(2\) with exactly one occurrence of the pattern, then, by Lemma~\ref{half-first-last}, we have
\begin{equation}\label{Nr9-expr-A}
   A(x,t) = t\,\frac{F(x) - x}{2}.
\end{equation}
Substituting~\eqref{Nr9-expr-A} into~\eqref{Nr113-expr} (which also holds for the pattern \(\hspace{-1mm}\pattern{scale = 0.6}{2}{1/1,2/2}{0/0,0/1,0/2,2/0,2/1,2/2}\)), we obtain
\begin{equation}\label{Nr9-expr-B}
    F(x,t) = \frac{(t+1)(F(x) - x)}{2} + x.
\end{equation}
Furthermore, substituting~\eqref{eq:separable-gf} into~\eqref{Nr9-expr-B} yields~\eqref{eq:Nr.9}. The formula for the joint distribution follows from the easy to see equation $F(x,t_1,t_2)=x+A(x,t_1)+A(x,t_2)$.

Finally, the arguments for the pattern
\(\hspace{-1mm}\pattern{scale = 0.6}{2}{1/1,2/2}{0/0,0/1,0/2,1/0,1/1,1/2}\)
are essentially identical, with the only difference being that an occurrence of this pattern must be formed by the leftmost two elements, which form an ascent.
\end{proof}

Interestingly, although the distribution of the mesh pattern
\(\hspace{-1mm}\pattern{scale = 0.6}{2}{1/1,2/2}{0/0,0/2,1/0,1/1,1/2,2/0,2/2}\)
in the next theorem is the same as that of the mesh pattern
\(\hspace{-1mm}\pattern{scale = 0.6}{2}{1/1,2/2}{0/0,0/1,1/0,1/1,1/2,2/0,2/2}\),
the joint distributions of these patterns within their respective pairs are different.
This is because, in one case, both patterns in the pair can occur in the same permutation,
whereas in the other case this is not possible.

\begin{thm}[Pairs~S11 and S12]\label{equ:Nr.11}
We have  $\hspace{-1mm}\pattern{scale = 0.6}{2}{1/1,2/2}
		{0/0,0/2,1/0,1/1,1/2,2/0,2/2}
\sim_{jd}\hspace{-1mm}\pattern{scale = 0.6}{2}{1/2,2/1}
		{0/0,0/2,1/0,1/1,1/2,2/0,2/2}$. Moreover, 
\begin{eqnarray}
	F(x,t) &:=& \sum_{n \ge 1} x^n \sum_{\pi \in S_n(2413,3142)} 
	t^{\hspace{-1mm}\pattern{scale = 0.5}{2}{1/1,2/2}{0/0,0/2,1/0,1/1,1/2,2/0,2/2}(\pi)}
	= \sum_{n \ge 1} x^n \sum_{\pi \in S_n(2413,3142)} 
	t^{\hspace{-1mm}\pattern{scale = 0.5}{2}{1/2,2/1}{0/0,0/2,1/0,1/1,1/2,2/0,2/2}(\pi)} \nonumber \\
	  &=& \frac{1}{4} \Big(t x^2(x-3 + \sqrt{1 - 6x + x^2})^2\nonumber\\
  &&\quad   -x^2(x-3 + \sqrt{1 - 6x + x^2})^2- 2(x-1 + \sqrt{1 - 6x + x^2}) \Big);\label{eq:Nr.11}
\end{eqnarray}	
\begin{eqnarray}F(x,t_1,t_2) &:=& \sum_{n \ge 1} x^n \sum_{\pi \in S_n(2413,3142)} 
t_1^{\hspace{-1mm}\pattern{scale = 0.5}{2}{1/1,2/2}{0/0,0/2,1/0,1/1,1/2,2/0,2/2}(\pi)}
t_2^{\hspace{-1mm}\pattern{scale = 0.5}{2}{1/2,2/1}{0/0,0/2,1/0,1/1,1/2,2/0,2/2}(\pi)} \nonumber \\
&=& (t_1+t_2-2)x^2(F(x)+1)^2+F(x), \label{eq:Nr117-jd1}
\end{eqnarray}
where $F(x)$ is given by (\ref{eq:separable-gf}).
\end{thm}

\begin{proof}
An occurrence of the pattern \hspace{-1mm}$\pattern{scale = 0.6}{2}{1/1,2/2}{0/0,0/2,1/0,1/1,1/2,2/0,2/2}$ in a permutation $\pi=\pi_1\cdots\pi_n$ can only be formed by the smallest and largest elements. In this case, $\pi = A1nB$, where every element in $A$, if any, is smaller than every element in $B$, if any, in order to avoid the pattern 3142.
There cannot be more than one occurrence of the pattern. If $A(x,t)$ is the g.f.\ for the number of separable permutations with one occurrence of the pattern, then clearly $A(x,t)$ satisfies~\eqref{A(x,t):Nr117} (there are $F(x)+1$ ways to generate $A$, and independently the same number of ways to generate $B$). Then,~\eqref{eq:Nr.11} can be obtained using the same arguments as those in the proof of Theorem~\ref{equ:Nr.117}.

To complete the proof, we observe that 
$$   F(x,t_1,t_2) = A(x,t_1) + A(x,t_2) + ( F(x) - 2A(x,1)) = (t_1+t_2-2) x^2 ( F(x) + 1 )^2 + F(x). $$
\end{proof}

\begin{thm}[Pairs~S17--S20]\label{equ:Nr.17}
The formulas in Theorem~\ref{equ:Nr.113} also hold for pairs S17--S20.
\end{thm}

\begin{proof}
It is sufficient to consider the pair S17. The pattern~\hspace{-1mm}$\pattern{scale = 0.6}{2}{1/1,2/2}{0/0,0/1,0/2,1/1,1/2,2/1,2/2,2/0}$ can occur at most once in any permutation $\pi = \pi_1 \cdots \pi_n$. In this case, we must have $\pi_1 = n-1$ and $\pi_n = n$ and $\red(\pi_2\cdots\pi_{n-1})$ can be any separable permutation of length $n-2$. If $A(x,t)$ is the g.f.\ of separable permutations with one occurrence of the pattern, then all formulas and functional relations in the proof of Theorem~\ref{equ:Nr.113} are exactly the same, giving the desired result.
\end{proof}

\begin{thm}[Pairs~S21--S28]\label{equ:Nr.21}
We have
\begin{eqnarray}
F(x,t) &=& \sum_{n \ge 1} x^n \sum_{\pi \in S_n(2413,3142)} 
	t^{\hspace{-1mm}\pattern{scale = 0.5}{2}{1/1,2/2}{0/0,0/1,1/0,1/1,1/2,2/1,2/0}(\pi)}
	= \sum_{n \ge 1} x^n \sum_{\pi \in S_n(2413,3142)} 
	t^{\hspace{-1mm}\pattern{scale = 0.5}{2}{1/2,2/1}{0/0,0/1,1/0,1/1,1/2,2/1,2/0}(\pi)}\nonumber\\
	&=& \sum_{n \ge 1} x^n \sum_{\pi \in S_n(2413,3142)} 
	t^{\hspace{-1mm}\pattern{scale = 0.5}{2}{1/1,2/2}{0/0,0/1,0/2,1/0,2/0,2/1,2/2}(\pi)}
	= \sum_{n \ge 1} x^n \sum_{\pi \in S_n(2413,3142)} 
	t^{\hspace{-1mm}\pattern{scale = 0.5}{2}{1/2,2/1}{0/0,0/1,0/2,1/0,2/0,2/1,2/2}(\pi)}\nonumber\\
	&=& \frac{1}{2} \left( 1 + (t-1)x \right) \left( 1-x- \sqrt{1 - 6x + x^2} \right).\label{eq-:Nr.21}
	\end{eqnarray}
	\begin{eqnarray}
	F(x,t_1,t_2) &:=& \sum_{n \ge 1} x^n \sum_{\pi \in S_n(2413,3142)} 
t_1^{\hspace{-1mm}\pattern{scale = 0.5}{2}{1/1,2/2}{0/0,0/1,1/0,1/1,1/2,2/0,2/1}(\pi)}
t_2^{\hspace{-1mm}\pattern{scale = 0.5}{2}{1/2,2/1}{0/0,0/1,1/0,1/1,1/2,2/0,2/1}(\pi)} = \sum_{n \ge 1} x^n \sum_{\pi \in S_n(2413,3142)} 
t_1^{\hspace{-1mm}\pattern{scale = 0.5}{2}{1/1,2/2}{0/0,0/1,0/2,1/0,2/0,2/1,2/2}(\pi)}
t_2^{\hspace{-1mm}\pattern{scale = 0.5}{2}{1/2,2/1}{0/0,0/1,0/2,1/0,2/0,2/1,2/2}(\pi)} \nonumber\\
&=& \frac{1}{2} \left( 1 + (t_1+t_2-2)x \right) \left( 1-x- \sqrt{1 - 6x + x^2} \right).\nonumber
\end{eqnarray}
\end{thm}

\begin{proof}
If the mesh pattern \(\hspace{-1mm}\pattern{scale = 0.6}{2}{1/1,2/2}{0/0,0/1,1/0,1/1,1/2,2/1,2/0}\) occurs in a permutation \(\pi\), then the occurrence must be unique and formed by the elements \(1\) and \(2\) placed consequently in increasing order. Let \(\pi'\) be the permutation obtained from \(\pi\) by removing the element \(2\) and decreasing each element larger than \(1\) by \(1\). Clearly, \(\pi\) is separable if and only if \(\pi'\) is separable. Hence, if the g.f.\ for the number of separable permutations with one occurrence of the pattern is \(A(x,t)\), then
\begin{equation}\label{Nr-21-expr-A}
    A(x,t) = t x F(x),
\end{equation}
where \(F(x)\) counts the choices for \(\pi'\) and \(x\) corresponds to the element \(2\). Substituting~\eqref{Nr-21-expr-A} into~\eqref{Nr113-expr} (which also holds for the pattern \(\hspace{-1mm}\pattern{scale = 0.6}{2}{1/1,2/2}{0/0,0/1,1/0,1/1,1/2,2/1,2/0}\)), we obtain
\begin{equation}\label{Nr-17-expr-B}
    F(x,t) = (t-1)x F(x) + F(x).
\end{equation}
Furthermore, substituting~\eqref{eq:separable-gf} into~\eqref{Nr-17-expr-B} yields~\eqref{eq-:Nr.21}. Finally, the formula for the joint distribution follows from 
$$F(x,t_1,t_2)=A(x,t_1)+A(x,t_2)+(F(x)-2A(x,1)).$$
On the other hand, if the mesh pattern 
\(\hspace{-1mm}\pattern{scale = 0.6}{2}{1/1,2/2}{0/0,0/1,0/2,1/0,2/0,2/1,2/2}\)
occurs in a permutation \(\pi=\pi_1\cdots\pi_n\), \(n\ge 2\), then the occurrence must be unique and formed by the elements \(\pi_1=1\) and \(\pi_n\).
Let \(\pi''\) be the permutation obtained from \(\pi\) by removing the element \(1\) and decreasing every other element by \(1\).
Clearly, \(\pi\) is separable if and only if \(\pi''\) is separable.
Hence, the g.f.\ for the number of separable permutations with one occurrence of the pattern is exactly \(A(x,t)\) satisfying~\eqref{Nr-21-expr-A}, and all other formulas derived by us for the pattern
\(\hspace{-1mm}\pattern{scale = 0.6}{2}{1/1,2/2}{0/0,0/1,1/0,1/1,1/2,2/1,2/0}\)
remain exactly the same, completing our proof.
\end{proof}

        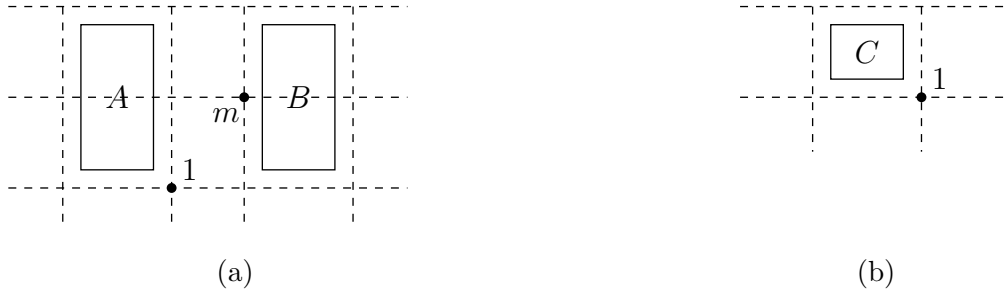
\begin{figure}[htbp]
          \centering
        
          % ===== (a) =====
          \begin{subfigure}{0.47\textwidth}
            \centering
            \begin{tikzpicture}[line width=0.5pt,scale=0.24,baseline=(bb.center)]
              \path[use as bounding box] (0,-12) rectangle (25,1);
              \coordinate (bb) at (13,-4.5); 
              \coordinate (O) at (0,0);
              \path (12,-6)  node {$m$};
              \path (10,-9)  node {$1$};
              \draw [dashed] (0,0)--++(22,0);
              \draw [dashed] (0,-5)--++(22,0);
              \draw [dashed] (0,-10)--++(22,0);
              \fill[black!100] (O)++(9,-10) circle(1.5ex);
              \fill[black!100] (O)++(13,-5) circle(1.5ex);
              \draw (14,-1) rectangle (18,-9);
              \path (16,-5)  node {$B$};
              \draw (4,-1) rectangle (8,-9);
              \path (6,-5)  node {$A$};
              \draw [dashed] (13,0)--++(0,-12);
              \draw [dashed] (9,0)--++(0,-12);
              \draw [dashed] (19,0)--++(0,-12);
              \draw [dashed] (3,0)--++(0,-12);
            \end{tikzpicture}
        
            \vspace{2mm}
            {\centering\subcaption{}\label{fig:perm-Nr49-a}\par}
          \end{subfigure}
          \hfill
          % ===== (b) =====
          \begin{subfigure}{0.47\textwidth}
            \centering
            \begin{tikzpicture}[line width=0.5pt,scale=0.24,baseline=(bb.center)]
              \path[use as bounding box] (0,-12) rectangle (25,1);
              \coordinate (bb) at (15,-4.5);
        
              \coordinate (O) at (0,0);
        		\path (16,-4)  node {$1$};
        		\draw [dashed] (5,0)--++(15,0);
        		\draw [dashed] (5,-5)--++(15,0);
        		\fill[black!100] (O)++(15,-5) circle(1.5ex);
        		\draw (10,-1) rectangle (14,-4);
        		\path (12,-2.5)  node {$C$};
        		\draw [dashed] (15,0)--++(0,-8);
        		\draw [dashed] (9,0)--++(0,-8);
            \end{tikzpicture}
        
            \vspace{2mm}
            {\centering\subcaption{}\label{fig:perm-Nr49-b}\par}
          \end{subfigure}

  \caption{Related to the proof of Theorem~\ref{equ:Nr.49}.}\label{fig:Nr.49}
\end{figure}

\begin{thm}[Pairs~S49--S52]\label{equ:Nr.49}
We have
\begin{eqnarray}
	F(x,t) &:=& \sum_{n \ge 1} x^n \sum_{\pi \in S_n(2413,3142)} 
	t^{\hspace{-1mm}\pattern{scale = 0.5}{2}{1/1,2/2}
    {0/0,1/0,1/1,1/2,2/0}(\pi)}
	= \sum_{n \ge 1} x^n \sum_{\pi \in S_n(2413,3142)} 
	t^{\hspace{-1mm}\pattern{scale = 0.5}{2}{1/2,2/1}
    {0/0,1/0,1/1,1/2,2/0}(\pi)} \nonumber\\
&=&\frac{1}{2} \Big(t \left(1 + x^2 - \sqrt{1 - 6x + x^2}
 - x (4 - \sqrt{1 - 6x + x^2})\right) \nonumber \\
& & +x \left(3-x- \sqrt{1 - 6x + x^2}\right)\Big). \label{eq:Nr.49}
\end{eqnarray}
\begin{eqnarray}F(x,t_1,t_2) &:=& \sum_{n \ge 1} x^n \sum_{\pi \in S_n(2413,3142)} 
t_1^{\hspace{-1mm}\pattern{scale = 0.5}{2}{1/1,2/2}{0/0,1/0,1/1,1/2,2/0}(\pi)}
t_2^{\hspace{-1mm}\pattern{scale = 0.5}{2}{1/2,2/1}{0/0,1/0,1/1,1/2,2/0}(\pi)} \nonumber \\
&=& (1-t_1t_2)x+\frac{1}{2}\Big((t_1+t_2)x+t_1t_2(1-2x)\Big)(1 - x - \sqrt{1 - 6x + x^2}).\nonumber
\end{eqnarray}
\end{thm}

\begin{proof}
By inspecting the pattern $\pattern{scale = 0.6}{2}{1/1,2/2}{0/0,1/0,1/1,1/2,2/0}$, we observe that the smallest element $1$, together with its adjacent element to the right (if it exists), forms a pattern $\pattern{scale = 0.6}{2}{1/1,2/2}{0/0,1/0,1/1,1/2,2/0}$, as illustrated in Figure~\ref{fig:perm-Nr49-a}. Therefore, a permutation avoids this pattern if and only if it ends with the smallest element $1$, as shown in Figure~\ref{fig:perm-Nr49-b}, and we denote its g.f.\ by $C(x)$. Consequently, the g.f. $F(x,t)$ can be expressed as follows:
\begin{equation}\label{Nr49-expr-A}
    F(x,t) = C(x) + t(F(x)-C(x)),
\end{equation}
since the pattern $\pattern{scale = 0.6}{2}{1/1,2/2}{0/0,1/0,1/1,1/2,2/0}$ can occur at most once in a permutation.

Since the block $C$ may be empty, the g.f. $C(x)$ for the separable permutations avoiding this pattern is given by:
\begin{equation}\label{Nr49-expr-B}
   C(x) =  x (F(x)+1) .
\end{equation}

Substituting~\eqref{Nr49-expr-B} into~\eqref{Nr49-expr-A}, we obtain
\begin{equation}\label{Nr49-expr-C}
    F(x,t) = F(x)\big(x + t(1 - x)\big) + x(1 - t).
\end{equation}
Furthermore, substituting~\eqref{eq:separable-gf}~into~\eqref{Nr49-expr-C} yields~\eqref{eq:Nr.49}.

To derive the joint distribution of the patterns, we observe that any permutation is either \(1\), or of the form \(1B\) with \(B\) nonempty, or of the form \(A1\) with \(A\) nonempty, or of the form \(A1B\) with both \(A\) and \(B\) nonempty. The corresponding generating functions for these cases are \(x\), \(t_1 x F(x)\), \(t_2 x F(x)\), and \(t_1 t_2 \bigl(F(x) - 2xF(x) - x\bigr)\), respectively, which leads to the desired result.
\end{proof}

\subsection{Distributions via functional equations}\label{function equ}
In this section, we investigate four classes of pairs of mesh patterns whose occurrences in separable permutations may appear more than once, making their distributions considerably more intricate. To determine their distributions, we establish the corresponding functional equations by analyzing all admissible configurations. The pairs are as follows:
\begin{center}
	\begin{tabular}{rrrrr}
		S13 = $\pattern{scale = 0.6}{2}{1/1,2/2}
		{1/0,1/1,1/2}\pattern{scale = 0.6}{2}{1/2,2/1}{1/0,1/1,1/2}$;&\hspace{-0.3cm}
		S45 = $\pattern{scale = 0.6}{2}{1/1,2/2}{0/0,0/2,1/0,1/1,2/0,2/2}\pattern{scale = 0.6}{2}{1/2,2/1}{0/0,0/2,1/0,1/1,2/0,2/2}$;&\hspace{-0.3cm}
		S61 = $\pattern{scale = 0.6}{2}{1/1,2/2}{0/2,1/2,2/2}\pattern{scale = 0.6}{2}{1/2,2/1}{0/2,1/2,2/2}$;&\hspace{-0.3cm}
		S65 = $\pattern{scale = 0.6}{2}{1/1,2/2}{0/0,0/1,2/0,2/1}\pattern{scale = 0.6}{2}{1/2,2/1}{0/0,0/1,2/0,2/1}$. &\hspace{-0.3cm}
	\end{tabular}
\end{center}

\begin{thm}[Pairs~S13 and S14] \label{equ:Nr.13}
	Let
	$$
	F(x,t) = \sum_{n \ge 1} x^n \sum_{\pi \in S_n(2413,3142)} 
	t^{\hspace{-1mm}\pattern{scale = 0.5}{2}{1/1,2/2}{1/0,1/1,1/2}(\pi)}
	= \sum_{n \ge 1} x^n \sum_{\pi \in S_n(2413,3142)} 
	t^{\hspace{-1mm}\pattern{scale = 0.5}{2}{1/2,2/1}{1/0,1/1,1/2}(\pi)}.
	$$   
	Then, $F(x,t)$ satisfies
	$$tF^3(x,t) + txF^2(x,t) + ((1+t)x - 1)F(x,t) + x = 0.$$
\end{thm}

\begin{proof}
Occurrences of the pattern $\hspace{-1mm}\pattern{scale = 0.6}{2}{1/1,2/2}{1/0,1/1,1/2}$ are precisely the occurrences of ascents, so the result follows from Lemma~\ref{Sxt}, proved in~\cite{fu2018two}.
\end{proof}

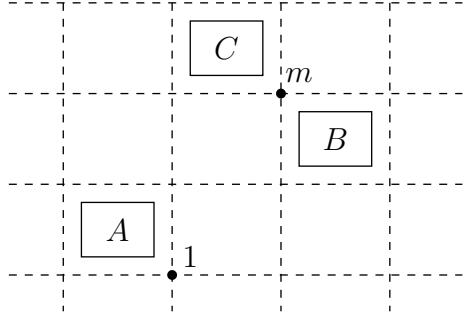
\begin{figure}[htbp]
	\centering
	\begin{tikzpicture}[line width=0.5pt,scale=0.24]
		\coordinate (O) at (0,0);
		\path (16,1)  node {$m$};
		\path (10,-9)  node {$1$};
		\draw [dashed] (0,5)--++(25,0);
		\draw [dashed] (0,0)--++(25,0);
		\draw [dashed] (0,-5)--++(25,0);
		\draw [dashed] (0,-10)--++(25,0);
		\fill[black!100] (O)++(15,0) circle(1.5ex);
		\fill[black!100] (O)++(9,-10) circle(1.5ex);
		\draw (4,-6) rectangle (8,-9);
		\path (6,-7.5)  node {$A$};
		\draw (16,-1) rectangle (20,-4);
		\path (18,-2.5)  node {$B$};
		\draw (10,4) rectangle (14,1);
		\path (12,2.5)  node {$C$};
		\draw [dashed] (15,5)--++(0,-17);
		\draw [dashed] (9,5)--++(0,-17);
		\draw [dashed] (3,5)--++(0,-17);
		\draw [dashed] (21,5)--++(0,-17);
	\end{tikzpicture}
	\caption{Related to the proof of Theorem~\ref{equ:Nr.45}.}\label{fig:Nr.45}
\end{figure}

\begin{thm}[Pairs~S45--48]\label{equ:Nr.45}
We have
\begin{eqnarray}	
F(x,t) &:=& \sum_{n \ge 1} x^n \sum_{\pi \in S_n(2413,3142)} 
	t^{\hspace{-1mm}\pattern{scale = 0.5}{2}{1/1,2/2}
		{0/0,0/2,1/0,1/1,2/0,2/2}(\pi)}
	= \sum_{n \ge 1} x^n \sum_{\pi \in S_n(2413,3142)} 
	t^{\hspace{-1mm}\pattern{scale = 0.5}{2}{1/2,2/1}
		{0/0,0/2,1/0,1/1,2/0,2/2}(\pi)}\nonumber\\
		&=& \frac{
			2\,\left(
			(4(t-1)-1)x^{3}
			+ \sqrt{1 - 6x + x^{2}}
			+ x\,\left(f_1(x,t)\right)
			+ x^{2}\,\left(f_2(x)\right)
			\right)
		}{
			(x - 3 - \sqrt{1 - 6x + x^{2}})\,(3 + (3 - 4t)x + \sqrt{1 - 6x + x^{2}})
		},\label{eq:Nr.45}
	\end{eqnarray}
	where
	\begin{equation}
		\begin{aligned}
			f_1(x,t)&
			= -6 + \sqrt{1 - 6x + x^{2}} + 2t\,(1 - \sqrt{1 - 6x + x^{2}}),\\ \nonumber
			f_2(x,t)&
			= 11 - 4\sqrt{1 - 6x + x^{2}} - 2t\,(5 - 2\sqrt{1 - 6x + x^{2}}).
		\end{aligned}
	\end{equation}
\end{thm}

\begin{proof}
	An occurrence of the pattern $\hspace{-1mm}\pattern{scale = 0.6}{2}{1/1,2/2}{0/0,0/2,1/0,1/1,2/0,2/2}$ must involve the element 1, and another element $m$ to its right.  Moreover, to avoid the pattern 3142, all elements to the left of $1$ must be smaller than those to the right of $m$, as illustrated in Figure~\ref{fig:Nr.45}.  Now, the element $m$ is an occurrence of the pattern \hspace{-1mm}$\pattern{scale = 0.6}{1}{1/1}{0/0,1/1}$ in the permutation $CmB$. Hence, the number of occurrences of the pattern \hspace{-1mm}$\pattern{scale = 0.6}{2}{1/1,2/2}{0/0,0/2,1/0,1/1,2/0,2/2}$ in the permutation $A1CmB$ is the same as that of the pattern \hspace{-1mm}$\pattern{scale = 0.6}{1}{1/1}{0/0,1/1}$ in $CmB$. Let \(S(x,t)\) denote the g.f.\ in
\eqref{eq:separ-length-1}. If $A(x,t)$ is the g.f.\ of separable permutations with at least one occurrence of $\hspace{-1mm}\pattern{scale = 0.6}{2}{1/1,2/2}{0/0,0/2,1/0,1/1,2/0,2/2}$, then 
	\begin{equation}\label{Nr45-expr-A}
		A(x,t) = x (S(x,t)-S(x,0)) ( F(x) + 1 ),
	\end{equation}
	where $x$ corresponds to the element 1, and $F(x)+1$ corresponds to the possibly empty block $A$; note that $S(x,t)-S(x,0)$ is the distribution with at least one occurrence of \hspace{-1mm}$\pattern{scale = 0.6}{1}{1/1}{0/0,1/1}$. 
		
	Adding the contributions from permutations with and without an occurrence, we obtain
	\begin{eqnarray}
		F(x,t) &=& x (S(x,t)-S(x,0)) ( F(x) + 1 ) + \big(F(x) - x (S(x,1)-S(x,0)) ( F(x) + 1 )\big)\nonumber\\
		&=& F(x)+ x(S(x,t)-S(x,1))(F(x)+1)).\label{Nr45-expr-B}
	\end{eqnarray}
	
	Finally, substituting~\eqref{eq:separable-gf} and~\eqref{eq:separ-length-1} into~\eqref{Nr45-expr-B} yields~\eqref{eq:Nr.45}. 
\end{proof}

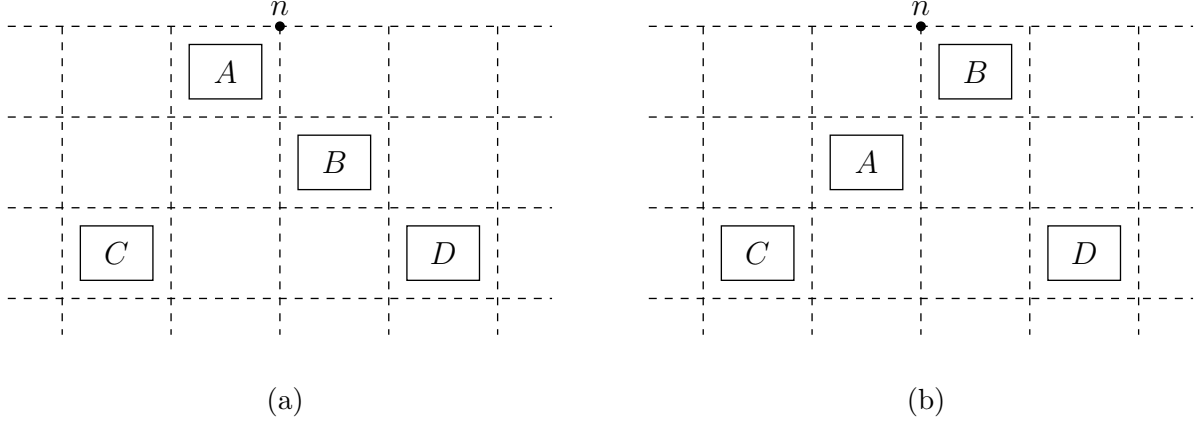
\begin{figure}[htbp]
	\centering
	% ===== (a) =====
	\begin{subfigure}{0.47\textwidth}
		\centering
		\begin{tikzpicture}[line width=0.5pt,scale=0.24,baseline=(bb.center)]
			\path[use as bounding box] (0,-13) rectangle (30,6);
			\coordinate (bb) at (15,-4.5); 
			\coordinate (O) at (0,0);
			\path (15,6)  node {$n$};
			\draw [dashed] (0,5)--++(30,0);
			\draw [dashed] (0,0)--++(30,0);
			\draw [dashed] (0,-5)--++(30,0);
			\draw [dashed] (0,-10)--++(30,0);
			\fill[black!100] (O)++(15,5) circle(1.5ex);
			\draw (22,-6) rectangle (26,-9);
			\path (24,-7.5)  node {$D$};
			\draw (4,-6) rectangle (8,-9);
			\path (6,-7.5)  node {$C$};
			\draw (16,-1) rectangle (20,-4);
			\path (18,-2.5)  node {$B$};
			\draw (10,4) rectangle (14,1);
			\path (12,2.5)  node {$A$};
			\draw [dashed] (15,5)--++(0,-17);
			\draw [dashed] (9,5)--++(0,-17);
			\draw [dashed] (3,5)--++(0,-17);
			\draw [dashed] (21,5)--++(0,-17);
			\draw [dashed] (27,5)--++(0,-17);
		\end{tikzpicture}
		\vspace{2mm}
		{\centering\subcaption{}\label{fig:perm-Nr61-a}\par}
	\end{subfigure}
	\hfill
	% ===== (b) =====
	\begin{subfigure}{0.47\textwidth}
		\centering
		\begin{tikzpicture}[line width=0.5pt,scale=0.24,baseline=(bb.center)]
			\path[use as bounding box] (0,-13) rectangle (30,6);
			\coordinate (bb) at (15,-4.5);
			\coordinate (O) at (0,0);
			\path (15,6)  node {$n$};
			\draw [dashed] (0,5)--++(30,0);
			\draw [dashed] (0,0)--++(30,0);
			\draw [dashed] (0,-5)--++(30,0);
			\draw [dashed] (0,-10)--++(30,0);
			\fill[black!100] (O)++(15,5) circle(1.5ex);
			\draw (22,-6) rectangle (26,-9);
			\path (24,-7.5)  node {$D$};
			\draw (4,-6) rectangle (8,-9);
			\path (6,-7.5)  node {$C$};
			\draw (16,1) rectangle (20,4);
			\path (18,2.5)  node {$B$};
			\draw (10,-4) rectangle (14,-1);
			\path (12,-2.5)  node {$A$};
			\draw [dashed] (15,5)--++(0,-17);
			\draw [dashed] (9,5)--++(0,-17);
			\draw [dashed] (3,5)--++(0,-17);
			\draw [dashed] (21,5)--++(0,-17);
			\draw [dashed] (27,5)--++(0,-17);
		\end{tikzpicture}
		\vspace{2mm}
		{\centering\subcaption{}\label{fig:perm-Nr61-b}\par}
	\end{subfigure}
	\caption{Related to the proof of Theorem~\ref{equ:Nr.61}.}\label{fig:Nr.61}
\end{figure}

\begin{thm}[Pairs~S61--S64]\label{equ:Nr.61}
	We have
	\begin{eqnarray}
	F(x,t) &:=& \sum_{n \ge 1} x^n \sum_{\pi \in S_n(2413,3142)} 
	t^{\hspace{-1mm}\pattern{scale = 0.5}{2}{1/1,2/2}{0/2,1/2,2/2}(\pi)}
	= \sum_{n \ge 1} x^n \sum_{\pi \in S_n(2413,3142)} 
	t^{\hspace{-1mm}\pattern{scale = 0.5}{2}{1/2,2/1}{0/2,1/2,2/2}(\pi)}\nonumber\\
		 &=& \frac{
			x\,\left(x-3 + \sqrt{1 - 6x + x^{2}}\right)
			\,\left(tx-3 + \sqrt{1 - 6tx + t^{2}x^{2}}\right)
		}{
			 tx^{2}-3
			- \sqrt{1 - 6x + x^{2}}
			- \sqrt{1 - 6tx + t^{2}x^{2}}
			+ f_1(x,t)
			- x\,f_2(x,t)
		},\label{eq:Nr.61}
	\end{eqnarray}
	where
	\begin{equation}
		\begin{aligned}
			f_1(x,t)&
			= \sqrt{1 - 6x + x^{2}} \, \sqrt{1 - 6tx + t^{2}x^{2}},\\ \nonumber
			f_2(x,t)&
			= 1 - \sqrt{1 - 6tx + t^{2}x^{2}}
			+ t\,(1 - \sqrt{1 - 6x + x^{2}}).
		\end{aligned}
	\end{equation}
\end{thm}

\begin{proof}
In a permutation $\pi = \pi_1 \cdots \pi_n$, the element $n$, together with each element to its left, forms an occurrence of the pattern $\pattern{scale = 0.6}{2}{1/1,2/2}{0/2,1/2,2/2}$.
	For $n \ge 2$, separable permutations can be divided into two cases 
	according to the position of the second largest element $n-1$, as illustrated in Figure~\ref{fig:Nr.61}. Specifically, the block $A$ in Figure~\ref{fig:perm-Nr61-a} and the block $B$ in Figure~\ref{fig:perm-Nr61-b} must be nonempty (each containing at least the element $n-1$). We denote by $f_1(x,t)$ and $f_2(x,t)$ the generating functions corresponding to the structures illustrated in Figures~\ref{fig:perm-Nr61-a} and~\ref{fig:perm-Nr61-b}, respectively.
	
	In the situation of Figure~\ref{fig:perm-Nr61-a}, there are three cases to consider. 
	\begin{itemize}
		\item[(a)] The blocks $B$, $C$, and $D$ are empty, while $A$ is nonempty. In this case, since any element in the block $A$, together with the element $n$, contributes an occurrence of the pattern $\pattern{scale = 0.6}{2}{1/1,2/2}{0/2,1/2,2/2}$, the block $A$ can be represented by $F(xt)$, while the element $n$ contributes the term $x$. Thus, we obtain the term
		$$
		x F(xt).
		$$
		\item[(b)] The blocks $C$ and $D$ are empty, while $A$ and
		$B$ is nonempty. Similar to case~(a), the block $A$ can be represented by $F(xt)$, while the element~$n$ contributes the term $x$. Clearly, the elements in the block $B$ are independent of this pattern, and hence $B$ is represented by $F(x)$. Hence, in this case, we have the term
		$$
		xF(xt)F(x).
		$$
		\item[(c)] The blocks $A$ and $B$ are nonempty, while among the blocks $C$ and $D$, at least $C$ is nonempty. In this case, after removing $A$ and $B$, the element $n$ together with all the remaining elements in $\pi$, still is a separable permutation in which the second largest element $n-1$ lies to the left of $n$; hence, it has the g.f.\ $f_1(x,t)$. Similar to the previous cases, the block $A$ contributes the term $F(xt)$, while the block $B$ contributes the term $F(x)$. Therefore, this case gives the term
		$$
		F(xt)F(x)f_1(x,t).
		$$
	\end{itemize}
	Combining cases~(a)--(c), we have
	\begin{equation}\label{Nr61-expr-A}
		f_1(x,t) = xF(xt) + xF(xt)F(x) + F(xt)F(x)f_1(x,t).
	\end{equation}
	
	When the element~$n-1$ is to the right of the largest element $n$, 
	as shown in Figure~\ref{fig:perm-Nr61-b}, similar to the analysis of $f_1(x,t)$, we obtain
	\begin{equation}\label{Nr61-expr-B}
		f_2(x,t) = xF(x) + xF(xt)F(x) + F(xt)F(x)f_2(x,t).
	\end{equation}
	
	It is clear that $F(x,t)$ can be expressed as follows:
	\begin{equation}\label{Nr61-expr-C}
		F(x,t) = x + f_1(x,t) + f_2(x,t),
	\end{equation}
	where the permutation 1 contributes the term $x$.
	
	Substituting~\eqref{Nr61-expr-A} and~\eqref{Nr61-expr-B} into~\eqref{Nr61-expr-C}, we obtain
	\begin{equation}\label{Nr61-expr}
		\begin{aligned}
			F(x,t) &= x + xF(xt) + xF(x) + 2xF(xt)F(x) 
			+ F(xt)F(x)\big(f_1(x,t) + f_2(x,t)\big) \\
			&=  x + xF(xt) + xF(x) + 2xF(xt)F(x) 
			+ F(xt)F(x)\big(F(x,t)-x)\big) \\
			&= x + xF(xt) + xF(x) + xF(xt)F(x) 
			+ F(xt)F(x)F(x,t).
		\end{aligned}
	\end{equation}
	Solving this equation for $F(x,t)$ and substituting~\eqref{eq:separable-gf} yields~\eqref{eq:Nr.61}.
\end{proof}

\begin{thm}[Pairs~S65--S68]\label{equ:Nr.65}
We have
	\begin{eqnarray}
	F(x,t) &:=& \sum_{n \ge 1} x^n \sum_{\pi \in S_n(2413,3142)} 
	t^{\hspace{-1mm}\pattern{scale = 0.5}{2}{1/1,2/2}{0/0,0/1,2/0,2/1}(\pi)}
	= \sum_{n \ge 1} x^n \sum_{\pi \in S_n(2413,3142)} 
	t^{\hspace{-1mm}\pattern{scale = 0.5}{2}{1/2,2/1}{0/0,0/1,2/0,2/1}(\pi)} \nonumber\\
			&=& \frac{
				3 - x  - 4tx + \sqrt{1 - 6x + x^2}
				- 4 \, f(x,t)
			}{
				2t\,\left(1 + x - \sqrt{1 - 6x + x^2}\right)
			},\label{eq:Nr.65}
		\end{eqnarray}
	where
	\begin{equation}
		\begin{aligned}
			f(x,t)
			= \sqrt{ -1 - x + \sqrt{1 - 6x + x^2}
				+ \tfrac{1}{16}\!\left(-5 + (4t-1)x + \sqrt{1 - 6x + x^2}\right)^{2}}.\nonumber
		\end{aligned}
	\end{equation}
\end{thm}

\begin{proof}
Let $ab$ be an occurrence of the pattern $\hspace{-1mm}\pattern{scale = 0.6}{2}{1/1,2/2}{0/0,0/1,2/0,2/1}$, so that all elements to the right of $b$ and all elements to the left of $a$ are larger than $b$. Therefore, $b$ is a right-to-left minimum, and $a$ is the leftmost element smaller than $b$ among those to its left. In other words, in a separable permutation, each right-to-left minimum together with the leftmost element smaller than it, if such element exists, contributes one occurrence of the pattern~$\pattern{scale = 0.6}{2}{1/1,2/2}{0/0,0/1,2/0,2/1}$. For every right-to-left minimum $b\neq 1$, there exists a unique smaller element $a$ to its left such that $ab$ forms an occurrence of the pattern~$\pattern{scale = 0.6}{2}{1/1,2/2}{0/0,0/1,2/0,2/1}$. Consequently, in each separable permutation, the number of occurrences of the pattern~$\pattern{scale = 0.6}{2}{1/1,2/2}{0/0,0/1,2/0,2/1}$ corresponds to the number of right-to-left minima minus one. 
	
	According to Theorem~\ref{four-stat}, the distribution of right-to-left minimum on separable permutations is given by the function $G(x,t)$ in~\eqref{eq:four-stat}. Therefore, we have
	\begin{equation}\label{Nr65-expr}
		F(x,t) = \sum_{n \ge 1} x^n \sum_{\pi \in S_n(2413,3142)} t^{\rmin(\pi)-1}
		= \frac{G(x,t)}{t}.
	\end{equation}
	
	By substituting~\eqref{eq:four-stat} into~\eqref{Nr65-expr}, we obtain~\eqref{eq:Nr.65}.
\end{proof}

\subsection{Distributions via reducible and irreducible decomposition}\label{re-irre-decom}
For separable permutations of length $n \ge 2$, each permutation is either reducible or irreducible, as illustrated in Figure~\ref{fig:perm-half-two}, where all four blocks $A_1, A_2, B_1$, and $B_2$ are nonempty. In the reducible case (see Figure~\ref{fig:perm-a}), the block $A_1$ is irreducible (possibly, of length 1), and $A_2$ can be any separable permutation. In the irreducible case (see Figure~\ref{fig:perm-b}), the block $B_1$ is either reducible or has a single (largest) element, while $B_2$ can be any separable permutations. In the proofs of the theorems below, we let $I(x,t)$ and $R(x,t)$ denote, respectively, the generating functions for irreducible and reducible separable permutations, where the variable $t$ records occurrences of the pattern in question.

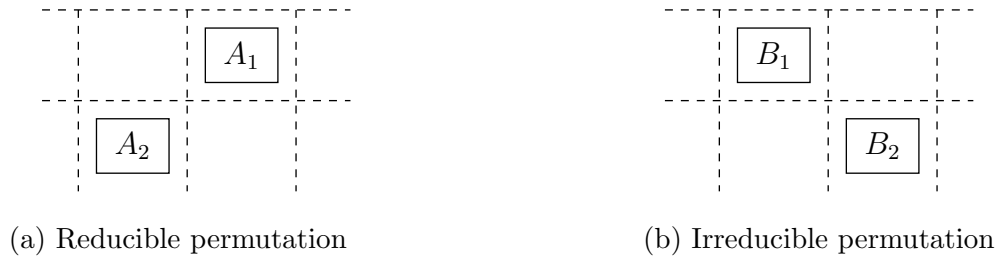
\begin{figure}[htbp]
	\centering
	
	% ===== (a) =====
	\begin{subfigure}{0.47\textwidth}
		\centering
		\begin{tikzpicture}[line width=0.5pt,scale=0.24,baseline=(bb.center)]
			\path[use as bounding box] (0,-10) rectangle (17,0);
			\coordinate (bb) at (15,-4.5); 
			\coordinate (O) at (0,0);
			\draw [dashed] (1,0)--++(17,0);
			\draw [dashed] (1,-5)--++(17,0);
			\draw (10,-1) rectangle (14,-4);
			\path (12,-2.5) node {$A_1$};
			\draw (4,-6) rectangle (8,-9);
			\path (6,-7.5) node {$A_2$};
			\draw [dashed] (15,0)--++(0,-10);
			\draw [dashed] (9,0)--++(0,-10);
			\draw [dashed] (3,0)--++(0,-10);
		\end{tikzpicture}
		
		\vspace{2mm}
		{\centering\subcaption{Reducible permutation}\label{fig:perm-a}\par}
	\end{subfigure}
	\hfill
	% ===== (b) =====
	\begin{subfigure}{0.47\textwidth}
		\centering
		\begin{tikzpicture}[line width=0.5pt,scale=0.24,baseline=(bb.center)]
			\path[use as bounding box] (0,-10) rectangle (17,0);
			\coordinate (bb) at (15,-4.5);     
			\coordinate (O) at (0,0);
			\draw [dashed] (0,0)--++(17,0);
			\draw [dashed] (0,-5)--++(17,0);
			\draw (4,-1) rectangle (8,-4);
			\path (6,-2.5) node {$B_1$};
			\draw (10,-6) rectangle (14,-9);
			\path (12,-7.5) node {$B_2$};
			\draw [dashed] (15,0)--++(0,-10);
			\draw [dashed] (9,0)--++(0,-10);
			\draw [dashed] (3,0)--++(0,-10);
		\end{tikzpicture}
		
		\vspace{2mm}
		{\centering\subcaption{Irreducible permutation}\label{fig:perm-b}\par}
	\end{subfigure}
	\caption{Structural decomposition of separable permutations of length~$n\ge2$.}
	\label{fig:perm-half-two}
\end{figure}

In this section, we investigate the distributions of eight classes of pairs of mesh patterns whose distributions are difficult to derive through a direct configuration analysis as in the previous section. Instead, by fully exploiting the structural decomposition of separable permutations, we obtain functional equations that give their distributions. The pairs are as follows:
\begin{center}
	\begin{tabular}{rrrrr}
S5 = $\pattern{scale = 0.6}{2}{1/1,2/2}
{0/1,1/0,1/1,1/2,2/1}\pattern{scale = 0.6}{2}{1/2,2/1}{0/1,1/0,1/1,1/2,2/1}$; & \hspace{-0.3cm}
S33 = $\pattern{scale = 0.6}{2}{1/1,2/2}{0/0,0/1,1/1,1/2,2/0,2/1}\pattern{scale = 0.6}{2}{1/2,2/1}{0/0,0/1,1/1,1/2,2/0,2/1}$;&\hspace{-0.3cm}
S37 = $\pattern{scale = 0.6}{2}{1/1,2/2}{0/0,0/1,0/2,2/0,2/2}\pattern{scale = 0.6}{2}{1/2,2/1}{0/0,0/1,0/2,2/0,2/2}$;&\hspace{-0.3cm}
S41 = $\pattern{scale = 0.6}{2}{1/1,2/2}{0/0,0/1,0/2,1/0,1/2}\pattern{scale = 0.6}{2}{1/2,2/1}{0/0,0/1,0/2,1/0,1/2}$;\\[5pt]
S53 = $\pattern{scale = 0.6}{2}{1/1,2/2}{1/1,2/0,2/1,2/2}\pattern{scale = 0.6}{2}{1/2,2/1}{1/1,2/0,2/1,2/2}$;&\hspace{-0.3cm}
S57 = $\pattern{scale = 0.6}{2}{1/1,2/2}{0/0,1/0,1/2,2/0}\pattern{scale = 0.6}{2}{1/2,2/1}{0/0,1/0,1/2,2/0}$;&\hspace{-0.3cm}
S69 = $\pattern{scale = 0.6}{2}{1/1,2/2}{0/0,0/1,1/0,1/2,2/0,2/1}\pattern{scale = 0.6}{2}{1/2,2/1}{0/0,0/1,1/0,1/2,2/0,2/1}$;&\hspace{-0.3cm}
S73 = $\pattern{scale = 0.6}{2}{1/1,2/2}{0/0,0/2,1/0,1/1,1/2} \pattern{scale = 0.6}{2}{1/2,2/1}{0/0,0/2,1/0,1/1,1/2}$;\\[5pt]
S77 = $\pattern{scale = 0.6}{2}{1/1,2/2}{0/0,0/2,1/0,1/2,2/1}\pattern{scale = 0.6}{2}{1/2,2/1}{0/0,0/2,1/0,1/2,2/1}$. 
	\end{tabular}
\end{center}

\begin{thm}[Pair~S5]\label{equ:Nr.5}
We have
	\begin{eqnarray}
	F(x,t) &:=& \sum_{n \ge 1} x^n \sum_{\pi \in S_n(2413,3142)} 
	t^{\hspace{-1mm}\pattern{scale = 0.5}{2}{1/1,2/2}{0/1,1/0,1/1,1/2,2/1}(\pi)}
	= \sum_{n \ge 1} x^n \sum_{\pi \in S_n(2413,3142)} 
	t^{\hspace{-1mm}\pattern{scale = 0.5}{2}{1/2,2/1}{0/1,1/0,1/1,1/2,2/1}(\pi)}\nonumber\\
		&=& \frac{t x - 1 + \sqrt{(t x-1)^2-4x\,(1 + x - t x)}}{2\,(-1 - x + t x)}.\label{eq:Nr.5}
	\end{eqnarray}
\end{thm}

\begin{proof}
Each ascent $a(a+1)$ formed by consecutive elements contributes exactly one occurrence of the pattern \hspace{-1mm}$\pattern{scale = 0.6}{2}{1/1,2/2}{0/1,1/0,1/1,1/2,2/1}$. 
	
In the irreducible case (see Figure~\ref{fig:perm-b}), the  pattern $\pattern{scale = 0.6}{2}{1/1,2/2}{0/1,1/0,1/1,1/2,2/1}$ cannot occur across blocks $B_1$ and $B_2$. Moreover, any occurrence of the pattern inside $B_1$ or $B_2$ is an occurrence of the pattern in the entire permutation. Consequently, $I(x,t)$ can be expressed as
\begin{equation}\label{Nr5-expr-A}
		I(x,t) = x+(x+R(x,t))\,F(x,t),
\end{equation}
where $x$ corresponds to the permutation 1, $F(x,t)$ corresponds to the block $B_2$, and the factor $x+R(x,t)$ represents the g.f.\ for the block $B_1$.

In the reducible case (see Figure~\ref{fig:perm-a}), the shadings of the pattern
\hspace{-1mm}$\pattern{scale = 0.6}{2}{1/1,2/2}{0/1,1/0,1/1,1/2,2/1}$
ensure that any occurrence of the pattern within either block $A_1$ or $A_2$
is also an occurrence of the pattern in the entire permutation.
In addition, if $A_1$ has a single element and $A_2$ ends with the largest element,
then an extra occurrence of the pattern arises.
If $g(x,t)$ denotes the distribution of
\hspace{-1mm}$\pattern{scale = 0.6}{2}{1/1,2/2}{0/1,1/0,1/1,1/2,2/1}$
on separable permutations ending with the largest element, then in the reducible case we have
\begin{equation}\label{Nr5-expr-C}
	\begin{aligned}
	R(x,t) &= (I(x,t)-x)F(x,t)+xtg(x,t)+x(F(x,t)-g(x,t)) 
	\end{aligned}
\end{equation}
where $I(x,t)-x$ records irreducible permutations of length at least~2, and
$F(x,t)-g(x,t)$ is the distribution of the pattern on separable permutations not ending with the largest element. Now, we have
$$g(x,t) = x + xtg(x,t)+x(F(x,t)-g(x,t)),$$
where the terms correspond, respectively, to the permutation 1, the separable permutations of length $n\geq 2$ ending with $(n-1)n$, and  the separable permutations of length $n\geq 2$ ending with $an$ for $a\neq n-1$. Therefore,
\begin{equation}\label{Nr5-expr-F}
	g(x,t) = \frac{x\,(F(x,t)+1)}{1 + x(1- t) }.
\end{equation}
Finally, using (\ref{Nr5-expr-A}) and (\ref{Nr5-expr-C}), we have
	\begin{equation}\label{Nr5-expr}
		\begin{aligned}
			F(x,t) &= I(x,t) + R(x,t) \\
			&= F(x,t)\big(I(x,t) + R(x,t)\big) +xF(x,t) + x t g(x,t) - x g(x,t) + x \\
			&= F(x,t)\big(F(x,t) + x\big) + x (t-1) g(x,t) + x.
		\end{aligned}
	\end{equation}
By substituting \eqref{Nr5-expr-F} into \eqref{Nr5-expr} and solving for $F(x,t)$, we obtain~\eqref{eq:Nr.5}. 
\end{proof}

\begin{thm}[Pairs~S33--S36]\label{equ:Nr.33}
	We have
	\begin{eqnarray}
	F(x,t) &:=& \sum_{n \ge 1} x^n \sum_{\pi \in S_n(2413,3142)} 
	t^{\hspace{-1mm}\pattern{scale = 0.5}{2}{1/1,2/2}{0/0,0/1,1/1,1/2,2/0,2/1}(\pi)}
	= \sum_{n \ge 1} x^n \sum_{\pi \in S_n(2413,3142)} 
	t^{\hspace{-1mm}\pattern{scale = 0.5}{2}{1/2,2/1}{0/0,0/1,1/1,1/2,2/0,2/1}(\pi)}\nonumber \\ 
		&=& \frac{2x\,(1+(t-1)\,x)}{1-x-2\,(t-1)\,x^{2}+\sqrt{1-6x+x^{2}}}. \label{eq:Nr.33}
	\end{eqnarray}
\end{thm}

\begin{proof}
In the irreducible case (see Figure~\ref{fig:perm-b}), the shading of the lower-right corner of the pattern
\hspace{-1mm}$\pattern{scale = 0.6}{2}{1/1,2/2}{0/0,0/1,1/1,1/2,2/0,2/1}$
prevents this pattern from occurring in $B_1$.
On the other hand, any occurrence of the pattern in $B_2$ is clearly an occurrence in the entire permutation,
and the pattern cannot occur across $B_1$ and $B_2$.
Therefore, also considering the permutation 1, we have
\begin{equation}\label{Nr33-expr-A}
	I(x,t) = x + F(x,t)\frac{F(x) + x}{2}.
\end{equation}
Here, the factor $F(x,t)$ corresponds to the block $B_2$, while the factor $\frac{F(x) + x}{2}$
represents the generating function for reducible separable permutations corresponding to the block $B_1$,
bearing in mind that $B_1$ may have length~1.
	
	In the reducible case (see Figure~\ref{fig:perm-a}), the shading of the lower-left corner of the pattern
\hspace{-1mm}$\pattern{scale = 0.6}{2}{1/1,2/2}{0/0,0/1,1/1,1/2,2/0,2/1}$
prevents this pattern from occurring in $A_1$. On the other hand, any occurrence of the pattern in $A_2$ is clearly an occurrence in the entire permutation. Additional occurrences of the pattern can only arise in the situation when $A_2$ begins with its largest element $a$, and $A_1$ is of length 1 (that is, $A_1=a+1$ and $a(a+1)$ is an occurrence of \hspace{-1mm}$\pattern{scale = 0.6}{2}{1/1,2/2}{0/0,0/1,1/1,1/2,2/0,2/1}$). Since the g.f.\ for $A_2$ beginning (resp., not beginning) with the largest element is given by $x(F(x,t)+1)$ (resp., $F(x,t)-x(F(x,t)+1)$), and the g.f.\ for irreducible separable permutations of length at least 2 is $\frac{F(x)+x}{2}-x$, we have
	\begin{equation}\label{Nr33-expr-C}
		R(x,t) = F(x,t)\frac{F(x)-x}{2} + tx^2(F(x,t)+1) + x(F(x,t)-x(F(x,t)+1)),
	\end{equation}
where the first term corresponds to the case of the irreducible $A_1$ being of length at least 2. Now using  \eqref{Nr33-expr-A} and \eqref{Nr33-expr-C}, we obtain
	\begin{eqnarray}\label{Nr33-expr}
		F(x,t)
		&=& I(x,t)+R(x,t)\nonumber \\
		&=& \bigl(F(x) + x\bigr) F(x,t) + (t - 1)\, x^2 \,(F(x,t) + 1) + x
	\end{eqnarray}
	
By substituting~\eqref{eq:separable-gf} in (\ref{Nr33-expr}), we obtain~\eqref{eq:Nr.33}.
\end{proof}

\begin{thm}[Pairs~S37--S40]\label{equ:Nr.41}
We have
	\begin{eqnarray}
	F(x,t) &:=& \sum_{n \ge 1} x^n \sum_{\pi \in S_n(2413,3142)} 
	t^{\hspace{-1mm}\pattern{scale = 0.5}{2}{1/1,2/2}{0/0,0/1,0/2,2/0,2/2}(\pi)}
	= \sum_{n \ge 1} x^n \sum_{\pi \in S_n(2413,3142)} 
	t^{\hspace{-1mm}\pattern{scale = 0.5}{2}{1/2,2/1}{0/0,0/1,0/2,2/0,2/2}(\pi)}\nonumber
	\end{eqnarray}
\begin{equation}\label{eq:Nr.41}
=\frac{1}{8}\Bigl(
		5 + x(1 - 4t)
		- \sqrt{\,1 - 6x + x^{2}\,}
		- \sqrt{2}\,
		\sqrt{
			5 + (1 - 4t + 8t^{2})\,x^{2}
			+ f_1(x,t)
		}
		\Bigr)
		\end{equation}
		where
		\begin{equation}
		\begin{aligned}
			f_1(x,t)&
			= 3\sqrt{\,1 - 6x + x^{2}\,}-x\Bigl(
			6 + \sqrt{\,1 - 6x + x^{2}\,}
			+ 4t\bigl(5 - \sqrt{\,1 - 6x + x^{2}\,}\bigr)
			\Bigr).\nonumber
		\end{aligned}
		\end{equation}
\end{thm}

\begin{proof}
 Any occurrence of \hspace{-1mm}$\pattern{scale = 0.6}{2}{1/1,2/2}{0/0,0/1,0/2,2/0,2/2}$ must involve the leftmost element in a given permutation. 
 	
	In the irreducible case (see Figure~\ref{fig:perm-b}), we cannot have any occurrences of the pattern because of the shaded lower-right and upper-left regions (and clearly no occurrences can happen across $B_1$ and $B_2$). Therefore,  also considering the permutation 1, we have
	\begin{equation}\label{Nr41-expr-A}
		I(x,t) = \frac{F(x) + x}{2}.
	\end{equation}
	
In the reducible case (see Figure~\ref{fig:perm-a}), if $ab$ is an occurrence of the pattern
\hspace{-1mm}$\pattern{scale = 0.6}{2}{1/1,2/2}{0/0,0/1,0/2,2/0,2/2}$,
then $a$ is the leftmost element in $A_2$ and $b$ must lie in $A_1$
(because the upper-right region of the pattern is shaded).
Moreover, $b$ must be a right-to-left maximum in $A_1$, and conversely,
any right-to-left maximum in $A_1$, together with the leftmost element in $A_2$,
forms an occurrence of the pattern.
Since all elements of $A_1$ are greater than those of $A_2$,
the right-to-left maxima in $A_1$ are exactly the right-to-left maxima
of the entire permutation.
Consequently, the number of occurrences of the pattern
$\pattern{scale = 0.5}{2}{1/1,2/2}{0/0,0/1,0/2,2/0,2/2}$
is equal to the number of right-to-left maxima in the entire permutation.
Therefore, the distribution of reducible separable permutations (of length $n \ge 2$)
with respect to this pattern coincides with that of the right-to-left maximum statistic. By Theorem~\ref{four-stat}, this distribution is 
	\begin{equation} \label{Nr41-expr-C}
		R(x,t)=\frac{G(x,t)}{1 + G(x,t)} - xt,
	\end{equation}
	where $G(x,t)$ is given by~\eqref{eq:four-stat}. 
	
Finally, to obtain~\eqref{eq:Nr.41}, we substitute \eqref{Nr41-expr-A} and \eqref{Nr41-expr-C} into
$F(x,t) = I(x,t) + R(x,t)$, and then substitute~\eqref{eq:separable-gf} and \eqref{eq:four-stat}.
\end{proof}

\begin{thm}[Pairs~S41--S44]\label{equ:Nr.37}
The formula in Theorem~\ref{equ:Nr.41} holds for the pairs S41--S44.
\end{thm}

\begin{proof}
Any occurrence of \hspace{-1mm}$\pattern{scale = 0.6}{2}{1/1,2/2}{0/0,0/1,0/2,1/0,1/2}$ must involve the leftmost element in a given permutation. 
	
In the irreducible case (see Figure~\ref{fig:perm-b}), it is clear that if $ab$ is an occurrence of \hspace{-1mm}$\pattern{scale = 0.6}{2}{1/1,2/2}{0/0,0/1,0/2,1/0,1/2}$, then both $a$ and $b$ must belong to $B_1$, so that $B_2$ cannot contribute any occurrences of the pattern. Therefore, also considering the permutation 1,
\begin{equation}\label{Nr37-expr-A}
	I(x,t) = x+(R(x,t)+x)\,F(x),
\end{equation}
where the factor $F(x)$ corresponds to the block $B_2$.
		
    In the reducible case (see Figure~\ref{fig:perm-a}), the leftmost element in the permutation is in block $A_2$, and hence the pattern cannot be found within $A_1$. We consider two subcases:
    \begin{itemize}
    \item The block $A_2$ begins with the smallest element. That is, $A_2 = 1A'_2$, as shown in Figure~\ref{fig:Nr.37}. Then the element~$1$, together with every left-to-right maximum in the remainder of the permutation, forms an occurrence of \hspace{-1mm}$\pattern{scale = 0.6}{2}{1/1,2/2}{0/0,0/1,0/2,1/0,1/2}$. The g.f.\ for this subcase is therefore clearly $xG(x,t)$, where $G(x,t)$ is given by Theorem~\ref{four-stat}. Indeed, the distribution in question coincides with the distribution of left-to-right maxima in $A'_2A_1$, which can be any separable permutation.
    \item The block $A_2$ does not start with the smallest element. Then every occurrence of \hspace{-1mm}$\pattern{scale = 0.6}{2}{1/1,2/2}{0/0,0/1,0/2,1/0,1/2}$, if any, must be entirely in $A_2$, and the number of choices for $A_1$ are given by $\frac{F(x) + x}{2}$. Therefore, since $A'_2$ can be empty, in this subcase, the g.f.\ is    	$$\bigl(F(x,t)-x\left(1+G(x,t)\right)\bigl)\frac{F(x)+x}{2}.$$
    \end{itemize}
 Consequently, we obtain $$R(x,t) = xG(x,t)+\bigl(F(x,t)-x\left(1+G(x,t)\right)\bigl)\frac{F(x)+x}{2}$$
		and
	\begin{eqnarray}\label{Nr37-expr-E}
		F(x,t) &=& I(x,t) + R(x,t)  \nonumber \\
		&=& \bigl(1 + F(x)\bigr) \Bigl( x + x G(x,t) + \frac{F(x)+x}{2} \bigl(F(x,t) - x (G(x,t)+1)\bigr) \Bigr).  
	\end{eqnarray}
	Solving \eqref{Nr37-expr-E} for $F(x,t)$ and substituting into the result \eqref{eq:separable-gf} and~\eqref{eq:four-stat} yields~\eqref{eq:Nr.41}.
\end{proof}

\begin{thm}[Pairs~S53--S56]\label{equ:Nr.53}
We have
\begin{eqnarray}
	F(x,t) &:=& \sum_{n \ge 1} x^n \sum_{\pi \in S_n(2413,3142)} 
	t^{\hspace{-1mm}\pattern{scale = 0.5}{2}{1/1,2/2}{1/1,2/0,2/1,2/2}(\pi)}
	= \sum_{n \ge 1} x^n \sum_{\pi \in S_n(2413,3142)} 
	t^{\hspace{-1mm}\pattern{scale = 0.5}{2}{1/2,2/1}{1/1,2/0,2/1,2/2}(\pi)}\nonumber\\
			&=& \frac{
				3 - x(1+4t) + \sqrt{1 - 6x + x^2}
				- 4 \, f(x,t)
			}{
				2t\,\left(1 + x - \sqrt{1 - 6x + x^2}\right)
			},\label{eq:Nr.53}
	\end{eqnarray}
	where
	\begin{equation}
		\begin{aligned}
			f(x,t)
			= \sqrt{ -1 - x + \sqrt{1 - 6x + x^2}
				+ \tfrac{1}{16}\!\left(-5 + (4t-1)x + \sqrt{1 - 6x + x^2}\right)^{2}}.\nonumber
		\end{aligned}
	\end{equation}
\end{thm}

\begin{proof}
Any occurrence of \hspace{-1mm}$\pattern{scale = 0.6}{2}{1/1,2/2}{1/1,2/0,2/1,2/2}$ must involve the leftmost element in a given permutation. 
	
In the irreducible case (see Figure~\ref{fig:perm-b}), the rightmost element lies in the block $B_2$. Moreover, since all elements in $B_1$ are larger than those in $B_2$, the first element required to form the pattern \hspace{-1mm}$\pattern{scale = 0.6}{2}{1/1,2/2}{1/1,2/0,2/1,2/2}$ cannot occur in $B_1$. In addition, the block $B_2$ can be any separable permutation associated with this pattern. Therefore, we have
\begin{equation}\label{Nr53-expr-A}
I(x,t) = x + F(x,t)\,\frac{F(x) + x}{2},
\end{equation}
where $x$ corresponds to the permutation $1$, the factor $F(x,t)$ corresponds to the block $B_2$, while $\frac{F(x) + x}{2}$ represents the g.f.\ for reducible separable permutations corresponding to $B_1$, taking into account that $B_1$ can also be of length~1.
	
	\begin{figure}[htbp]
		\centering
		\begin{tikzpicture}[line width=0.5pt,scale=0.24]
			\coordinate (O) at (0,0);
			\path (22,-4)  node {$m$};
			\draw [dashed] (6,0)--++(19,0);
			\draw [dashed] (6,-5)--++(19,0);
			\draw [dashed] (6,-10)--++(19,0);
			\fill[black!100] (O)++(21,-5) circle(1.5ex);
			\draw (10,-6) rectangle (14,-9);
			\path (12,-7.5)  node {$A_2$};
			\draw (16,-1) rectangle (20,-4);
			\path (18,-2.5)  node {$A'_1$};
			\draw [dashed] (15,0)--++(0,-12);
			\draw [dashed] (9,0)--++(0,-12);
			\draw [dashed] (21,0)--++(0,-12);
		\end{tikzpicture}
		\caption{Related to the proof of Theorem~\ref{equ:Nr.53}.}\label{fig:Nr.53}
	\end{figure}
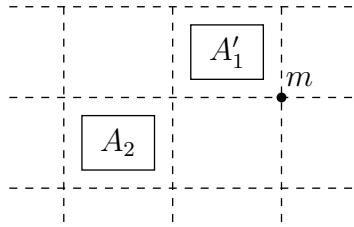
	
	In the reducible case (see Figure~\ref{fig:perm-a}), the rightmost element lies in the block $A_1$. The first element contributing to an occurrence of \hspace{-1mm}$\pattern{scale = 0.6}{2}{1/1,2/2}{1/1,2/0,2/1,2/2}$ appears in $A_2$ if and only if the block~$A_1$ ends with its smallest element $m$, as illustrated in Figure~\ref{fig:Nr.53}. Here, $A_1=A'_1m$. In this configuration, $m$ cannot form the pattern \hspace{-1mm}$\pattern{scale = 0.6}{2}{1/1,2/2}{1/1,2/0,2/1,2/2}$ 
	with any element in $A'_1$, but it can form an occurrence of this pattern together with any right-to-left maximum in $A_2$. Conversely, in this configuration, whenever an occurrence of this pattern appears, its first element must be a right-to-left maximum in $A_2$. Moreover, $A'_1$ can be any (possibly empty) separable permutation because $A_1$ ends with $m$, which guarantees that $A_1$ is irreducible.  Let $A(x,t)$ denote the g.f. associated with this configuration. Then, 
	\begin{equation}\label{Nr53-expr-B}
		\begin{aligned}
			A(x,t)
			= x\,G(x,t)
			\left(F(x)+1\right),
		\end{aligned}
	\end{equation}
	where the block $A_2$, together with $m$, contributes the term $xG(x,t)$ 
	($G(x,t)$ is given by \eqref{eq:four-stat}), while $F(x)+1$ corresponds to the block $A'_1$.
	
Consequently, we obtain
	\begin{equation}\label{Nr53-expr-C}
		\begin{aligned}
			R(x,t) = A(x,t) + (\left(I(x,t)+x\right)\,F(x)  -A(x,1)),
		\end{aligned}
	\end{equation}
	where the factor $F(x)$ corresponds to the block $A_2$, and the factor $I(x,t)+x$ represents the g.f.\ for irreducible permutations of length $n\ge1$ corresponding to the block $A_1$.
	
	Substituting~\eqref{Nr53-expr-A} and~\eqref{Nr53-expr-C} into~$F(x,t) = I(x,t) + R(x,t)$ yields 
	\begin{equation}\label{Nr53-expr}
		\begin{aligned}
	F(x,t) =
	 x + x F(x) +(F(x)+1)
	F(x,t)\,\frac{F(x) + x}{2} +  x\,	\left(F(x)+1\right) \left(G(x,t)-F(x)\right).
		\end{aligned}
	\end{equation}
	
Solving~\eqref{Nr53-expr} for $F(x,t)$ and substituting~\eqref{eq:separable-gf} and~\eqref{eq:four-stat} yields~\eqref{eq:Nr.53}.
\end{proof}

\begin{thm}[Pairs~S57--S60]\label{equ:Nr.57}
The formula in Theorem~\ref{equ:Nr.53} holds for the pairs S57--S60.
\end{thm}

\begin{proof}
Any occurrence of \hspace{-1mm}$\pattern{scale = 0.6}{2}{1/1,2/2}{0/0,1/0,1/2,2/0}$ in a permutation must begin with the smallest element. 
	
	In the irreducible case (see Figure~\ref{fig:perm-b}), the smallest element lies in $B_2$, which implies that the pattern \hspace{-1mm}$\pattern{scale = 0.6}{2}{1/1,2/2}{0/0,1/0,1/2,2/0}$ cannot occur within $B_1$, while $B_2$ can be any separable permutation.  Therefore, we have
	\begin{equation}\label{Nr57-expr-A}
		I(x,t) =  x + F(x,t)\,\frac{F(x) + x}{2},
	\end{equation}
	where the factor $F(x,t)$ corresponds to the block $B_2$, and $\frac{F(x) + x}{2}$ represents the g.f.\ for reducible separable permutations associated with the block $B_1$, that can be of length 1.
	
In the reducible case (see Figure~\ref{fig:perm-a}), the smallest element lies in $A_2$, and $A_2$ can be any separable permutation.
Moreover, the smallest element in $A_2$, together with any left-to-right maximum in $A_1$, contributes an additional occurrence of this pattern.
Conversely, whenever an occurrence of
\hspace{-1mm}$\pattern{scale = 0.6}{2}{1/1,2/2}{0/0,1/0,1/2,2/0}$
appears with its second element in $A_1$, this element must be a left-to-right maximum in $A_1$. Consequently, we obtain
$$R(x,t) = F(x,t)\,\frac{G(x,t)}{1 + G(x,t)}$$
where $F(x,t)$ represents the contribution of the block $A_2$,  while the term $\frac{G(x,t)}{1 + G(x,t)}$ is the distribution of left-to-right maxima on irreducible separable permutations(including the permutation of length $n=1$) given by Theorem~\ref{four-stat}.  Now,
\begin{equation}\label{Nr57-expr}
			F(x,t)
			= I(x,t)+F(x,t)=x + F(x,t)\,\left(\frac{G(x,t)}{1 + G(x,t)}  + \frac{F(x) + x}{2}\right).
\end{equation}
Solving~\eqref{Nr57-expr} for $F(x,t)$ and substituting  \eqref{eq:separable-gf} and~\eqref{eq:four-stat} yields~\eqref{eq:Nr.53}.
\end{proof}

\begin{thm}[Pairs~S69--S72]\label{equ:S69}
	We have
	\begin{eqnarray}
		F(x,t) &:=& \sum_{n \ge 1} x^n \sum_{\pi \in S_n(2413,3142)} 
		t^{\hspace{-1mm}\pattern{scale = 0.5}{2}{1/1,2/2}{0/0,0/1,1/0,1/2,2/0,2/1}(\pi)}
		= \sum_{n \ge 1} x^n \sum_{\pi \in S_n(2413,3142)} 
		t^{\hspace{-1mm}\pattern{scale = 0.5}{2}{1/2,2/1}{0/0,0/1,1/0,1/2,2/0,2/1}(\pi)}\nonumber\\
		&=& \frac{
			2x\,\bigl(3 - x + \sqrt{1 - 6x + x^2}\bigr)
		}{
			\bigl(1 - x + \sqrt{1 - 6x + x^2}\bigr)
			\bigl(3 + (3 - 4t)x + \sqrt{1 - 6x + x^2}\bigr)
		}.\label{eq:Nr.S69}
	\end{eqnarray}
\end{thm}

\begin{proof} 
	If $ab$ is an occurrence of  \hspace{-1mm}$\pattern{scale = 0.6}{2}{1/1,2/2}{0/0,0/1,1/0,1/2,2/0,2/1}$ then $a$ must be the element 1 in the permutation. 
	
	In the irreducible case (see Figure~\ref{fig:perm-b}), the bottom-right shading implies that $B_1$, a reducible separable permutation, contributes no occurrences of
	\hspace{-1mm}$\pattern{scale = 0.6}{2}{1/1,2/2}{0/0,0/1,1/0,1/2,2/0,2/1}$.
	Since the top-left region is unshaded, every occurrence of
	\hspace{-1mm}$\pattern{scale = 0.6}{2}{1/1,2/2}{0/0,0/1,1/0,1/2,2/0,2/1}$ in $B_2$ is also an occurrence in the entire permutation. As no such pattern can occur across $B_1$ and $B_2$, we obtain
	\begin{equation}\label{equ:S69-irr}
		I(x,t)=x+F(x,t)\frac{F(x)+x}{2},
	\end{equation}
	where $F(x,t)$ corresponds to $B_2$, the case $B_1=n$ is considered, and $F(x)$ is given by~\eqref{eq:separable-gf}.
	
	\begin{figure}[htbp]
		\centering
		\begin{tikzpicture}[line width=0.5pt,scale=0.24]
			\coordinate (O) at (0,0);
			\path (8,-9)  node {$1$};
			\draw [dashed] (6,0)--++(19,0);
			\draw [dashed] (6,-5)--++(19,0);
			\draw [dashed] (6,-10)--++(19,0);
			\fill[black!100] (O)++(9,-10) circle(1.5ex);
			\draw (10,-6) rectangle (14,-9);
			\path (12,-7.5)  node {$A'_2$};
			\draw (16,-1) rectangle (20,-4);
			\path (18,-2.5)  node {$A_1$};
			\draw [dashed] (15,0)--++(0,-12);
			\draw [dashed] (9,0)--++(0,-12);
			\draw [dashed] (21,0)--++(0,-12);
		\end{tikzpicture}
		\caption{Related to the proof of Theorem~\ref{equ:Nr.37} and~\ref{equ:S69}.}\label{fig:Nr.37} 
	\end{figure}
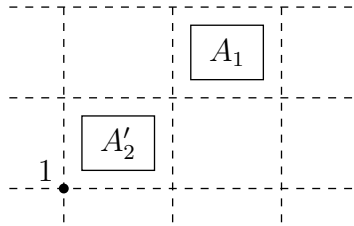
	
	In the reducible case (see Figure~\ref{fig:perm-a}), 
	since the upper-right corner is unshaded and the bottom-left corner is shaded, 
	$A_1$ cannot contribute any occurrences of 
	\hspace{-1mm}$\pattern{scale=0.6}{2}{1/1,2/2}{0/0,0/1,1/0,1/2,2/0,2/1}$.  
	On the other hand, every occurrence of this pattern in $A_2$ is also an occurrence in the entire permutation.  
	We then consider two subcases:
	\begin{itemize}
		\item The block $A_2$ begins with the smallest element. That is,
		$A_2 = 1A'_2$,
		as illustrated in Figure~\ref{fig:Nr.37}, where $A'_2$ may be empty. Then the element~$1$, together with any strict fixed point (i.e., an occurrence of \hspace{-1mm}$\pattern{scale = 0.6}{1}{1/1}{0/1,1/0}$) in $A_1$ and $A'_2$, forms an occurrence of
		\hspace{-1mm}$\pattern{scale=0.6}{2}{1/1,2/2}{0/0,0/1,1/0,1/2,2/0,2/1}$.  Since $A_1$ is an irreducible permutation,
		such an occurrence of
		\hspace{-1mm}$\pattern{scale = 0.6}{1}{1/1}{0/1,1/0}$
		can appear if and only if $A_1$ consists of a single point. 
		Therefore, the g.f.\ for this subcase is
		$$
		tx^2\bigl(G(x,t)+1\bigr)
		\;+\;
		x\bigl(1+G(x,t)\bigr)\frac{F(x)-x}{2},
		$$
		where $G(x,t)$ is given by Theorem~\ref{four-stat}. Here the term $tx^2\bigl(G(x,t)+1\bigr)$ corresponds to the case where
		$A_1$ consists of a single element (i.e., $A_1= n$), while
		$x\bigl(1+G(x,t)\bigr)\frac{F(x)-x}{2}$ accounts for the case  where $|A_1|\ge 2$.
		\item The block $A_2$ does not start with the smallest element. Then every occurrence of \hspace{-1mm}$\pattern{scale=0.6}{2}{1/1,2/2}{0/0,0/1,1/0,1/2,2/0,2/1}$, if any, must be entirely in $A_2$, and the number of choices for $A_1$ are given by $\frac{F(x) + x}{2}$. Therefore, in this subcase, the g.f.\ is
		$$\bigl(F(x,t)-x\bigl(1+G(x,t)\bigr)\bigl)\frac{F(x)+x}{2}.$$
	\end{itemize}
	Consequently, we obtain
	\begin{equation}\label{equ:S69-red}
		R(x,t) = F(x,t)\frac{F(x)+x}{2}+(t-1)x^2\bigl(1+G(x,t)\bigr) .
	\end{equation}
	
	Substituting (\ref{equ:S69-irr}) and (\ref{equ:S69-red}) into $F(x,t)=I(x,t)+R(x,t)$, we obtain (\ref{eq:Nr.S69}).
\end{proof}

\begin{thm}[Pairs~S73--S76]\label{equ:Nr.73}
We have
\begin{eqnarray}
	F(x,t) &=& \sum_{n \ge 1} x^n \sum_{\pi \in S_n(2413,3142)} 
	t^{\hspace{-1mm}\pattern{scale = 0.5}{2}{1/1,2/2}{0/0,0/2,1/0,1/1,1/2}(\pi)}
	= \sum_{n \ge 1} x^n \sum_{\pi \in S_n(2413,3142)} 
	t^{\hspace{-1mm}\pattern{scale = 0.5}{2}{1/2,2/1}{0/0,0/2,1/0,1/1,1/2}(\pi)}\nonumber\\
&=& \frac{
			2\,\left(
			x + \tfrac{1}{4}(t-1)\,x\,\bigl(1 + x - \sqrt{\,1 - 6x + x^{2}\,}\bigr)
			\right)
		}{
			1 - x + \sqrt{\,1 - 6x + x^{2}\,}
			+ \tfrac{1}{2}(1- t)\,x\,\bigl(1 + x - \sqrt{\,1 - 6x + x^{2}\,}\bigr)
		}.\label{eq:Nr.73}
	\end{eqnarray}
\end{thm}

\begin{proof}
An occurrence of the pattern \hspace{-1mm}$\pattern{scale = 0.6}{2}{1/1,2/2}{0/0,0/2,1/0,1/1,1/2}$ is an ascent satisfying extra conditions.  
	
In the irreducible case (see Figure~\ref{fig:perm-b}), since in
 \hspace{-1mm}$\pattern{scale = 0.6}{2}{1/1,2/2}{0/0,0/2,1/0,1/1,1/2}$ the upper-left region is shaded,
the block $B_2$ cannot contain any occurrences of the pattern. Moreover, there are clearly
no occurrences spanning both $B_1$ and $B_2$. Consequently, the g.f.\ $I(x,t)$ is
\begin{equation}\label{Nr73-expr-A}
	I(x,t) = x + \bigl(R(x,t)+x\bigr)F(x).
\end{equation}
	
	\begin{figure}[htbp]
		\centering
		\begin{tikzpicture}[line width=0.5pt,scale=0.24]
			\coordinate (O) at (0,0);
			\path (13.8,-3)  node {$m$};
			\path (10,-9)  node {$1$};
			\draw [dashed] (0,0)--++(25,0);
			\draw [dashed] (0,-5)--++(25,0);
			\draw [dashed] (0,-10)--++(25,0);
			\fill[black!100] (O)++(15,-3) circle(1.5ex);
			\fill[black!100] (O)++(9,-10) circle(1.5ex);
			\draw (4,-6) rectangle (8,-9);
			\path (6,-7.5)  node {$A'_2$};
			\draw (16,-1) rectangle (20,-4);
			\path (18,-2.5)  node {$A'_1$};
			\draw [dashed] (15,0)--++(0,-12);
			\draw [dashed] (9,0)--++(0,-12);
			\draw [dashed] (3,0)--++(0,-12);
			\draw [dashed] (21,0)--++(0,-12);
		\end{tikzpicture}
		\caption{Related to the proof of Theorem~\ref{equ:Nr.73}.}\label{fig:Nr.73}
	\end{figure}
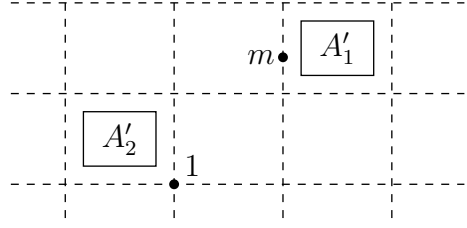
	
In the reducible case (see Figure~\ref{fig:perm-a}), since in
\hspace{-1mm}$\pattern{scale = 0.6}{2}{1/1,2/2}{0/0,0/2,1/0,1/1,1/2}$ the lower-left region is shaded,
the block $A_1$ cannot contain an occurrence of the pattern, and it does not affect the block $A_2$.
However, an extra occurrence can arise across $A_1$ and $A_2$, but only in the situation illustrated
in Figure~\ref{fig:Nr.73}, where $A_2 = A'_2 1$ and $A_1 = m A'_1$. Let $A(x,t)$ denote the g.f.\
associated with this configuration. Then
\begin{equation}\label{Nr73-expr-B}
	A(x,t) = t x \bigl(F(x,t) + 1\bigr)\frac{F(x) + x}{2},
\end{equation}
since $1m$ contributes a factor $t$, the block $A_2$ contributes the term $x\bigl(F(x,t) + 1\bigr)$,
and $A_1$ contributes the term $\frac{F(x) + x}{2}$. Therefore, we have
	\begin{equation}\label{Nr73-expr-C}
		R(x,t) = F(x,t) \,\frac{F(x) + x}{2} + A(x,t) - \frac{A(x,t)}{t},
	\end{equation}
	where the factor $F(x,t)$ corresponds to the block $A_2$, and the term $\frac{F(x) + x}{2}$ to $A_1$. 
	
	Combining \eqref{Nr73-expr-A}--\eqref{Nr73-expr-C} with $F(x,t) = I(x,t) + R(x,t) $, we obtain 
	\begin{equation}\label{Nr73-expr}
		\begin{aligned}
		F(x,t)=
		 x + x F(x) +(F(x)+1)\,\left(\frac{F(x) + x}{2} F(x,t)
		+ (t-1)\,x\,\bigl(F(x,t) + 1\bigr)\,\frac{F(x) + x}{2}\right).
		\end{aligned}
	\end{equation}
By solving (\ref{Nr73-expr}) for $F(x,t)$ and substituting~\eqref{eq:separable-gf}, we obtain~\eqref{eq:Nr.73}.   
\end{proof}

\begin{thm}[Pairs~S77--S80]\label{equ:Nr.77}
We have
\begin{eqnarray}
F(x,t) &=& \sum_{n \ge 1} x^n \sum_{\pi \in S_n(2413,3142)} 
t^{\hspace{-1mm}\pattern{scale = 0.5}{2}{1/1,2/2}{0/0,0/2,1/0,1/2,2/1}(\pi)}
= \sum_{n \ge 1} x^n \sum_{\pi \in S_n(2413,3142)} 
t^{\hspace{-1mm}\pattern{scale = 0.5}{2}{1/2,2/1}{0/0,0/2,1/0,1/2,2/1}(\pi)}\nonumber\\
&=& \frac{x\left(3-x - \sqrt{1 - 6x + x^2}\right)}{1 + (t-1)x^2+ \sqrt{1 - 6x + x^2}+ f(x,t)}\label{eq:Nr.77}
\end{eqnarray}
where
\begin{equation}
\begin{aligned}
f(x,t)
= x\left(4 - \sqrt{1 - 6x + x^2} - t\left(3 - \sqrt{1 - 6x + x^2}\right)\right). \nonumber
\end{aligned}
\end{equation}
\end{thm}

\begin{proof}	
	In the irreducible case (see Figure~\ref{fig:perm-b}), because the upper-left region of \hspace{-1mm}$\pattern{scale = 0.6}{2}{1/1,2/2}{0/0,0/2,1/0,1/2,2/1}$ is shaded, the block $B_2$ cannot contain an occurrence of the pattern, and clearly, the pattern does not occur across $B_1$ and $B_2$. Therefore, $I(x,t)$ is	\begin{equation}\label{Nr77-expr-A}
		I(x,t) = x+(R(x,t)+x) \,F(x).
	\end{equation}
	
In the reducible case (see Figure~\ref{fig:perm-a}), since the lower-left region of
\hspace{-1mm}$\pattern{scale = 0.6}{2}{1/1,2/2}{0/0,0/2,1/0,1/2,2/1}$ is shaded,
this pattern cannot occur in $A_1$, and clearly any occurrence of the pattern in $A_2$
is also an occurrence in the entire permutation. Moreover, if the irreducible block $A_1$
has length greater than~1, then there cannot be any occurrences of the pattern spanning
both $A_1$ and $A_2$. If, on the other hand, $A_1$ has length~1, then the element~$1$ in $A_2$
together with $A_1$ forms an occurrence of
\hspace{-1mm}$\pattern{scale = 0.6}{2}{1/1,2/2}{0/0,0/2,1/0,1/2,2/1}$.
These two subcases give, respectively, the first and second summands in
\begin{equation}\label{Nr77-expr-D}
	R(x,t) = F(x,t)\,\frac{F(x) - x}{2} + t x F(x,t).
\end{equation}
	
	By substituting~\eqref{Nr77-expr-A} and \eqref{Nr77-expr-D} into $F(x,t) = I(x,t) + R(x,t)$, we obtain
	\begin{equation}\label{77-expr}
		F(x,t) = \left(1+F(x)\right)\,\left(F(x,t)\,\frac{F(x)-x}{2} + tx\,F(x,t)\right) + x + xF(x).
	\end{equation}
Solving (\ref{77-expr}) for $F(x,t)$ and substituting~\eqref{eq:separable-gf}, we obtain~\eqref{eq:Nr.77}.
\end{proof}

\section{Conclusion}
In this paper, we first determined the distributions of all mesh patterns of length~1 (see Table~\ref{tab:length-1} and the references therein). In one instance, Theorem~\ref{joint-separ-length-1}, we derived the joint equidistribution of strict fixed points and reverse strict fixed points. Using computer experiments and symmetry operations, we then classified the mesh patterns $12$ and $21$ with identical shadings into 38 equivalence classes based on their distributions; these classes are listed in Tables~\ref{tab:sym-1group}--\ref{tab:g.f.}. Among the 24 (joint) distributions we derived, several were identical, allowing us to merge some equivalence classes and reduce the total number to 31. These mergings are summarized in Table~\ref{tab:fur-mergy}, where patterns with identical distributions are grouped within the same block. Computer experiments confirm that the first several terms of the distributions of the 31 equivalence classes (17 of which we enumerated) are distinct, so no further merging is possible. Enumerating the remaining 14 classes -- those missing references in Tables~\ref{tab:sym-1group}--\ref{tab:g.f.} -- remains an open problem, providing a natural direction for future work.

\begin{table}[ht]
	{
		\renewcommand{\arraystretch}{1.5}    
		\setlength{\tabcolsep}{6pt}          
		\begin{center}
			\begin{tabular}{|c c c c c c c c c c|}
				\hline
				\footnotesize{Nr.} & \footnotesize{Patterns} &
				\footnotesize{Nr.} & \footnotesize{Patterns} &
				\footnotesize{Nr.} & \footnotesize{Patterns} &
				\footnotesize{Nr.} & \footnotesize{Patterns} &
				\footnotesize{Nr.} & \footnotesize{Patterns}  \\
				\hline\hline
				
				\footnotesize{S2} &
				$\pattern{scale = 0.6}{2}{1/1,2/2}
				{0/0,0/1,0/2,1/0,1/2,2/0,2/1,2/2}\pattern{scale = 0.6}{2}{1/2,2/1}{0/0,0/1,0/2,1/0,1/2,2/0,2/1,2/2}$ &
				\footnotesize{S17} & $\pattern{scale = 0.6}{2}{1/1,2/2}{0/0,0/1,0/2,1/1,1/2,2/1,2/2,2/0}\pattern{scale = 0.6}{2}{1/2,2/1}{0/0,0/1,0/2,1/1,1/2,2/1,2/2,2/0}$  &
				\footnotesize{S18} & $\pattern{scale = 0.6}{2}{1/1,2/2}{0/0,0/1,1/0,0/2,1/1,2/1,2/2,2/0}\pattern{scale = 0.6}{2}{1/2,2/1}{0/0,0/1,1/0,0/2,1/1,2/1,2/2,2/0}$  &
				\footnotesize{S19} & $\pattern{scale = 0.6}{2}{1/1,2/2}{0/0,0/1,1/0,0/2,1/1,1/2,2/2,2/0}\pattern{scale = 0.6}{2}{1/2,2/1}{0/0,0/1,1/0,0/2,1/1,1/2,2/2,2/0}$  &
				\footnotesize{S20} & $\pattern{scale = 0.6}{2}{1/1,2/2}{0/0,1/0,0/2,1/1,1/2,2/1,2/2,2/0}\pattern{scale = 0.6}{2}{1/2,2/1}{0/0,1/0,0/2,1/1,1/2,2/1,2/2,2/0}$  \\[5pt]
				\cline{9-10} 
				\footnotesize{S113} & $\pattern{scale = 0.6}{2}{1/1,2/2}{0/0,0/1,1/0,0/2,1/1,1/2,2/1,2/2}\pattern{scale = 0.6}{2}{1/2,2/1}{0/0,0/1,1/0,0/2,1/1,1/2,2/1,2/2}$  &
				\footnotesize{S114} & $\pattern{scale = 0.6}{2}{1/1,2/2}{0/1,1/0,0/2,1/1,1/2,2/1,2/2,2/0}\pattern{scale = 0.6}{2}{1/2,2/1}{0/1,1/0,0/2,1/1,1/2,2/1,2/2,2/0}$  &
				\footnotesize{S115} & $\pattern{scale = 0.6}{2}{1/1,2/2}{0/0,0/1,1/0,1/1,1/2,2/1,2/2,2/0}\pattern{scale = 0.6}{2}{1/2,2/1}{0/0,0/1,1/0,1/1,1/2,2/1,2/2,2/0}$  &
				\footnotesize{S116} &
				\multicolumn{1}{c|}{$\pattern{scale = 0.6}{2}{1/1,2/2}{0/0,0/1,1/0,0/2,1/1,1/2,2/1,2/0}
				\pattern{scale = 0.6}{2}{1/2,2/1}{0/0,0/1,1/0,0/2,1/1,1/2,2/1,2/0}$} &
				\footnotesize{S9} &
				$\pattern{scale = 0.6}{2}{1/1,2/2}
				{0/0,0/1,0/2,2/0,2/1,2/2}\pattern{scale = 0.6}{2}{1/2,2/1}{0/0,0/1,0/2,2/0,2/1,2/2}$ 
				\\[5pt]
				\cline {1-8}
				\footnotesize{S10} &
				$\pattern{scale = 0.6}{2}{1/1,2/2}
				{0/0,0/2,1/0,1/2,2/0,2/2}\pattern{scale = 0.6}{2}{1/2,2/1}{0/0,0/2,1/0,1/2,2/0,2/2}$ &
				\footnotesize{S29} & $\pattern{scale = 0.6}{2}{1/1,2/2}{0/0,0/1,0/2,1/0,1/1,1/2}\pattern{scale = 0.6}{2}{1/2,2/1}{0/0,0/1,0/2,1/0,1/1,1/2}$ &
				\footnotesize{S30}& $\pattern{scale = 0.6}{2}{1/1,2/2}{0/1,0/2,1/1,1/2,2/1,2/2}\pattern{scale = 0.6}{2}{1/2,2/1}{0/1,0/2,1/1,1/2,2/1,2/2}$  &
				\footnotesize{S31} & $\pattern{scale = 0.6}{2}{1/1,2/2}{1/0,1/1,1/2,2/0,2/1,2/2}\pattern{scale = 0.6}{2}{1/2,2/1}{1/0,1/1,1/2,2/0,2/1,2/2}$ &
				\footnotesize{S32} & $\pattern{scale = 0.6}{2}{1/1,2/2}{0/0,0/1,1/0,1/1,2/0,2/1}\pattern{scale = 0.6}{2}{1/2,2/1}{0/0,0/1,1/0,1/1,2/0,2/1}$\\[5pt] \hline
				\footnotesize{S11} &
				$\pattern{scale = 0.6}{2}{1/1,2/2}
				{0/0,0/2,1/0,1/1,1/2,2/0,2/2}
				\pattern{scale = 0.6}{2}{1/2,2/1}
				{0/0,0/2,1/0,1/1,1/2,2/0,2/2}$ &
				\footnotesize{S12} &
				$\pattern{scale = 0.6}{2}{1/1,2/2}
				{0/0,0/1,0/2,1/1,2/0,2/1,2/2}
				\pattern{scale = 0.6}{2}{1/2,2/1}
				{0/0,0/1,0/2,1/1,2/0,2/1,2/2}$ &
				\footnotesize{S117} & $\pattern{scale = 0.6}{2}{1/1,2/2}{0/0,0/1,1/0,1/1,1/2,2/0,2/2}\pattern{scale = 0.6}{2}{1/2,2/1}{0/0,0/1,1/0,1/1,1/2,2/0,2/2}$  &
				\footnotesize{S118} & $\pattern{scale = 0.6}{2}{1/1,2/2}{0/1,0/2,1/0,1/1,1/2,2/0,2/2}\pattern{scale = 0.6}{2}{1/2,2/1}{0/1,0/2,1/0,1/1,1/2,2/0,2/2}$  &
				\footnotesize{S119} & $\pattern{scale = 0.6}{2}{1/1,2/2}{0/0,0/2,1/0,1/1,1/2,2/1,2/2}\pattern{scale = 0.6}{2}{1/2,2/1}{0/0,0/2,1/0,1/1,1/2,2/1,2/2}$   \\[5pt]
				\footnotesize{S120} & $\pattern{scale = 0.6}{2}{1/1,2/2}{0/0,0/2,1/0,1/1,1/2,2/0,2/1}\pattern{scale = 0.6}{2}{1/2,2/1}{0/0,0/2,1/0,1/1,1/2,2/0,2/1}$  &
				\footnotesize{S121} & $\pattern{scale = 0.6}{2}{1/1,2/2}{0/0,0/2,0/1,1/1,1/2,2/0,2/1}\pattern{scale = 0.6}{2}{1/2,2/1}{0/0,0/2,0/1,1/1,1/2,2/0,2/1}$  &
				\footnotesize{S122} & $\pattern{scale = 0.6}{2}{1/1,2/2}{0/0,0/1,1/0,1/1,0/2,2/1,2/2}\pattern{scale = 0.6}{2}{1/2,2/1}{0/0,0/1,1/0,1/1,0/2,2/1,2/2}$  &
				\footnotesize{S123} & $\pattern{scale = 0.6}{2}{1/1,2/2}{0/0,1/2,0/1,1/1,2/0,2/1,2/2}\pattern{scale = 0.6}{2}{1/2,2/1}{0/0,1/2,0/1,1/1,2/0,2/1,2/2}$  &
				\footnotesize{S124} & $\pattern{scale = 0.6}{2}{1/1,2/2}{0/1,0/2,1/0,1/1,2/0,2/1,2/2}\pattern{scale = 0.6}{2}{1/2,2/1}{0/1,0/2,1/0,1/1,2/0,2/1,2/2}$ \\[5pt]
				\hline
				\footnotesize{S21} & $\pattern{scale = 0.6}{2}{1/1,2/2}{0/0,0/1,1/0,1/1,1/2,2/1,2/0}\pattern{scale = 0.6}{2}{1/2,2/1}{0/0,0/1,1/0,1/1,1/2,2/1,2/0}$  &
				\footnotesize{S22}& $\pattern{scale = 0.6}{2}{1/1,2/2}{2/2,0/1,1/0,1/1,1/2,2/1,0/2}\pattern{scale = 0.6}{2}{1/2,2/1}{2/2,0/1,1/0,1/1,1/2,2/1,0/2}$  &
				\footnotesize{S23}&$\pattern{scale = 0.6}{2}{1/1,2/2}{0/0,0/1,1/0,1/1,1/2,2/1,0/2}\pattern{scale = 0.6}{2}{1/2,2/1}{0/0,0/1,1/0,1/1,1/2,2/1,0/2}$  &
				\footnotesize{S24} &$\pattern{scale = 0.6}{2}{1/1,2/2}{2/0,0/1,1/0,1/1,1/2,2/1,2/2}\pattern{scale = 0.6}{2}{1/2,2/1}{2/0,0/1,1/0,1/1,1/2,2/1,2/2}$ &
				\footnotesize{S25} & $\pattern{scale = 0.6}{2}{1/1,2/2}{0/0,0/1,0/2,1/0,2/0,2/1,2/2}\pattern{scale = 0.6}{2}{1/2,2/2}{0/0,0/1,0/2,1/0,2/0,2/1,2/2}$   \\[5pt]
				\cline{7-10}
				\footnotesize{S26} & $\pattern{scale = 0.6}{2}{1/1,2/2}{0/0,0/1,0/2,1/2,2/0,2/1,2/2}\pattern{scale = 0.6}{2}{1/2,2/1}{0/0,0/1,0/2,1/2,2/0,2/1,2/2}$  &
				\footnotesize{S27} & $\pattern{scale = 0.6}{2}{1/1,2/2}{0/0,0/2,1/0,1/2,2/0,2/1,2/2}\pattern{scale = 0.6}{2}{1/2,2/1}{0/0,0/2,1/0,1/2,2/0,2/1,2/2}$  &
				\footnotesize{S28} & \multicolumn{1}{c|}{ $\pattern{scale = 0.6}{2}{1/1,2/2}{0/0,0/1,0/2,1/0,1/2,2/0,2/2}\pattern{scale = 0.6}{2}{1/2,2/1}{0/0,0/1,0/2,1/0,1/2,2/0,2/2}$} &
			    \footnotesize{S37} & $\pattern{scale = 0.6}{2}{1/1,2/2}{0/0,0/1,0/2,2/0,2/2}\pattern{scale = 0.6}{2}{1/2,2/1}{0/0,0/1,0/2,2/0,2/2}$  &
			    \footnotesize{S38} & $\pattern{scale = 0.6}{2}{1/1,2/2}{0/0,0/2,1/2,2/0,2/2}\pattern{scale = 0.6}{2}{1/2,2/1}{0/0,0/2,1/2,2/0,2/2}$  \\[5pt]
				\cline{1-6}
				\footnotesize{S39} & $\pattern{scale = 0.6}{2}{1/1,2/2}{0/0,0/2,2/0,2/1,2/2}\pattern{scale = 0.6}{2}{1/2,2/1}{0/0,0/2,2/0,2/1,2/2}$  &
				\footnotesize{S40} & $\pattern{scale = 0.6}{2}{1/1,2/2}{0/0,0/2,1/0,2/0,2/2}\pattern{scale = 0.6}{2}{1/2,2/1}{0/0,0/2,1/0,2/0,2/2}$ &
				\footnotesize{S41} & $\pattern{scale = 0.6}{2}{1/1,2/2}{0/0,0/1,0/2,1/0,1/2}\pattern{scale = 0.6}{2}{1/2,2/1}{0/0,0/1,0/2,1/0,1/2}$ &
				\footnotesize{S42} & $\pattern{scale = 0.6}{2}{1/1,2/2}{0/1,0/2,1/2,2/1,2/2}\pattern{scale = 0.6}{2}{1/2,2/1}{0/1,0/2,1/2,2/1,2/2}$  &
				\footnotesize{S43} & $\pattern{scale = 0.6}{2}{1/1,2/2}{1/0,1/2,2/0,2/1,2/2}\pattern{scale = 0.6}{2}{1/2,2/1}{1/0,1/2,2/0,2/1,2/2}$  \\[5pt]
				\cline{3-10}
				\footnotesize{S44} & \multicolumn{1}{c|} {$\pattern{scale = 0.6}{2}{1/1,2/2}{0/0,0/1,1/0,2/0,2/1}\pattern{scale = 0.6}{2}{1/2,2/1}{0/0,0/1,1/0,2/0,2/1}$} &
				\footnotesize{S57} & $\pattern{scale = 0.6}{2}{1/1,2/2}{0/0,1/0,1/2,2/0}\pattern{scale = 0.6}{2}{1/2,2/1}
				{0/0,1/0,1/2,2/0}$  &
				\footnotesize{S58} & $\pattern{scale = 0.6}{2}{1/1,2/2}{0/2,1/0,1/2,2/2}\pattern{scale = 0.6}{2}{1/2,2/1}
				{0/2,1/0,1/2,2/2}$ &
				\footnotesize{S59} & $\pattern{scale = 0.6}{2}{1/1,2/2}{0/1,2/0,2/1,2/2}\pattern{scale = 0.6}{2}{1/2,2/1}
				{0/1,2/0,2/1,2/2}$  &
				\footnotesize{S60} & $\pattern{scale = 0.6}{2}{1/1,2/2}{0/0,0/1,0/2,2/1}\pattern{scale = 0.6}{2}{1/2,2/1}
				{0/0,0/1,0/2,2/1}$ \\[5pt]
				\cline{1-2}
				\footnotesize{S65} & $\pattern{scale = 0.6}{2}{1/1,2/2}{0/0,0/1,2/0,2/1}\pattern{scale = 0.6}{2}{1/2,2/1}
				{0/0,0/1,2/0,2/1}$ &
				\footnotesize{S66} & $\pattern{scale = 0.6}{2}{1/1,2/2}{1/0,1/2,2/0,2/2}\pattern{scale = 0.6}{2}{1/2,2/1}
				{1/0,1/2,2/0,2/2}$  &
				\footnotesize{S67} & $\pattern{scale = 0.6}{2}{1/1,2/2}{0/1,0/2,2/1,2/2}\pattern{scale = 0.6}{2}{1/2,2/1}
				{0/1,0/2,2/1,2/2}$ &
				\footnotesize{S68} & $\pattern{scale = 0.6}{2}{1/1,2/2}{0/0,0/2,1/0,1/2}\pattern{scale = 0.6}{2}{1/2,2/1}{0/0,0/2,1/0,1/2}$  & & 
				\\[5pt] 
				\hline
			\end{tabular}
		\end{center}
	}
	\vspace{-0.5cm}
	\caption{Equidistribution classes after further merging.}
	\label{tab:fur-mergy}
\end{table}

\section*{Acknowledgements} The authors thank Philip B. Zhang for sharing his computational experiments on the equidistribution of the mesh patterns studied here. They also thank the anonymous reviewers for their helpful comments, which improved the presentation of the paper. The first author is supported by the National Natural Science Foundation of China (No.~12671393) and the Guangdong Basic and Applied Basic Research Foundation \\ (No.~2024A1515011276).


\begin{thebibliography}{10}


	\bibitem{S-S-J} S. Avgustinovich, S. Kitaev, J. Liese, V. Potapov and A. Taranenko. Singleton mesh patterns in multidimensional permutations, {\em J. Comb. Theory, Ser. A.}, {\bf201} (2024) 105801.

	\bibitem{AKV} S. Avgustinovich, S. Kitaev, and A. Valyuzhenich. Avoidance of boxed mesh patterns on permutations, {\em Discrete Appl. Math.}, {\bf 161} (2013) 43--51.

\bibitem{Baril} 
J.-L. Baril,  Classical sequences revisited with permutations avoiding dotted pattern,  {\em Electron. J. Combin.}, {\bf 18(1)} (2011), \#P178. 

\bibitem{BBL98} P. Bose, J. F. Buss, and A. Lubiw. Pattern matching for permutations. {\em Inform. Process. Lett.} {\bf 65} (1998) 5, 277--283.

\bibitem{BrCl} P. Br\"and\'en and A. Claesson. Mesh patterns and the expansion of permutation statistics as sums of permutation patterns, {\em Electronic J. Combin.}, {\bf 18(2)} (2011), \#P5. 

    \bibitem{CKZ2024} J. N. Chen, S. Kitaev and P. B. Zhang. Distributions of statistics on separable permutations, {\em Discrete Appl. Math.}, {\bf 355} (2024), 169--179.
    
    \bibitem{FFKL} Q. Fang, S. Fu, S. Kitaev and H. Li. On fourteen equidistribution conjectures of Lv and Zhang and monotone mesh patterns with corner shadings, arXiv:2507.04698 (2025).
    
 \bibitem{fu2018two} S.~Fu, Z.~Lin, and J.~Zeng. On two unimodal descent polynomials, {\em Discrete Math.}, 341(9):2616--2626, 2018.

\bibitem{HanZeng} B. Han and J. Zeng. Equidistributions of mesh patterns of length two. {\em S\'{e}m. Lothar. Combin.} {\bf 85B} (2021), Art. 12, 12 pp. 

\bibitem{HanZeng1} B. Han and J. Zeng,  Equidistributions of mesh patterns of length two and Kitaev and Zhang's conjectures. {\em Adv. in Appl. Math.} {\bf 127} (2021), Paper No. 102149, 17 pp.

\bibitem{Hilmarsson2015Wilf}
I.~Hilmarsson, I.~J\'onsd\'ottir, S.~Sigurdard\'ottir, L. Vidarsd\'ottir, and H.~Ulfarsson. Wilf-classification of mesh patterns of short length,
  {\em Electr. J. Combin.}, {\bf 22(4)} (2015),  \#P4.13.
	
\bibitem{Kit} S. Kitaev, \emph{Patterns in permutations and words}, Springer, 2011.

\bibitem{KL} S. Kitaev and J. Liese. Harmonic numbers, Catalan's triangle and mesh patterns, {\em Discrete Math.}, {\bf 313} (2013), 1515--1531.

\bibitem{KZ} S. Kitaev and P. B. Zhang. Distributions of mesh patterns of short lengths, {\em Adv. Appl. Math.}, {\bf 110} (2019) 1--32.

\bibitem{KZZ}  S. Kitaev, P. B. Zhang and X. Zhang. Distributions of several infinite families of mesh patterns, {\em Appl. Math. Comput.}, {\bf 372} (2020) 124984.

\bibitem{LiZhang} D. Li and P. B. Zhang. Distributions of mesh patterns of short lengths on king permutations. {\em Discrete Math.} {\bf 349} (2026), no. 1, Paper No. 114681, 15 pp.

\bibitem{KLv} S. Lv  and S. Kitaev. On (joint) equidistributions of mesh patterns 123 and 132 with symmetric shadings, {\em Adv. Appl. Math.}, {\bf 166} (2025), 102856.

\bibitem{LvZhang} S. Lv and P. B. Zhang. Joint equidistributions of mesh patterns 123 and 132 with minus antipodal shadings.{\em Discrete Appl. Math.} {\bf 379} (2026), 419--433.

\bibitem{LvZhang1} S. Lv and P. B. Zhang. Joint equidistributions of mesh patterns 123 and 321 with symmetric and minus-antipodal shadings, {\em Ann. Comb.} (2026). https://doi.org/10.1007/s00026-025-00793-8.

\bibitem{T2}  B. E. Tenner. Mesh patterns with superfluous mesh, {\em Adv. Appl. Math.}, {\bf 51} (2013), 606--618.
	
\end{thebibliography}
\end{document}